\documentclass[11pt]{amsart}
\usepackage{geometry}  
\usepackage{graphicx}
\usepackage{amssymb,amsmath}
\usepackage{epstopdf}
\usepackage[labelformat=empty]{subfig}
\usepackage[pdftex]{color}
\usepackage{xcolor}
\usepackage{bm}
\usepackage{tikz}
\usetikzlibrary{matrix}
\usepackage{multicol}
\usepackage{array}
\allowdisplaybreaks
\DeclareMathOperator{\sign}{sign}
\DeclareMathOperator{\sech}{sech}

\title{Symplectic classification of coupled angular momenta}
\author{Jaume Alonso \qquad Holger R.\ Dullin \qquad Sonja Hohloch}

\numberwithin{equation}{section}
\newtheorem{de}{Definition}[section]
\newtheorem{theo}[de]{Theorem}
\newtheorem{pro}[de]{Proposition}
\newtheorem{lemm}[de]{Lemma}
\newtheorem{co}[de]{Corollary}
\newtheorem{theol}{Theorem}

\newtheorem{re}[de]{Remark}

\usepackage[parfill]{parskip}

\newcommand{\dee}{\mathrm{d}}
\newcommand{\omc}{\omega_\text{can}}

\newcommand{\vungoc}{V\~u Ng\d{o}c }
\newcommand{\nff}{{n_{\text{FF}}}}
\newcommand{\ra}{{r_A}}
\newcommand{\rb}{{r_B}}
\newcommand{\rg}{{r_C}}

\newcommand{\rep}{\text{Re}}
\newcommand{\imp}{\text{Im}}

\def\ja#1{{\color{red}{#1}}} 

\newcommand{\C}{\mathbb{C}}

\newcommand{\N}{\mathbb{N}}

\newcommand{\R}{\mathbb{R}}
\newcommand{\mbS}{\mathbb{S}}
\newcommand{\T}{\mathbb{T}} 
\newcommand{\Z}{\mathbb{Z}}

\newcommand{\al}{\alpha}
\newcommand{\be}{\beta}
\newcommand{\ga}{\gamma}

\newcommand{\ze}{\zeta}

\newcommand{\vt}{\vartheta}
\newcommand{\ka}{\kappa}
\newcommand{\lam}{\lambda}

\newcommand{\si}{\sigma}

\newcommand{\om}{\omega}

\newcommand{\Ga}{\Gamma}
\newcommand{\De}{\Delta}

\newcommand{\Lam}{\Lambda}

\newcommand{\pti}{{\tilde{p}}}
\newcommand{\qti}{{\tilde{q}}}

\newcommand{\mcA}{\mathcal A}
\newcommand{\mcB}{\mathcal B}

\newcommand{\mcG}{\mathcal G}
\newcommand{\mcH}{\mathcal H}
\newcommand{\mcI}{\mathcal I}
\newcommand{\mcJ}{\mathcal J}

\newcommand{\mcL}{\mathcal L}

\newcommand{\mcN}{\mathcal N}
\newcommand{\mcO}{\mathcal O}

\newcommand{\mcS}{\mathcal S} 
\newcommand{\mcT}{\mathcal T}

\newcommand{\mcV}{\mathcal V}
\newcommand{\mcW}{\mathcal W}
\newcommand{\mcX}{\mathcal X}

\newcommand{\mfI}{\mathfrak I}

\newcommand{\mfh}{\mathfrak h}

\newcommand{\mfp}{\mathfrak p}

\begin{document}

\begin{abstract}
The coupled angular momenta are a family of completely integrable systems that depend on three parameters and have a compact phase space. They correspond to the classical version of the coupling of two quantum angular momenta and they constitute one of the fundamental examples of so-called semitoric systems. Pelayo \& \vungoc have given a classification of semitoric systems in terms of five symplectic invariants. Three of these invariants have already been partially calculated in the literature for a certain parameter range, together with the linear terms of the so-called Taylor series invariant for a fixed choice of parameter values.

In the present paper we complete the classification by calculating the polygon invariant, the height invariant, the twisting-index invariant, and the higher-order terms of the Taylor series invariant for the whole family of systems. We also analyse the explicit dependence of the coefficients of the Taylor series with respect to the three parameters of the system, in particular near the Hopf bifurcation where the focus-focus point becomes degenerate.
\end{abstract}

\maketitle

\section{Introduction}
\label{sec:intro}
Within the modern theory of dynamical systems, completely integrable systems have a predominant position since they often appear in classical mechanics and because their many conserved quantities make them easier to understand and analyse. From a symplectic point of view, Arnold-Liouville's theorem \cite{Ar} and the normal forms by Eliasson \cite{El1,El2} and Miranda \& Zung \cite{MZ} provide a useful local description of the behaviour of completely integrable systems. Unfortunately, a global symplectic classification of completely integrable systems in general seems out of reach at the moment. 

Results are however possible if we add some restrictions. This is the case of toric systems, i.e.\ completely integrable systems for which the Hamiltonian flows of all conserved quantites are periodic. These systems were classified by Delzant \cite{Del} using a certain type of polygons. Furthermore, in  recent years Pelayo \& \vungoc \cite{PV1,PV4} have provided a classification of so-called semitoric systems, a class of four-dimensional integrable systems with a global $\mbS^1$-action admitting only non-degenerate non-hyperbolic singularities, together with some additional restrictions. This classification is achieved by five symplectic invariants (see below) in such a way that two systems are equivalent if and only if their invariants agree. Moreover, given an admissible list of invariants, the corresponding semitoric system can be constructed. The computation of these invariants, however, is far from trivial.

\subsection*{Semitoric systems}

Let $(M,\om)$ denote a 4-dimensional connected symplectic manifold throughout the paper. A semitoric system is a 4-dimensional completely integrable system $(M,\om,F=(L,H))$ with 2 degrees of freedom such that all singularities are non-degenerate, they have no hyperbolic components, the map $L$ is \emph{proper} (i.e.\ the preimage of a compact set by $L$ is again compact) and it induces a faithful Hamiltonian $\mbS^1$-action on $M$. In particular, the flow of $L$ is $2\pi$-periodic.

The theory of semitoric systems has recently become an active field of research, see Pelayo $\&$ \vungoc\!\! \cite{PV2} and Sepe $\&$ \vungoc\!\! \cite{SV} for an overview. The symplectic classification of semitoric systems is achieved in terms of the following five symplectic invariants:
\begin{itemize}
 	\item The \emph{number of focus-focus singularities}, denoted by $\nff$.
 	\item The \emph{Taylor series invariant}, a collection of $\nff$ formal Taylor series in two variables describing the foliation around each focus-focus singular fibre, cf.\ \S \ref{sec:taylor} for more details.
 	\item The \emph{polygon invariant}, an equivalence class of labelled collections of rational convex poligons and vertical lines crossing them, cf.\ \S \ref{sec:polygon}.
	\item The \emph{height invariant}, $\nff$ numbers corresponding to the height of the focus-focus critical values in the rational convex polygons of the polygon invariant, 	cf.\ \S \ref{sec:height}.
	\item The \emph{twisting index invariant}, $\nff$ integers associated to each of the collections of the polygon invariant, measuring the twisting of the system when moving from one singularity to another, cf.\ \S \ref{sec:twistingIndex}.
\end{itemize}

Let $(M_1,\om_1,(L_1,H_1))$ and $(M_2,\om_2,(L_2,H_2))$ be two semitoric systems. We say that they are \emph{isomorphic} as semitoric systems if there exists a symplectomorphism $\psi:M_1 \to M_2$ and a smooth map $g : \R^2 \to \R$ such that $\psi^*(L_2,H_2) = (L_1,g(L_1,H_1))$ and $\frac{\partial g}{\partial H_1}>0$. The classification result for semitoric system states that two semitoric systems are isomorphic if and only if they have the same list of symplectic invariants. Furthermore, given any `admissible' list of invariants, a semitoric system with such invariants can be constructed. 

One of the best understood examples is the coupled spin-oscillator, cf.\ Pelayo $\&$ \vungoc\!\! \cite{PV3}, a special case of the Jaynes-Cunning model (see Babelon $\&$ Cantini $\&$ Dou\c{c}ot \cite{BCD}). The authors of the present paper have recently completed the symplectic classification of this system in \cite{ADH}. Another important example are coupled angular momenta, the classical counterpart of the coupling of two quantum angular momenta as studied by Sadovksii \& Zhilinskii 
\cite{SZ}. Le Floch \& Pelayo have computed some of the symplectic invariants in \cite{LFP} for a particular family parameter, and they also recover the invariants from the joint quantum spectrum. Recently, a compact semitoric system with two focus-focus points has been described by the third author together with Palmer \cite{HP}. More details about semitoric systems and their classification are given in \S \ref{sec:seminf}.

\subsection*{Coupled angular momenta} 
Semitoric systems appear naturally in physics. One of the first problems that motivated the study of semitoric systems was the redistribution of certain energy levels in the coupling of two quantum angular momenta observed by Sadovksii \& Zhilinskii in \cite{SZ}. The classical counterpart of such system 
constitutes one of the simplest examples of a compact semitoric system.

Consider $M=\mbS^2 \times \mbS^2$ and endow it with the symplectic form $\om = -(R_1 \om_{\mbS^2} \oplus R_2 \om_{\mbS^2})$, where $\om_{\mbS^2}$ is the standard symplectic form of $\mbS^2$ and $R_1,R_2$ are positive constants with $R_1 < R_2$. 

Let $(x_i,y_i,z_i)$ be Cartesian coordinates on the unit sphere $x_i^2 + y_i^2 + z_i^2 = 1$, $i=1,2$ and consider a parameter $t \in \R$. The coupled angular momenta is  the 4-dimensional completely integrable system with family parameter $t$ defined by 
\begin{equation}
\begin{cases}
L(x_1,y_1,z_1,x_2,y_2,z_2)\;:= R_1(z_1-1) + R_2 (z_2+1)\\
H(x_1,y_1,z_1,x_2,y_2,z_2):= (1-t) z_1 + t (x_1x_2 + y_1y_2 + z_1z_2) +2t-1.
\label{sysin}
\end{cases}
\end{equation} 

The map $L(x_1,y_1,z_1,x_2,y_2,z_2)$ corresponds to the momentum map of a simultaneous rotation on the two spheres and induces a global $\mbS^1$-action. The system has 4 fixed points, out of which three are always of elliptic-elliptic type and the other is of focus-focus type for a certain interval of values of $t$ and of elliptic-elliptic type otherwise. 
If the parameters in the symplectic form satisfy $R_1> R_2$ we will call the system {\em reverse} coupled angular momenta.

\subsection*{Main results}
The coupled angular momenta are a 4-dimensional semitoric system and can thus be classified using five symplectic invariants. Le Floch \& Pelayo computed in \cite{LFP} the number of focus-focus points invariant, the polygon invariant, the height invariant and the linear terms of the Taylor series invariant for $t = \frac{1}{2}$. The Taylor series invariant describes the singular foliation around focus-focus singularities and is particularly difficult to compute explicitly. The first main result of the present paper is the following:
\begin{theol}
The Taylor series symplectic invariant of the coupled angular momenta is given by
\begin{align*}
S(l,j) =& \, l \arctan \left( \dfrac{{R_2}^2 (-1 + 2 t) - R_1 R_2 (1 + t) + {R_1}^2 t}{(R_1 - R_2)R_1 \ra} \right)+j \ln \left( \frac{4 {R_1}^{5/2}\ra^3}{{R_2}^{3/2}(1-t)t^2} \right) \\ \nonumber
+& \dfrac{l^2}{16 {R_1}^4R_2 \ra^3} \left( - {R_2}^4 (-1 + 2 t)^3 +R_1 {R_2}^3 (1 - 17 t + 46 t^2 - 32 t^3)\right. \\ \nonumber &\hspace{2.2cm} \left. -3 {R_1}^2 {R_2}^2 t (1 - 7 t + 4 t^2)   
 +{R_1}^3 R_2 (3 - 5 t) t^2- {R_1}^4 t^3   \right) \\ \nonumber
+& \dfrac{lj}{8 {R_1}^3 R_2\ra^2} \left( (R_2 - R_1)({R_2}^2 (1 - 2 t)^2   + 2 R_1 R_2 t (-1 + 6 t)+{R_1}^2 t^2) \right)\\ \nonumber
+& \dfrac{j^2}{16 {R_1}^4 R_2 \ra^3} \left( 
  {R_2}^4 (-1 + 2 t)^3 
- R_1 {R_2}^3 (1 + 15 t - 42 t^2 + 16 t^3)
\right. \\ \nonumber &\hspace{2.2cm} \left.
+  {R_1}^2 {R_2}^2 t (3 + 3 t - 28 t^2) 
+ {R_1}^3 R_2 t^2 (-3 + 13 t)
+ {R_1}^4 t^3
 \right)
+ \mcO(3),
\end{align*}
where 
\begin{equation*}
 \ra: = \sqrt{-{R}^2 (1 - 2 t)^2 + 2 R\, t - t^2},\qquad R:=\frac{R_2}{R_1},
\end{equation*} 
 $l$ is the value of the first integral $L$ generating the global $\mbS^1$-action and $j$ is the value of Eliasson's $Q_2$ function from \eqref{qs}.
\end{theol}

Theorem A is reformulated and proven in Theorem \ref{theoinv}. Using a discrete symmetry, the result is extended to the reverse coupled angular momenta in Corollary \ref{Taylinv}. To our knowledge this is the first time that higher order terms of the Taylor series invariant have been computed for a compact semitoric system. The authors calculated the Taylor series invariant for the coupled spin-oscillator in \cite{ADH}, which is a non-compact semitoric system, and the second author also calculated the Taylor series invariant of the spherical pendulum in \cite{Du}, which is in fact not a semitoric system in the usual sense of the word, since the angular momentum integral is not proper (more details on the different types of properness in semitoric systems can be found in Pelayo $\&$ Ratiu $\&$ \vungoc\!\! \cite{PRV}).

Since the system \eqref{sysin} depends on three parameters, the computation of the Taylor series invariant becomes quite involved at some points. To deal with this situation, a different approach is used than in \cite{ADH}. More precisely, the invariant is not obtained directly from the computation of the expansion of the action of the system, but from the period of the reduced system and the rotation number instead. These two quantities are related to the partial derivatives of the Taylor series and completely determine the invariant, since by definition it has no constant term.
All computations in this paper have been done with the assistance of \emph{Mathematica}. 

Theorem A reveals the explicit dependence of the coefficients of the Taylor series on the parameters of the system. We show some symmetries of these coefficients in Corollary \ref{Taylinv} and apply our result to the Kepler problem in prolate spheroidal coordinates in Proposition \ref{propkep}. An interesting observation is that the coefficients diverge when we move the parameter $t$ towards the limits of the interval for which the singularity is of focus-focus type (Remark \ref{reminf}). This opens the future question about how to generalise the metric on the moduli spaces of semitoric systems introduced by Palmer in \cite{Pa} to make it compatible with the transition of the singularity from focus-focus type to elliptic-elliptic type. 

The second main result of this paper is the computation of the polygon invariant. Le Floch \& Pelayo calculated in \cite{LFP} this invariant for the case $R_1<R_2$. We extend their result to the case $R_1>R_2$ and $R_1=R_2$. We also show that in the latter case, the reflection of each element of the invariant is also contained in the invariant.

\begin{theol}
The polygon invariant of the coupled angular momenta for the cases $R_1 < R_2$ and $R_1 = R_2$ are given by the $(\Z_2 \times \mcG)$-orbits represented in Figures \ref{figpolrev} and \ref{figpolkep} respectively.
\end{theol}

Theorem B is restated and proven in Theorem \ref{theopol}. These polygons correspond also to the second and third column of Figure \ref{figtwistab}. The third main result of the present paper is the computation of the height invariant. Le Floch \& Pelayo \cite{LFP} calculated this invariant for the specific case $t=\frac{1}{2}$. We extend their result to any value of $t$ for which there is a singularity of focus-focus type.

\begin{theol}
The height invariant of the coupled angular momenta is
\begin{align*}
\mfh=2\min\{R_1,R_2\} + \dfrac{R_1}{\pi t} &\left( 
\ra - 
2R\,t \arctan \left( \dfrac{\ra}{R-t} \right)-
2\,t \arctan \left( \dfrac{\ra}{R+t-2R\,t} \right)
\right).
\end{align*}
\end{theol}

Theorem C is restated and proven in Theorem \ref{theoheight}. The last main result of the present paper is the computation of  the twisting-index invariant. To our knowledge, this invariant has never been calculated for a compact semitoric system before. For the non-compact case, the authors computed it for the coupled spin-oscillator in \cite{ADH}. For the coupled angular momenta we have the following result:

\begin{theol}
The twisting-index invariant of the coupled angular momenta is given by the association of a twisting index $k$ to each of the labelled polygons in the polygon invariant as represented in Figure \ref{figtwistab}. In particular, the polygons in the last two rows of Figure \ref{figtwistab} have $k=0$.
\end{theol}
Theorem D is restated and proven in Theorem \ref{theotwis}. This result completes the symplectic classification of the whole family of coupled angular momenta.

\subsection*{Structure of the paper}
In Section 2 we briefly summarise the definition of each of the five symplectic invariants of semitoric systems. In Section 3 we introduce the coupled angular momenta system and calculate its Taylor series invariant. Section 4 is devoted to the polygon invariant of the system. In Section 5 we calculate the height invariant of the system and in Section 6 we compute the twisting-index invariant.

In Appendix A we review some properties of elliptic integrals for the reader's convenience. Appendix B consists of the lists of coefficients that appear in the calculations of the Taylor series invariant. Since they depend on several parameters, they would otherwise make the paper cumbersome to read.

\subsection*{Figures} 
Figures \ref{figtwistab}, \ref{figmom}, \ref{physical}, \ref{hami}, \ref{hami2}, \ref{linterms}, \ref{quadterms}, \ref{figpolstan}, \ref{figpolrev}, \ref{figpolkep}, \ref{figheight}, \ref{twisfig}, \ref{twisfigmat} and \ref{twistbig}, have been made with \emph{Mathematica} and Figures \ref{cycles2} and \ref{cycles} have been made with \emph{Inkscape}.

\subsection*{Acknowledgements} We wish to thank Yohann Le Floch, Joseph Palmer and San \vungoc for inspiring conversations and Wim Vanroose for sharing his computational resources with us. Moreover, the first author has been fully and the third authour has been partially funded by the Research Fund of the University of Antwerp. 

\begin{figure}[ht!]
\centering
\begin{tabular}{m{1.5cm} m{3.6cm} m{3cm} m{3.6cm}}
\begin{center}$\mathbf{(k,\epsilon)}$\end{center} &
\begin{center}$\mathbf{R_1 < R_2}$ \end{center} &
\begin{center}$\mathbf{R_1 = R_2}$ \end{center} &
\begin{center}$\mathbf{R_1 > R_2}$ \end{center} \\ 

\begin{minipage}{1.5cm} $k=-2$\\$\varepsilon=-1$ \end{minipage} &
\begin{center}\includegraphics[width=3.4cm]{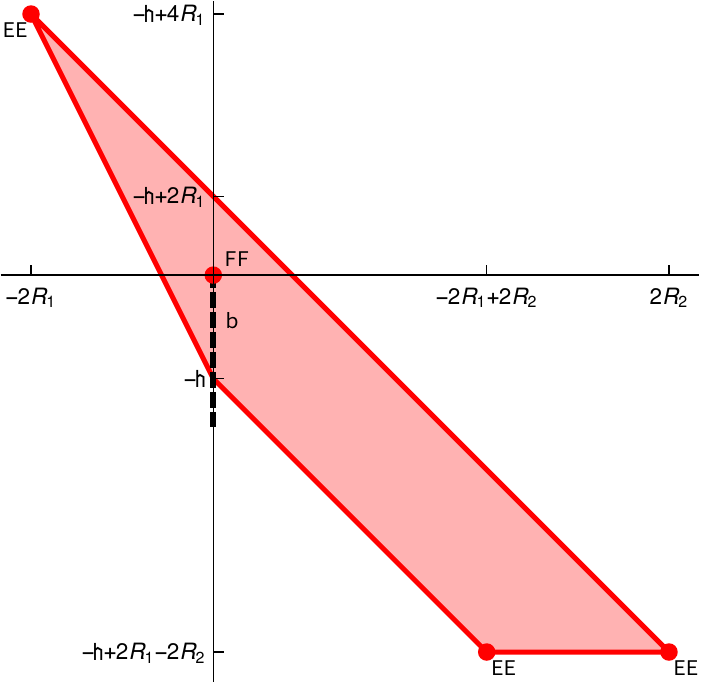}\end{center} &
\begin{center}\includegraphics[width=3.0cm]{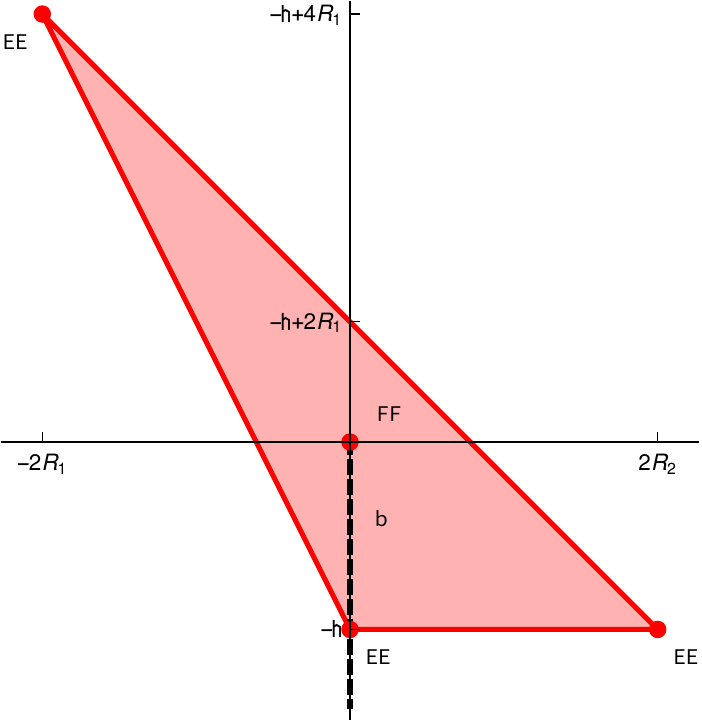}\end{center} &
\begin{center}\includegraphics[width=3.4cm]{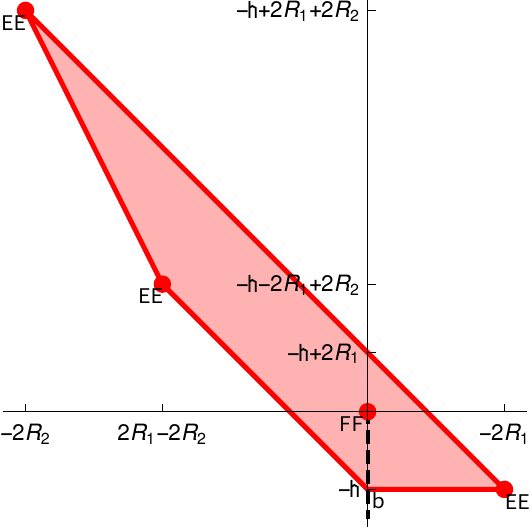}\end{center} \\

\begin{minipage}{1.5cm} $k=-1$\\$\varepsilon=+1$ \end{minipage} &
\begin{center}\includegraphics[width=3.4cm]{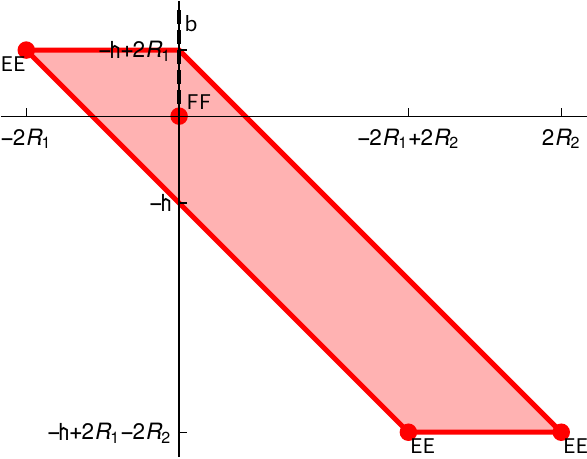}\end{center} &
\begin{center}\includegraphics[width=3.0cm]{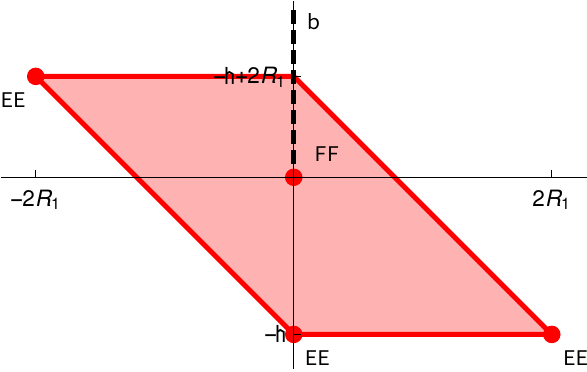}\end{center} &
\begin{center}\includegraphics[width=3.4cm]{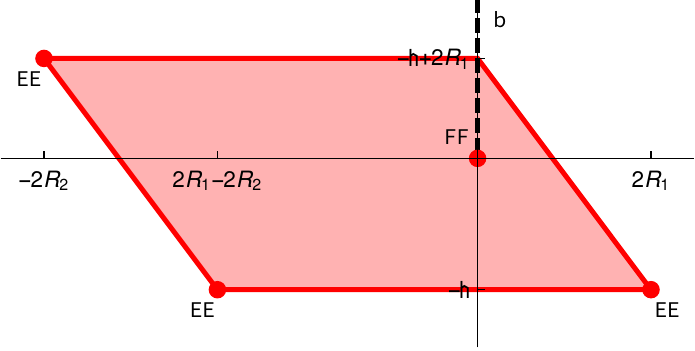}\end{center} \\

\begin{minipage}{1.5cm} $k=-1$\\$\varepsilon=-1$ \end{minipage} &
\begin{center}\includegraphics[width=3.4cm]{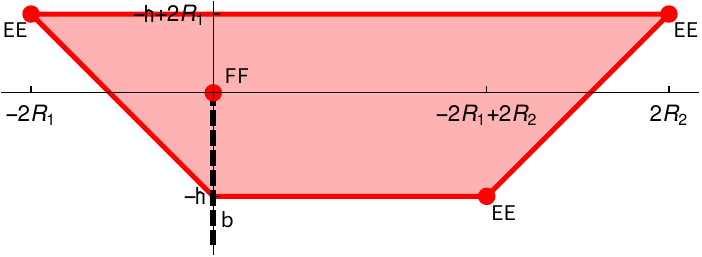}\end{center} &
\begin{center}\includegraphics[width=2.3cm]{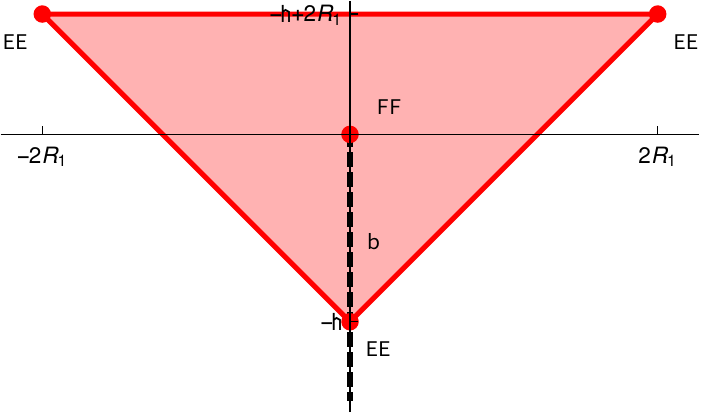}\end{center} &
\begin{center}\includegraphics[width=3.4cm]{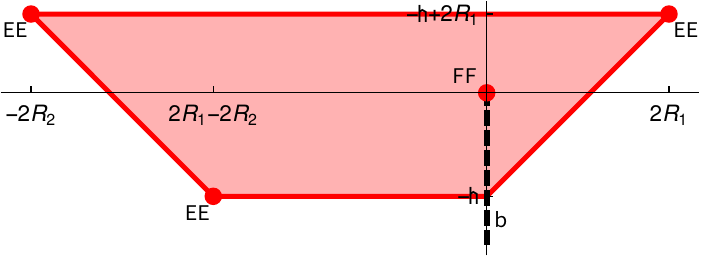}\end{center} \\

\begin{minipage}{1.5cm} $k=0$\\$\varepsilon=+1$ \end{minipage} &
\begin{center}\includegraphics[width=3.4cm]{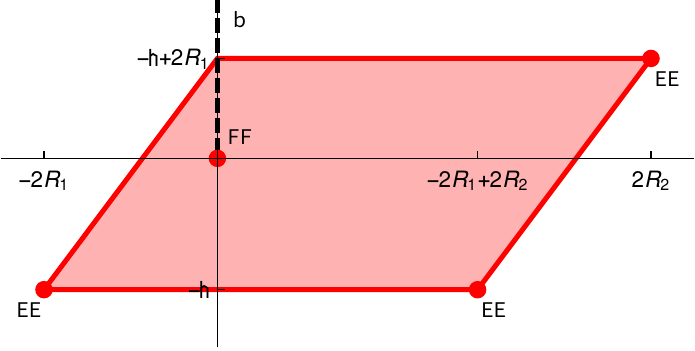}\end{center} &
\begin{center}\includegraphics[width=3.0cm]{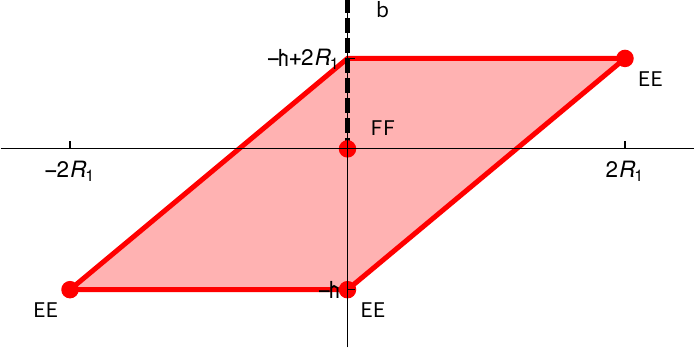}\end{center} &
\begin{center}\includegraphics[width=3.4cm]{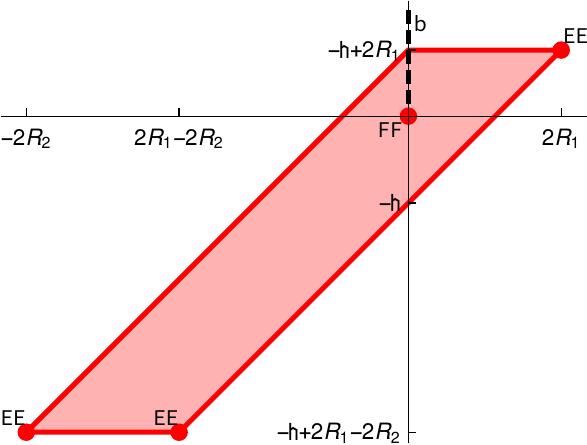}\end{center} \\

\begin{minipage}{1.5cm} $k=0$\\$\varepsilon=-1$ \end{minipage} &
\begin{center}\includegraphics[width=3.4cm]{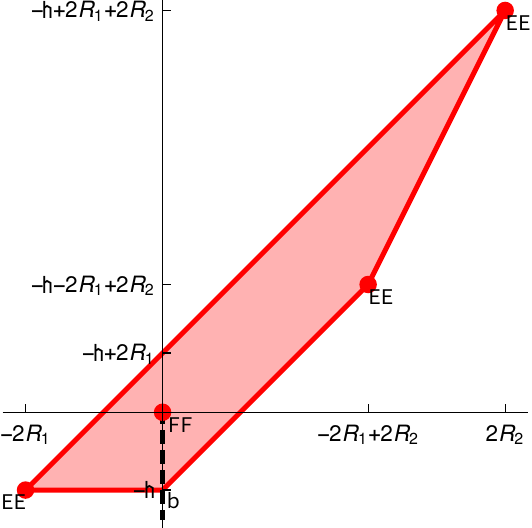}\end{center} &
\begin{center}\includegraphics[width=3.0cm]{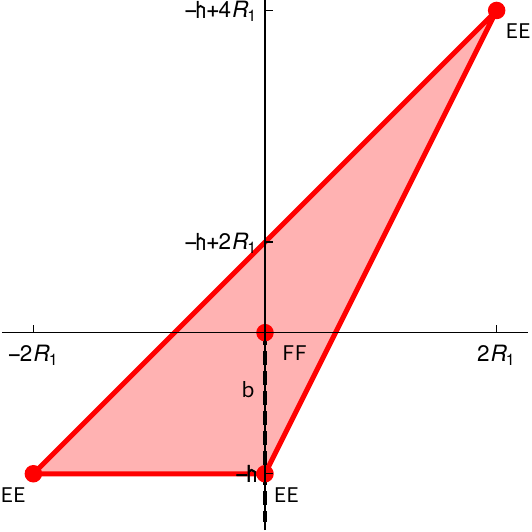}\end{center} &
\begin{center}\includegraphics[width=3.4cm]{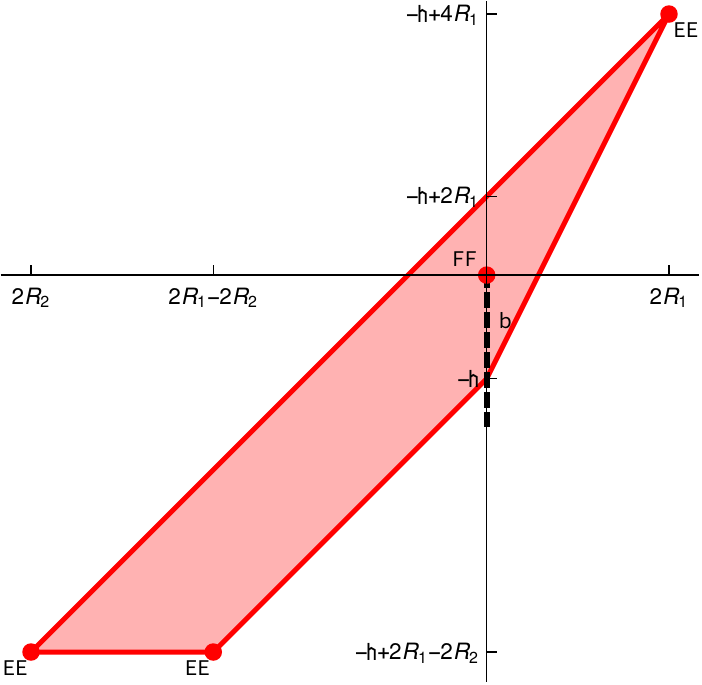}\end{center} \\
\end{tabular}
\caption{Representation of the twisting-index invariant of the coupled angular momenta for the standard case $R_1<R_2$, the Kepler problem $R_1=R_2$, and the reverse case $R_1>R_2$.}
\label{figtwistab}
\end{figure}

\newpage
\section{Semitoric invariants}
\label{sec:seminf}

In this section we briefly recall the definition of the five symplectic invariants associated to semitoric systems with one focus-focus singularity. A more detailed description of the invariants and their construction can be found in the original papers by Pelayo \& \vungoc \cite{PV1} and \cite{PV4} and in the recent notes by Sepe \& \vungoc \cite{SV}. We start with the Taylor series invariant, which can be regarded as a \emph{semi-global} invariant, i.e.\ it only depends on the characteristics of the system in neighbourhoods of the focus-focus critical fibres, and therefore can be constructed in more general classes of systems (see \cite{PRV}). After that we introduce the other four symplectic invariants, which are of \emph{global} nature since their construction depends on the system as a whole. 

\subsection{Setting}
Let $(M,\om)$ be a 4-dimensional connected symplectic manifold. A \emph{Hamiltonian function} is a smooth function $f: M \to \R$. Using the non-degeneracy of the symplectic form $\om$ we define the vector field $\mcX_f$, called the \emph{Hamiltonian vector field} associated to $f$, via $\om(\mcX_f,\cdot)=-df$. The flow of $\mcX_f$ is the \emph{Hamiltonian flow} of $f$. Given two Hamiltonian functions $f,g:M\to \R$, the symplectic form $\om$ induces the \emph{Poisson bracket}, defined as $\{f,g\} := \om(\mcX_f,\mcX_g) = -df(\mcX_g) = dg(\mcX_f)$. If the Poisson-bracket vanishes, i.e.\ $\{f,g\}=0$, then $f$ and $g$ are said to \emph{Poisson-commute}, which has as a consequence that each of the functions is constant along the Hamiltonian flow of the other. 

Consider now a pair of Hamiltonian functions $(L,H):M \to \R^2$. The triple $(M,\om,(L,H))$ is said to be a \emph{completely integrable system} with two degrees of freedom if $L$ and $H$ Poisson-commute and their differentials $dL$, $dH$ are linearly independent almost everywhere. In such a case, the \emph{momentum map} $F:=(L,H):M \to \R^2$ induces a fibration on $M$. The points where the differential $dF$ has no maximal rank, i.e.\ where $dL$, $dH$ fail to be linearly independent, are called \emph{singularities} or \emph{critical points}. The points that are not critical are said to be \emph{regular}.

The dynamics of a completely integrable system on \emph{regular fibres}, i.e.\ compact connected fibres containing only regular points, occurs along so-called \emph{Liouville tori} and is well understood thanks to the \emph{Arnold-Liouville theorem} (cf.\ Arnold \cite{Ar}). Therefore, the key to understanding the dynamical particularities of a system must reside in the fibres that contain at least one critical point, the \emph{singular fibres}. Eliasson \cite{El1,El2} and Miranda \& Zung \cite{MZ} described the singularities of completely integrable systems using normal forms. Restricting ourselves to the 4-dimensional case $(M,\om,(L,H))$, let $m \in M$ be a \emph{non-degenerate} singularity (cf.\ Bolsinov \& Fomenko \cite{BoF} or Vey \cite{Ve} for a precise definition). Then there are local symplectic coordinates $(x_1,y_1,x_2,y_2)$ defined on a neighbourhood $U \subseteq M$ and centered at $m$, together with functions $Q_1,Q_2:U \to \R$ such that $\{L,Q_i\} = \{ H,Q_i \}=0$, $i=1,2$ and the $Q_i$'s take one of the following forms:
\begin{itemize}
	\item \emph{Regular component:} $Q_i(x_1,y_1,x_2,y_2)=y_i$.
	\item \emph{Elliptic component:} $Q_i(x_1,y_1,x_2,y_2) = \dfrac{{x_i}^2+{y_i}^2}{2}.$
	\item \emph{Hyperbolic component:} $Q_i(x_1,y_1,x_2,y_2)=x_i y_i.$
	\item \emph{Focus-focus components} (always come in pairs):
	\begin{equation*}
	\begin{cases}
	Q_1(x_1,y_1,x_2,y_2) = x_1 y_2 - x_2 y_1 \\
	Q_2(x_1,y_1,x_2,y_2) = x_1 y_1 + x_2y_2.
	\end{cases}
	\end{equation*}
\end{itemize}

A \emph{semitoric system} is a 4-dimensional completely integrable system $(M,\om,F=(L,H))$ with 2 degrees of freedom satisfying that all singularities are non-degenerate and have no hyperbolic components, the map $L$ is \emph{proper} (i.e.\ the preimage of a compact set by $L$ is again compact) and it induces a faithful Hamiltonian $\mbS^1$-action on $M$. The flow of $L$ is thus $2\pi$-periodic. 

In order to keep the notation light, we restrict ourselves to semitoric systems with only one focus-focus singularity $m\in M$, with critical value $c:=F(m)$.


\subsection{The Taylor series invariant}
\label{sec:taylor}
Let us consider the local model of a focus-focus point $(\R^4,\omc,Q)$, where $(x_1,y_1,x_2,y_2)$ are local coordinates in $\R^4$, $\omc = dx_1 \wedge dy_1 + dx_2 \wedge dy_2$ and $Q:=(Q_1,Q_2)$ are given by
\begin{equation}
Q_1(x_1,y_1,x_2,y_2):=x_1 y_2 - x_2y_1,\qquad Q_2(x_1,y_1,x_2,y_2) := x_1 y_1 + x_2 y_2.
\label{qs}
\end{equation} By the normal form theorem by Eliasson \cite{El1,El2} and Miranda \& Zung \cite{MZ} we find neighbourhoods $U\subseteq M$ of $m$ and $V \subseteq \R^4$ of 0, a symplectomorphism $\varphi: (U,\om|_{U}) \to (V,\omc|_{V})$ and a diffeomorphism $g=(g^{(1)}, g^{(2)}):F(U) \to Q(V)$ such that the following diagram commutes:

\begin{center}
\begin{tikzpicture}
 \matrix (m) [matrix of math nodes,row sep=4em,column sep=6em,minimum width=2em]
 {
 U \subseteq M & F(U) \subseteq \R^2 \\
 V \subseteq \R^4 & Q(V) \subseteq \R^2 \\};
 \path[-stealth]
 (m-1-1) edge node [left] {$\varphi$} 
 				node [below,rotate=90] {$\sim$} (m-2-1)
 			edge node [above] {$F$} (m-1-2)
 (m-1-2) edge node [left] {$g$}
 				node [below,rotate=90] {$\sim$} (m-2-2)
 (m-2-1) edge node [above] {$Q$} (m-2-2);
\end{tikzpicture}
\end{center}


The pair $(\varphi,g)$ is called \emph{Eliasson isomorphism} for $(M,\om,F)$. It can be chosen such that it satisfies
\begin{equation*}
L|_U = Q_1 \circ \varphi,\qquad \frac{\partial g^{(2)}}{\partial h}>0,
\end{equation*} which implies that $g$ is of the form $g(l,h) = (l,g^{(2)}(l,h))$ and thus orientation-preserving. Using the diffeomorphism $g$ we can make our construction semi-global by defining the map $\Phi:W \to Q(V)$ by $\Phi:= g \circ F$, where $W:=F^{-1}(F(U)) \subseteq M$ is a neighbourhood of the focus-focus fibre. The map $\Phi=(\Phi_1,\Phi_2)$ agrees with $Q \circ \varphi$ on $U$, which means that the flow of $\mathcal{X}^{\Phi_1}$ is $2\pi$-periodic and generates an effective $\mbS^1$-action on the regular fibres close to the singular fibre. Define a complex coordinate $w=(w_1,w_2)=w_1 + i w_2$ on $Q(V) \subseteq \R^2 \simeq \C$ and consider a regular fibre $\Lam_w := \Phi^{-1}(w)$ for $w$ close to the origin. From any point $a \in \Lam_z$ we follow the Hamiltonian flow of $\mcX^{\Phi_2}$ until we reach again the $\Phi_1$-orbit of $a$ for the first time. Once at this point, we follow the Hamiltonian flow of $\mcX^{\Phi_1}$ until we are back at $a$, as illustrated in Figure \ref{cycles2}.

\begin{figure}[ht]
 \centering
 \includegraphics[width=9cm]{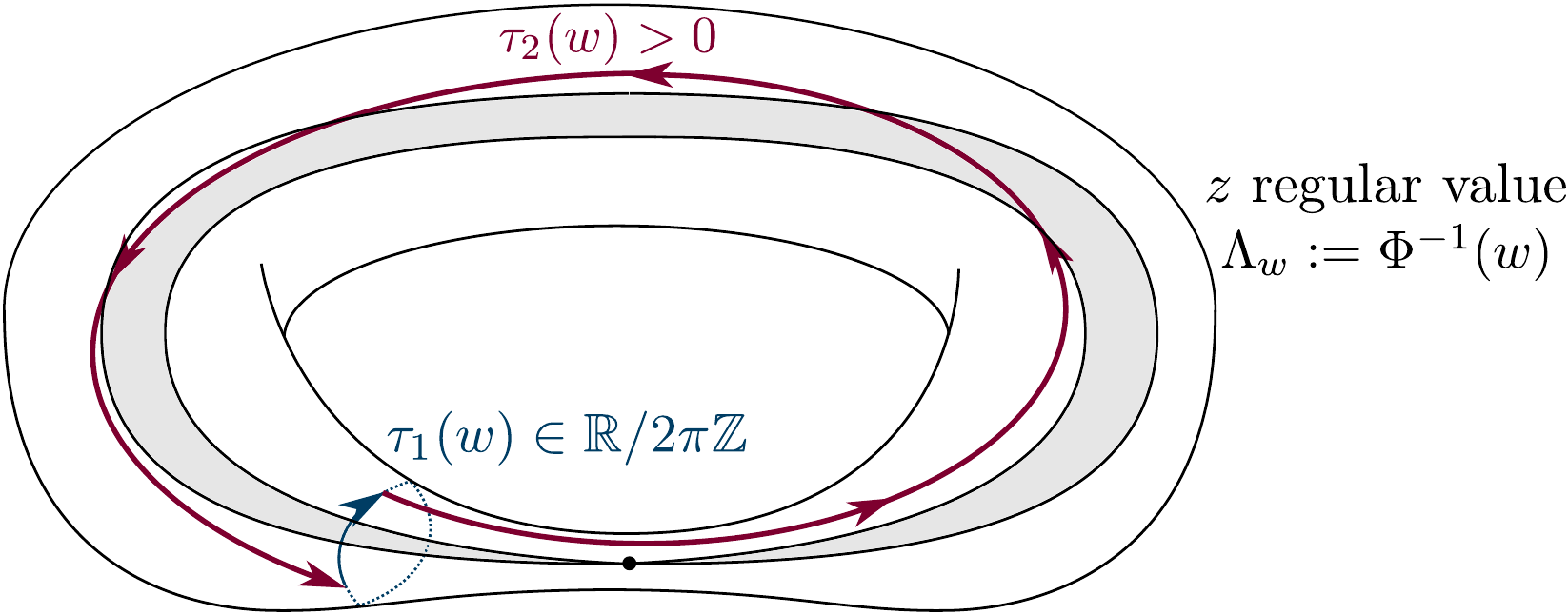}
 \caption{\small A regular fibre $\Lambda_z$ close to the singular fibre $\Lambda_0$ (gray). First we follow the flow generated by $\Phi_2$ (red) and then the flow generated by $\Phi_1$ (blue).}
 \label{cycles2}
\end{figure}

Denote by $\tau_1(w)\in \R / 2\pi \Z$ the time needed for the displacement following $\mcX^{\Phi_1}$ and $\tau_2(w)>0 $ for the displacement following $\mcX^{\Phi_2}$, which are independent of $a$. For any choice of determination $\ln$ of the complex logarithm, define
\begin{equation}
\begin{cases}
\sigma_1 (w): = \tau_1(w) - \imp(\ln w)\\
\sigma_2 (w): = \tau_2(w) + \rep (\ln w).
\end{cases}
\label{taus}
\end{equation} \vungoc proved in \cite{Vu1} that $\si_1$ and $\si_2$ extend to smooth single-valued functions around the origin and that the 1-form
\begin{equation}
\sigma := \sigma_1 dw_1+ \sigma_2 dw_2
\label{sigma}
\end{equation} is closed. For later convenience, we impose $\sigma_2(0)\in[0,2\pi[\ $, which in turn fixes the determination of the logarithm. We keep this choice throughout the paper.

\begin{de}[From \cite{Vu1}]
\label{defS}
Let $S$ be the unique smooth function defined around $0 \in \mathbb{R}^2$ such that \begin{equation*}
 \begin{cases}
 \dee S =\sigma\\
S(0)=0,
 \end{cases}
 \end{equation*} where $\sigma$ is the one form given by \eqref{sigma}. The Taylor series of $S$ at $(0,0)$ is denoted by $(S)^\infty$. We say that $(S)^\infty$ is the Taylor series invariant of the system.
\end{de}

By construction the Taylor series $(S)^\infty$ has no constant term and $\frac{\partial S}{\partial z_1}(0) \in [0,2\pi[\ $. The function $S(w)$ can be interpreted as a regularised action. Let $\varpi$ be a semi-global primitive of the symplectic form $\om$, $\be_w \subset \Lam_w$ the trajectory described in the definition of the times $\tau_i(w)$ and $\mcA(w) := \oint_{\be_w} \varpi$. Then,
\begin{equation}
S(w) = \mcA(w) - \mcA(0)+ \imp (w \ln w - w).
\label{inv5} 
\end{equation}

A small remark is that due to the different conventions in the literature on the definition of the Taylor series invariant, the linear coefficient in the first variable can sometimes appear shifted by $\frac{\pi}{2}$ (compare Le Floch $\&$ Pelayo \cite{LFP}, Pelayo $\&$ Ratiu $\&$ \vungoc\!\! \cite{PRV}, Pelayo $\&$ \vungoc\!\! \cite{PV3}, Sepe $\&$ \vungoc\!\! \cite{SV} with Dullin \cite{Du}, Pelayo $\&$ \vungoc\!\! \cite{PV1}), or scaled by $2\pi$ (cf.\ Sepe $\&$ \vungoc\!\! \cite{SV}).

\subsection{The number of focus-focus points invariant}
The number $\nff \in \N \cup \{0\}$ of singularities of focus-focus type is the second symplectic invariant of the semitoric system $(M,\om,F)$. By restriction to $t \in\ ]t^-, t^+[$ we have $\nff = 1$.

\subsection{The polygon invariant} 

\label{sec:polygon}
Let $B:=F(M)$ be the image of the momentum map. We denote vertical lines by $b_\lam := \{(\lam,y)\;|\;y \in \R\} \in \R^2$ and consider a linear transformation $T^k$, $k \in \Z$, where

\begin{equation}
 T^k := \begin{pmatrix}
1& 0 \\ k & 1
\end{pmatrix} \in \text{GL}(2,\Z).
\label{TT}
\end{equation} For a given $n \in \Z$ we define the transformation $t^n_{b_\lam}:\R^2 \to \R^2$ as the identity on the left half-plane and $T^n$ on the right half-plane with respect to $b_\lam$. Given a sign $\epsilon \in \{-1,+1\}$, let $b_{\lam}^{\epsilon} \subset b_{\lam}$ be the half-line that starts in $c$ and extends upwards if $\epsilon=+1$ or downwards if $\epsilon=-1$.  

A subset $\De \subseteq \R^2$ is said to be a \emph{convex polygon} if it is the -- possibly infinite -- intersection of closed half-planes such that the number of corner points on each compact subset of $\De$ is finite. If the slopes of all the edges are moreover rational, $\De$ is said to be \emph{rational}. Pelayo $\&$ \vungoc\ \cite{Vu2} show that, for a given choice of sign $\epsilon \in \{-1,+1\}$, there exists a map $f=f_{\epsilon}:B \to \R^2$ such that, among other properties, is a homeomorphism onto its image $\De = f(M)$ and satisfies that $\De$ is a rational convex polygon, that $f|_{B \backslash b^{\epsilon}_\lam}$ is a diffeomorphism into its image and that $f$ preserves the first coordinate, i.e.\ $f(l,h) = (l,f^{(2)}(l,h))$. This map is often called the \emph{cartographic homeomorphism} (cf.\ Sepe $\&$ \vungoc\ \cite{SV}).

Let $\mcV$ denote the subgroup of the group of integral-affine transformations Aff$(2,\Z):=GL(2,\Z) \ltimes \R^2$ consisting of the transformations obtained as a composition of a vertical translation with a $T^k$ as in \eqref{TT} for some $k \in \Z$. Then the map $f$ is unique up to a left composition with an element of $\tau \in \mcV$. Such a composition changes the map by $f' = \tau \circ f$ and its image by $\De' = \tau(\De)$. Similarly, a different choice of sign $\epsilon' \in \{-1,+1\}$ changes the map by $f_{\epsilon'} = t_{u} \circ f_{\epsilon}$ and its image by $\De' = t_{u}(\De)$, where $u = \frac{\epsilon-\epsilon'}{2}$. 

\begin{de}
A \emph{weighted polygon} is a triple of the form
$$ \De_\text{weight} := (\De,b_{\lam},\epsilon)$$ where $\De \subseteq \R^2$ is a rational convex polygon, $b_{\lam} \subset \R^2$ is a vertical line and $\epsilon \in \{-1,+1\}$. The space of weighted polygons is $\mcW$Polyg$(\R^2)$. 
\end{de} We write $\mathbb{Z}_2:=\Z \slash 2 \Z$ and denote by $\mcG := \{ T^k \;|\;k \in \Z\}$ the group of transformations as in \eqref{TT}. Then the product group $\Z_2 \times \mcG$ acts on $\mcW Polyg(\R^2)$ via $$ (\epsilon',T^{k'}) \cdot (\De,b_{\lam},\epsilon) = (t_{u}(T^{k'}(\De)),b_{\lam},\epsilon' \epsilon).$$ This group action encodes the freedom of definition of the cartographic homeomorphism $f$. From here we finally define:

\begin{de}
 Let $(M,\om,F)$ be a semitoric integrable system with one focus-focus singularity and $f$ a cartographic homeomorphism. The \emph{polygon invariant} of the system is the $(\Z_2\times \mcG)$-orbit 
 $$ (\Z_2 \times \mcG) \cdot (\De,b_{\lam},\epsilon) \in \mcW\text{Polyg}(\R^2)/(\Z_2\times \mcG)$$
 where $\De = f(M)$. 
\end{de} 

For systems with one focus-focus singularity the polygon invariant consists thus of a collection of $\Z_2 \times \Z$ weighted polygons, i.e.\ rational convex polygons together with the specification of the line $b_{\lam}$ and the corresponding sign choice $\epsilon\in \{-1,+1\}$.


\subsection{The height invariant} 

\label{sec:height}

Let $f$ be a cartographic homeomorphism as in \S \ref{sec:polygon}. The map $\mu:M \to \De \subseteq \R^2$ defined by $\mu := f \circ F = f \circ (L,H) = (L, f^{(2)} \circ (L,H))$ is said to be a \emph{generalised toric momentum map} of the system $(M,\om,F)$. Consider the number
$$ \mfh := \mu(m) - \!\!\min_{s \in \De \cap b_{\lam}} \!\!\pi_2(s),$$
where $\pi_2:\R^2 \to \R$ is the canonical projection onto the second coordinate. The number $\mfh$ can be understood as the height of the image of the critical value $c$ under the cartographic homeomorphism, i.e.\ $\mu(m) = f(c)$, inside the rational convex polygon $\De$. A more geometric way to interpret $h$ is the following. Consider the submanifold $Y := L^{-1}(L(m)) \subset M$, which we split into $Y^+ := Y \cap \{ p \in M \; : \; H(p) > H(m)\}$ and $Y^- := Y \cap \{ p \in M : H(p) < H(m)\}$. Then the number $h$ corresponds to the symplectic volume of $Y^-$, that is, the real volume of $Y^-$ divided by 2$\pi$.

The number $\mfh$ is called \emph{height invariant} or \emph{volume invariant} and it is a symplectic invariant of the semitoric system $(M,\om,F)$. In particular, it is independent of the choice of $f$ and its image $\De$.


\subsection{The twisting-index invariant}

\label{sec:twistingIndex}
Fix a sign  $\epsilon \in \{-1,+1\}$, a cartographic homeomorphism $f=f_{\epsilon}$ and consider the map $\Phi: W \to F(U)$ from \S \ref{sec:twistingIndex}, where $\Phi =(\Phi_1,\Phi_2) = (L|_W,\Phi_2)$. On each regular torus fibre in $W$ we can construct a Hamiltonian vector field
\begin{equation}
 2\pi \mcX := (\tau_1 \circ \Phi) \mcX_{\Phi_1} + (\tau_2 \circ \Phi) \mcX_{\Phi_2} = (\tau_1 \circ \Phi) \mcX_L + (\tau_2 \circ \Phi) \mcX_{\Phi_2},
 \label{Xp}
\end{equation} where $\tau_1, \tau_2$ are the functions defined in \eqref{taus}. This vector field is smooth on $F^{-1}(F(U)\backslash b_{\lam}^{\epsilon})$ and there exists a unique smooth function $I: F^{-1}(F(U)\backslash b_{\lam}^{\epsilon}) \to \R$ whose vector field is $\mcX$ and such that $\lim_{x \to m} I(x)=0$, as proven in Lemma 5.6 of Pelayo $\&$ \vungoc\!\! \cite{PV1}. This map is of the form $I = \mathfrak{p} \circ \Phi$, where $\mathfrak{p}$ satisfies $\partial_i \mathfrak{p} = \tau_i / 2\pi$, $i=1,2$, and as a consequence is of the form 
\begin{equation*}
2\pi \mathfrak{p} (w) = S(w) - \imp(w \ln w - w) + Const.
\label{strucp}
\end{equation*}
We define the \emph{privileged momentum map} around the focus-focus point $m$ as $\nu := (L,I)$. The map $\nu$ can be related to the generalised toric momentum map $\mu$ corresponding to $f$ by
\begin{equation}
\mu = T^{k} \circ \nu\quad \text{ on }W,
\label{twistwis}
\end{equation}
for a certain $k \in \Z$, sometimes refered to as {\em twisting index of $\De_\text{weight}$ at the focus-focus critical value $c$}. However, this integer is dependent on our choice of $f$. In particular, $k$ is unchanged by the action of $\Z_2$, but if we apply a global transformation $T^{k'} \in \mcG$, then $\mu$ transforms into $T^{k'} \circ \mu$ while $\nu$ remains unchanged, which implies that $k $ becomes $k' + k$. 

This group action can be formally written as follows. Let $\mcW$Polyg$(\R^2) \times \Z$ be the space of weighted polygons together with their corresponding twisting index $k$. The group $\Z_2 \times \mcG$ acts on $\mcW$Polyg$(\R^2) \times \Z$ as
$$ (\epsilon',T^{k'}) \star (\De,b_{\lam},\epsilon,k): = (t_{u}(T^{k'}(\De)),b_{\lam},\epsilon' \epsilon,k'+k).$$ 

This allows to define the twisting-index invariant for systems with one focus-focus singularity. 

\begin{de}
The {\em twisting-index invariant} of $(M,\omega,(L,H))$ is the $(\Z_2 \times \mcG)$-orbit 
$$ (\Z_2 \times \mcG) \star (\De,b_{\lam},\epsilon,k) \in (W\text{Polyg}(\R^2) \times \Z) / (\Z_2 \times \mcG)$$
of weighted polygons labelled by the twisting index at the focus-focus singularity of the system.	
\end{de}

For systems with one focus-focus singularity, the twisting-index invariant is thus completely determined by the association of an index $k$ to one of the weighted polygons of the polygon invariant or, equivalently, by finding the polygon corresponding to $k=0$. The rest of the associations can be calculated using the fact that $\Z_2$ does not act on $k$ and $\mcG \simeq \Z$ acts by addition.

\section{The Taylor series invariant of the coupled angular momenta}
\subsection{Coupled angular momenta}
\label{sec:CAM}
Let $R_1,R_2$ be two positive real numbers and $t\in[0,1]$ a parameter. Consider the product manifold $M=\mbS^2 \times \mbS^2$ with symplectic form $\omega = - (R_1 \omega_{\mbS^2} \oplus R_2 \omega_{\mbS^2})$, where $\omega_{\mbS^2}$ is the standard symplectic form of the unit sphere $\mbS^2$. Let $(x_1,y_1,z_1,x_2,y_2,z_2)$ be coordinates on $M$, where $(x_i,y_i,z_i),$ $i=1,2$, are Cartesian coordinates on the unit sphere $\mbS^2 \subset \R^3$. The \emph{coupled angular momenta} is a family of 4-dimensional  completely integrable system $(M,\omega,(L,H))$, where the smooth functions $L,H:M \to \mathbb{R}$ are given by
\begin{equation}
\begin{cases}
L(x_1,y_1,z_1,x_2,y_2,z_2)\;:= R_1(z_1-1) + R_2 (z_2+1), \\
H(x_1,y_1,z_1,x_2,y_2,z_2):= (1-t) z_1 + t (x_1x_2 + y_1y_2 + z_1z_2) +2t-1.\\
\end{cases}
\label{sys1}
\end{equation} and $R_1 < R_2$. The case $R_1 > R_2$ will be called \emph{reverse coupled angular momenta}. 

Sadovksii \& Zhilinskii proved in \cite{SZ} that $L$ and $H$ Poisson-commute, i.e.\ $\{L,H\}=0$, and that the system has four critical points of maximal corank, namely $(0,0,\pm 1,0,0,\pm 1)$. All of them are of elliptic-elliptic type except for $m=(0,0,1,0,0,-1)$, which is non-degenerate and of focus-focus type for $t \in\ ]t^-,t^+[$ and degenerate for $t\in \{t^-,t^+\}$, where
\begin{equation*}
t^\pm = \dfrac{R_2}{2R_2+R_1\mp 2\sqrt{R_1 R_2}}.
\end{equation*} 

In particular $0<t^- < \frac{1}{2}  < t^+\leq 1$ so for the value $t=\frac{1}{2}$ there is always a focus-focus singularity. If $t<t^-$ or $t > t^+$ then $m$ is of elliptic-elliptic type. From now on we will assume that $t \in\ ]\ t^-,t^+[ $. Note that in \eqref{sys1} we have conveniently shifted $L$ and $H$ so that $(L,H)(m)=(0,0)$. The real part of the eigenvalue of the linearisation of the Hamiltonian vector field at the equilibrium point $m$ for $t \in\ ]t^-,t^+[$ is given by $\frac{r_A}{2R}$ which we now define for later convenience.

\begin{de}The \emph{discriminant square root} $\ra$ is defined as 
\begin{equation}
\ra: = \sqrt{-R^2 (1 - 2 t)^2 + 2 R t - t^2 } = \sqrt{(1 + 4 {R}^2)(t-t^-)(t^+-t)}, \quad R := \frac{R_2}{R_1}.
\label{RA} 
\end{equation} Note that $r_A$ is real when $t^- < t < t^+$, that is when $m$ is of focus-focus type and vanishes when $t = t^\pm$. When $m$ is of elliptic-elliptic type, the argument of the square root becomes negative. 
\label{sqr}
\end{de}

\begin{figure}[ht]
 \centering
 \includegraphics[width=7cm]{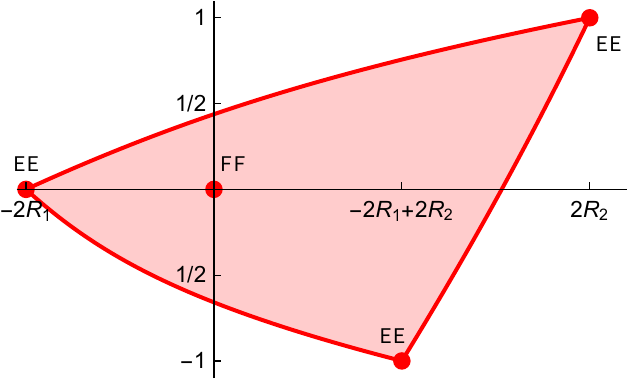}
 \caption{\small Image of the momentum map $F=(L,H)$ of the coupled angular momenta for $t=1/2$. The system has three singularities of elliptic-elliptic type and one of focus-focus type.}
 \label{figmom}
\end{figure}

Using the discrete symmetries of the system, the case $R_1 > R_2$ can be mapped to the cases $R_1 <  R_2$. The discrete symmetries are reflections about any plane that contains the $z$-axis of each sphere, e.g.\ the $xz$-plane, so that $y_i \to - y_i$. The reflection through any other plane can be obtained by composition with the $\mbS^1$-action. Recalling that $(x_i, y_i, z_i)$ with $i=1,2$ is an angular momentum, the natural time reversal is to flip all the sign, thus reversing the angular momenta, as Sadovksii \& Zhilinskii point out in \cite{SZ}. This does reverse $L$ (up to a constant) but changes $H$. If we modify $H$ by adding the right amount of $L$ we can construct an $H'$ that is a linear combination of $x_1 x_2 +y_1 y_2 + z_1 z_2$ and $z_1 - z_2$ so that it is invariant when in addition to flipping all signs also the indices of the coordinates are exchanged.

When $R_1 \not = R_2$ then also the indices of $R_1$ and $R_2$ need to be exchanged, thus changing the symplectic form. Composing these three symmetries leads to the following result, which relates the \emph{reverse} case  $R_1 >R_2$ to the standard case $R_1 < R_2$.
The exchange of $R_1$ and $R_2$ may seem surprising, but if instead of working with coordinates $x_i^2 + y_i^2 + z_i^2 = 1$ restricted to the unit sphere we would work on spheres of arbitrary size then the symplectic form would be free of parameters.

\begin{pro}
\label{proptrans}
The reverse coupled angular momenta are isomorphic as a semitoric system to the standard coupled angular momenta with the opposite sign of L.
\end{pro}

\begin{proof}
Consider $M=\mbS^2 \times \mbS^2$ with coordinates $(x_1,y_1,z_1,x_2,y_2,z_2) \in \R^3 \times \R^3$, positive constants $R_1, R_2>0$ and symplectic form
\begin{align*}
 \omega = - (R_1 \omega_{\mbS^2} \oplus R_2 \omega_{\mbS^2}) =& - R_1 (x_1 \dee y_1 \wedge \dee z_1 + y_1 \dee z_1 \wedge \dee x_1 + z_1 \dee x_1 \wedge \dee y_1) \\&- R_2 (x_2 \dee y_2 \wedge \dee z_2 + y_2 \dee z_2 \wedge \dee x_2 + z_2 \dee x_2 \wedge \dee y_2).
\end{align*} 
For $b \in \R$, define the integrable system $(M,\omega,(L',H'))$ given by
\begin{equation}
\begin{cases}
\,L'(x_1,y_1,z_1,x_2,y_2,z_2):= R_1  z_1 + R_2 z_2 \\
H'(x_1,y_1,z_1,x_2,y_2,z_2):= t (x_1 x_2 + y_1 y_2 + z_1 z_2) + b(z_1-z_2).\\
\end{cases}
\label{relsys} 
\end{equation} 
To prove that this system is isomorphic as a semitoric system to the standard coupled angular momenta $(M,\omega,(L,H))$ we must find a suitable choice of $b$ and a smooth function $g$ such that $\frac{\partial g}{\partial H'}>0$ and $(L,H) = (L',g(L',H'))$ up to adding constants. By choosing 
\begin{equation}
b = \dfrac{(1-t) R_2}{R_1+R_2} \quad \mbox{and}\quad  \gamma = \dfrac{1-t}{R_1+R_2}, 
\label{fort}
\end{equation}
the function $g(l',h') := h' + \gamma l' $ has the desired properties. Denote now by $\omega'$ the symplectic form given by $\omega$ with $R_1$ and $R_2$ exchanged. Consider the smooth transformation $\Xi: (M,\om) \to (M,\om')$ defined by 
$$
\Xi(x_1,y_1,z_1,x_2,y_2,z_2) := (x_2',-y_2',-z_2',x_1',-y_1',-z_1') \quad \text{ and swap $R_1$ and $R_2$}.$$ 

$\Xi$ is a symplectomorphism, i.e.\ $\Xi^* \om' = \om$. If we now apply $\Xi$ to the system \eqref{relsys} we recover the same system but with the sign of $L'$ changed. This system has thus the desired discrete symmetry.
\end{proof}

It is also possible to analyse directly the more symmetric system $(L',H')$ instead of $(L,H)$, but we wanted to stick to the conventions
started by Sadovksii $\&$ Zhilinskii \cite{SZ} and continued in Le Floch $\&$ Pelayo \cite{LFP}. The proposition shows that from the point of view of the semitoric invariants there is no difference between these systems.

From now on we will consider $R_1 < R_2$ unless otherwise stated. Let $(\theta_i,z_i)$ be symplectic cylindrical coordinates on $\mbS^2$. We rewrite the system \eqref{sys1} using symplectic coordinates $(\theta_1,z_1,\theta_2,z_2)$ on $M$ obtaining
\begin{equation*}
\begin{cases}
L(\theta_1,z_1,\theta_2,z_2)\,= R_1(z_1-1) + R_2 (z_2+1)\\
H(\theta_1,z_1,\theta_2,z_2)= (1-t) z_1 + t\sqrt{(1-{z_1}^2)(1-{z_2}^2)}\cos(\theta_1-\theta_2)+ t z_1z_2 +2t-1.\\
\end{cases}
\end{equation*} The symplectic form becomes $\omega = -(R_1\, \dee \theta_1 \wedge \dee z_1 + R_2\, \dee \theta_2 \wedge \dee z_2)$. We now perform the affine coordinate change 
\begin{equation}
\begin{array}{lll}
\qti_1 := -\theta_1 && \qti_2 := \theta_1-\theta_2 \\[0.4cm]
\pti_1 := R_1(z_1-1) + R_2 (z_2+1) && \pti_2 := R_2 (z_2+1)
\end{array}
\label{change}
\end{equation} such that $L(\qti_1,\pti_1,\qti_2,\pti_2)=\pti_1$ and
\begin{align}
H(\qti_1,\pti_1,\qti_2,\pti_2) = \dfrac{1}{R_1 R_2}& \left( R_2 \pti_1(1-2t) + (R_1 t -R_2(1-2t) + \pti_1t + 2R_2 t )\pti_2 - t {\pti_2}^2 \right.\nonumber \\ 
&+ \left. t \sqrt{\pti_2(\pti_2-\pti_1)(\pti_2-2R_2)(\pti_2-\pti_1-2R_1)}\cos( \qti_2) \right).
\end{align} The symplectic form in these coordinates becomes the standard symplectic form $\omega = \dee \qti_1 \wedge \dee \pti_1 + \dee \qti_2 \wedge \dee \pti_2$. In these coordinates, $H$ is independent of $\qti_1$, which means that the function $L(\qti_1,\pti_1,\qti_2,\pti_2)=\pti_1$, is conserved by the flow of $H$ and is thus a first integral of the system. The function $L$ corresponds to the total angular momentum of the system and induces a global $\mbS^1$-action. 

In order to make our notation lighter, we will scale some variables and functions by a factor $R_1$:
\begin{equation}
\begin{array}{lllll}
q_1 := \qti_1 && q_2 := \qti_2 && \mcL := \frac{1}{R_1}L\\[0.4cm]
p_1 :=\frac{1}{R_1} \pti_1 && p_2 :=\frac{1}{R_1} \pti_2 && \mcH := H.
\end{array}
\label{chan2}
\end{equation} This way we can express our problem entirely in terms of $R:= \frac{R_2}{R_1}$. We perform now a symplectic reduction on the level $\mcL=l$, which we will call the \emph{reduced coupled angular momenta} system. Expressed in coordinates $(q_2,p_2)$, we obtain the reduced Hamiltonian
\begin{align}
\mcH_l(q_2,p_2) := \dfrac{1}{R}& \left( R l(1-2t) + ( t -R(1-2t) + lt + 2R t )p_2 - t {p_2}^2 \right.\nonumber \\ 
&+ \left. t \sqrt{p_2(p_2-l)(p_2-2R)(p_2-l-2)}\cos( q_2) \right).
\label{hred}
\end{align} 
\eqref{change} and \eqref{hred} imply the following domains of definition (or physical reagions) of $l$ and $p_2$, namely $l \in [-2,2R]$ and $p_2$ satisfying the inequalities $p_2\geq 0$, $p_2 \geq l$, $p_2 \leq 2R$ and $p_2 \leq l + 2$. Otherwise $\mcH_l$ would not be well-defined. Figure \ref{physical} displays a representation of the physical region. 

\begin{figure}[ht]
 \centering
 \includegraphics[width=10cm]{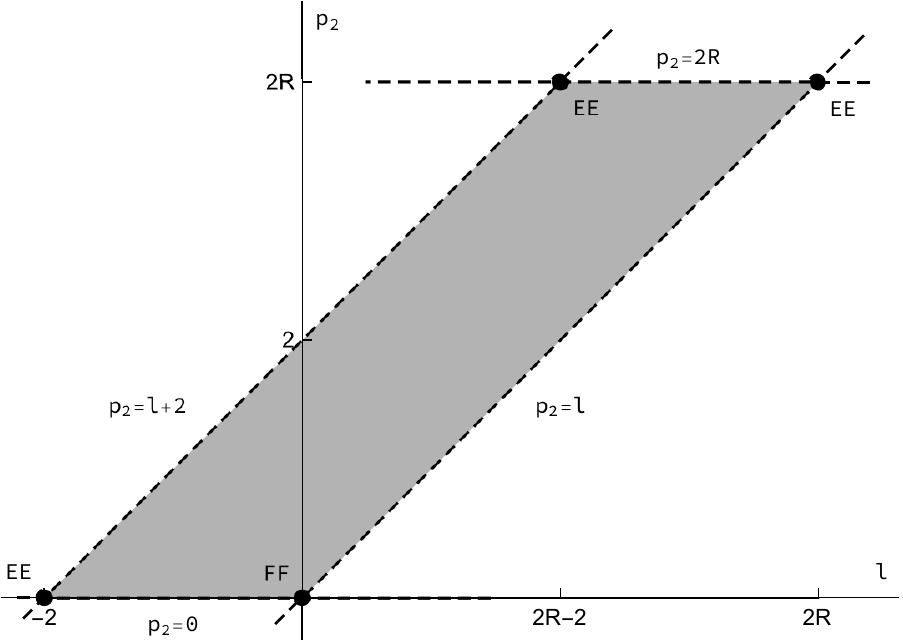}
 \caption{\small Physical region represented in the plane $(l,p_2)$. The fixed points marked as ``EE" are always of elliptic-elliptic type are and the fixed point marked as ``FF" is of focus-focus type for $t \in\ ]t^-,t^+[$ and of elliptic-elliptic type for $t \notin [t^-,t^+]$.
 }
 \label{physical}
\end{figure}

It is useful to introduce the following polynomials in $p_2$:
\begin{align}
A(p_2) &:= \dfrac{1}{R} \left( R l(1-2t) + ( t -R(1-2t) + lt + 2R t )p_2 - t {p_2}^2\right) \label{ABdef} \\[0.15cm] 
B(p_2) &:= 	\dfrac{t^2}{R^2}\, p_2(p_2-l)(p_2-2R)(p_2-l-2)\nonumber 
\end{align} so that we can write
\begin{equation}
\mcH_l(q_2,p_2) = A(p_2) + \sqrt{B(p_2)} \cos(q_2).
\label{red}
\end{equation} 

For a fixed energy value $\mcH_l=h$, the physical motion will happen along the curve $h = A(p_2) + \sqrt{B(p_2)} \cos(q_2)$. We see that $B(p_2)$ vanishes along the edges of the parallelogram in Figure \ref{physical}, thus there we have $\mcH_l(q_2,p_2)=A(p_2)$, which is independent of $q_2$. More precisely, along the four edges we will have
\begin{equation}
	\begin{array}{lll}
		\mcH_l(q_2,0) = A(0)=(1-2t)l,	&& \mcH_l(q_2,2R)= A(2R)=l-2R+2t,\\[0.2cm]
		\mcH_l(q_2,l)=A(l)=\dfrac{t}{R}l, && \mcH_l(q_2,l+2)=A(l+2)=-2+4t-\dfrac{t}{R}(2+l).
	\end{array}
\label{edgeval}
\end{equation} These values will be important because their corresponding level sets correspond to separatrices in the phase portrait of $\mcH_l(q_2,p_2)$.

Let us now define the polynomial 
\begin{equation}
	P(p_2) := B(p_2) - (h-A(p_2))^2.
	\label{Pdef}
\end{equation} 
Note that even though it appears as if $P(p_2)$ has degree 4, in fact it only has degree 3 because the leading order term of $A(p_2)$ squared is equal to the leading order term in $B(p_2)$. Regarding the dependence on $h$ and $l$, clearly $P$ is quadratic in $h$, $A$ is linear in $l$, and $B$ is quadratic in $l$, so that $P$ is quadratic in $l$.

\begin{de} Let $\mathcal{P}(z)$ be a real polynomial in $z$. We say that the (hyper)elliptic curve $w^2=\mathcal{P}(z)$ \emph{avoids a root} at a position $p$ if $-\mathcal{P}(p)$ is a perfect square, that is, if $-\mathcal{P}(p)$ is the square of a polynomial.
\end{de}

In our case, it is obvious from the definition of $P$ that the elliptic curve $s^2=P(p_2)$ avoids the roots $p_2=0,$ $l,$ $2R,$ $l+2$. 
It is helpful to know the positioning of the roots $\zeta_1, \zeta_2, \zeta_3$ of $P$ relative to the positions of the avoided roots. There is always one non-positive root $\zeta_1$ that collides with zero at the focus-focus critical value. If we look at Figure \ref{physical}, in the region without any physical motion (white colour) $P$ has two complex roots. In the physical region (grey colour) $P$ always has two non-negative roots $\zeta_2,$ $\zeta_3$ lying inside the parallelogram. When $(l,h) \not = (0,0)$, these roots are different. More precisely, the order of the roots and the poles is as follows:
\begin{equation}
\begin{cases} 
\zeta_1 < l < 0 < \zeta_2 < \zeta_3 < l + 2 < 2R \qquad &\text{ if }l<0,\\
\zeta_1 < 0 < l < \zeta_2 < \zeta_3 < l + 2 < 2R \qquad &\text{ if } 0 <l<2R-2,\\
\zeta_1 < 0 < l < \zeta_2 < \zeta_3 < 2R<l+2 \qquad &\text{ if }l>2R-2.
\end{cases} 
\label{roots} 
\end{equation}

\subsection{Action, period and rotation number}
We want to define action-angle coordinates for the reduced system. The reduced phase space is given by
$$ -\pi \leq q_2 \leq \pi, \qquad \max\{0,l\} \leq p_2 \leq \min\{l+2,2R\} .$$ 

The first step is to express the action of the reduced system, $\mcI(l,h)$, as a (complete) Abelian integral. However, this is not possible using the standard expression $\frac{1}{2\pi}\oint p\, \dee q$, since it leads to a cubic equation. Therefore, we use the expression $\frac{1}{2\pi}\oint q\, \dee p$, for which integration by parts turns the action into an Abelian integral. The same situation happens in the case of the coupled spin-oscillator in Alonso $\&$ Dullin $\&$ Hohloch \cite{ADH} and the Kovalevskaya top in Dullin $\&$ Richter $\&$ Veselov \cite{DRV}.

More formally, let $\varpi$ be a primitive of the reduced symplectic form $\omega$. We define the action $\mcI=\mcI(l,h)$ of the reduced system by
\begin{equation}
 \mcI(l,h) := \dfrac{1}{2\pi} \oint_{\be_{l,h}} \!\!\varpi = \dfrac{1}{2\pi}\oint_{\be_{l,h}} \!\! q_2\, \dee p_2,
 \label{int1}
\end{equation} where ${\be_{l,h}}$ is the curve defined by $h=\mcH_l(q_2,p_2)$, and thus the integral corresponds to the area enclosed by the curve. The non-scaled version of the action, i.e.\ without applying the scaling \eqref{chan2}, will be $I = R_1 \mcI$.

At this point we must make two remarks. The first one is that since the symplectic form is $\om = -(R_1 \om_{\mbS^2} \oplus R_2 \om_{\mbS^2})$, it has the opposite sign compared to the area form of the spheres and therefore a positive area must correspond to a negative value of the action. The second one is that since the reduced phase space is compact, the closed curve ${\be_{(l,h)}}$ divides it into two parts and therefore we must specify which of the two areas we mean by \eqref{int1}. For later convenience we choose the area that decreases with $h$, or equivalently, we impose that
\begin{equation}
\dfrac{\partial \mcI}{\partial h}(l,h) \geq 0.
\label{derI}
\end{equation} 

\begin{lemm}
The action integral $\mcI(l,h)$ can be expressed as 
\begin{equation}
\mcI(l,h) = 
\begin{cases}
\mfI(l,h) &\text{ if } l<0 \;\;\,\text{ or }\;\; h R > lt\\
\mfI(l,h) + l &\text{ if }l>0 \;\text{ and }\; h R < lt
\end{cases}
\label{i2l}
\end{equation} in a neighbourhood of $(l,h)=(0,0)$, where $\mfI(l,h)$ is the complete elliptic integral
\begin{equation*}
\mfI(l,h):= \dfrac{1}{2\pi} \oint_\beta R_\mcI(p_2) \frac{ \dee p_2}{\sqrt{P(p_2)}},
\end{equation*} defined over the real $\beta$-cycle of the elliptic curve $s^2 =P(p_2)$ and where 
\begin{equation}
\begin{aligned}
R_\mcI(p_2) &= 1-3t+R+\dfrac{t}{R}-l(1-t)+2h+(1-t) p_2 \\
 &+\dfrac{1}{2} \left(\frac{( h - A(l)) l }{ p_2- l} + \frac{(h - A(2 R))2 R}{p_2 - 2 R}+ \frac{ (h - A(l+2))(l + 2) }{p_2 - l - 2} \right).
 \label{RI}
\end{aligned}
\end{equation}
\label{actint}
\end{lemm}

The proof needs a little bit of preparation.

\begin{re} 
In \eqref{i2l}, we need to add $l$, related to the $\mbS^1$-action of the system, in order to make the action $\mcI(l,h)$ continuous at the focus-focus critical value $(0,0)$. This is a consequence of the monodromy of the system (cf.\ Zung in \cite{Z2} and Matveev in \cite{Ma}). Another observation is that the rational terms in \eqref{RI} change sign precisely when $h$ crosses the special values in \eqref{edgeval}.
\end{re}

Before proving Lemma \ref{actint}, let us analyse the phase portrait of $\mcH_l(q_2,p_2)$, displayed in Figure \ref{hami}. If we regard the phase space as a finite cylinder, then there are three types of orbits: 
\begin{itemize}
	\item Type I: Orbits that are homotopic to a point and cross the line $q_2=0$, which are represented by full lines.
	\item Type II: Orbits that are not homotopic to a point, which are represented by dashed lines.
	\item Type III: Orbits that are homotopic to a point and cross the line $q_2=\pi$, which are represented by dotted lines.
\end{itemize}
\begin{figure}[ht]
 \centering
\subfloat[$l<0$]{\includegraphics[width=4.5cm]{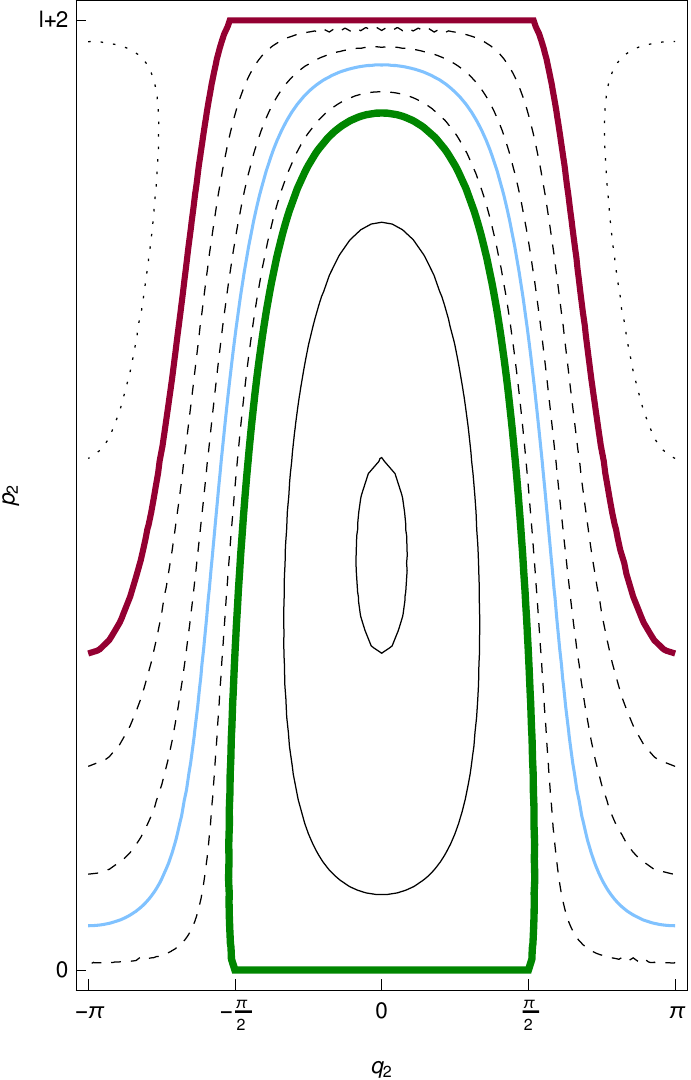} } \hfill
\subfloat[$l=0$]{\includegraphics[width=4.5cm]{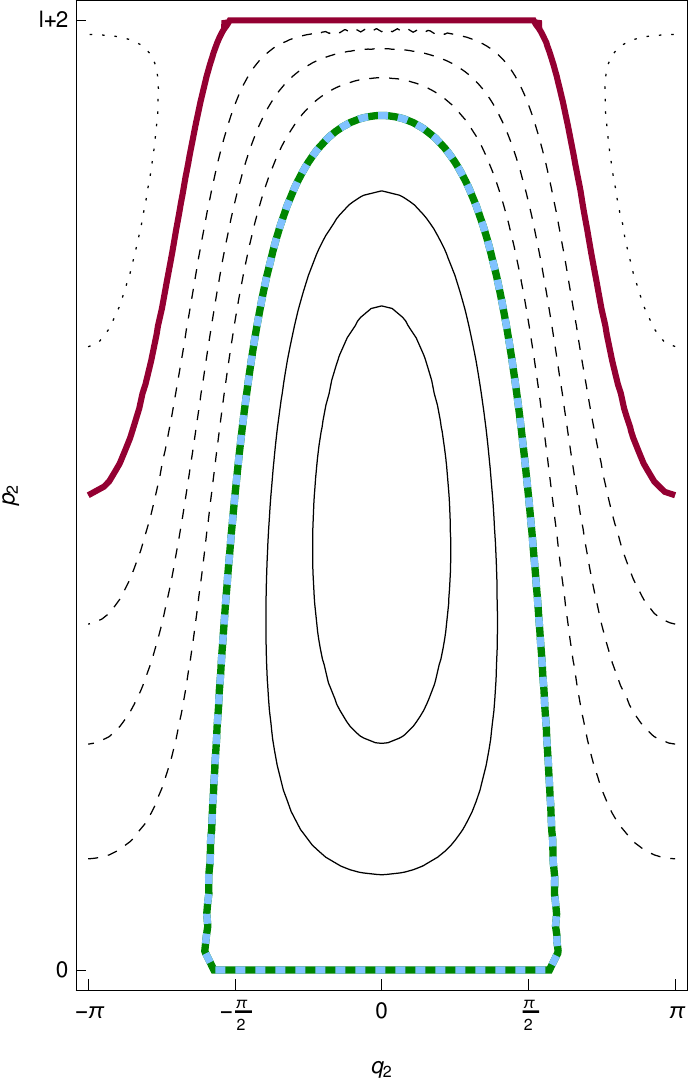}} \hfill
\subfloat[$l>0$]{\includegraphics[width=4.5cm]{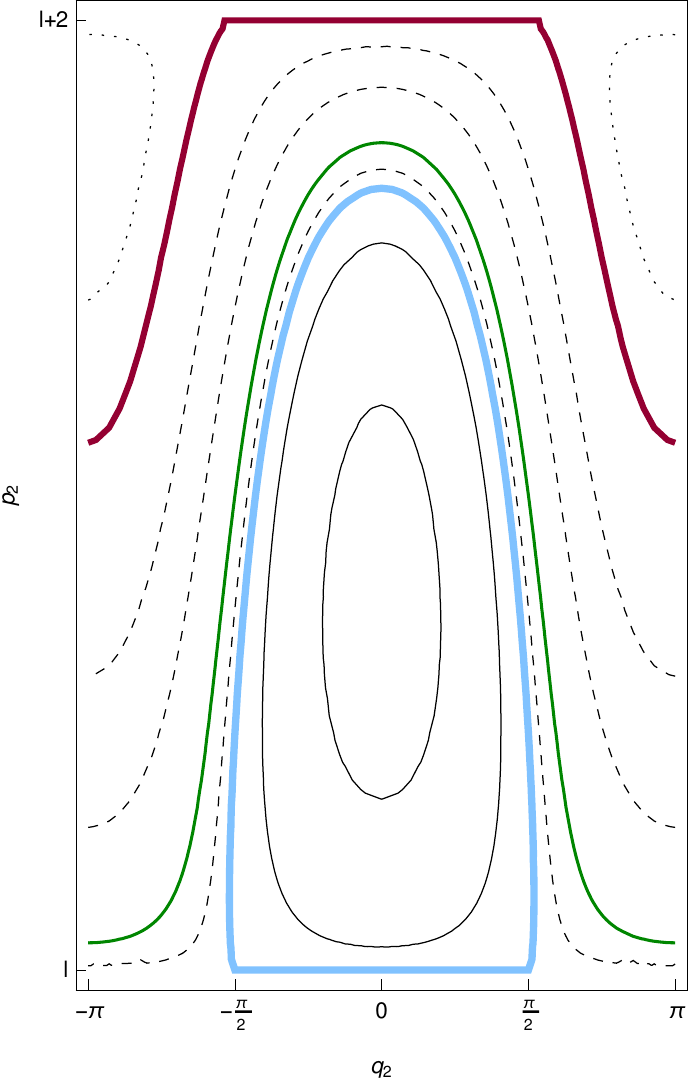}\label{hami03}}
\caption{\small Level sets of the function $\mcH_l(q_2,p_2)$ for different values of $l$ close to $0$. The full lines represent orbits of type I, the dashed lines represent orbits of type II and the dotted lines represent orbits of type III. The level set of $h=A(l)=\frac{t}{R}l$ is in light blue, the level set of $h=A(0)=(1-2t)l$ is in green and the level set of $h=A(l+2)=-2 +4t-\frac{t}{R}(2+l)$ is in dark red. The special orbits that mark the change between orbits of different types are thickened, so $h=A(\max\{0,l\})$ marks the change from type I to type II and $h=A(\min\{l+2,2R\})$ from II to III. For values $l<0$ we see that $h=A(\max\{0,l\})=A(0)$ in green is thickened while $h=A(l)$ in light blue is a normal orbit. For $l=0$, these two orbits coincide. For $l>0$, the orbit $h=A(\max\{0,l\})=A(l)$ in light blue is thickened and $h=A(0)$ in green becomes a normal orbit. A similar crossing of special orbits happens at $l=2R-2$ between $h=A(l+2)$ and $h=A(2R)$.}
 \label{hami}
\end{figure}

The regions containing the different types of orbits are limited by separatrices corresponding to level sets in \eqref{edgeval}. More concretely, the orbits of type I and II are separated by the level set $A(\max\{0,l\})$ and the orbits of type II and III are separated by the level set $A(\min\{l+2,2R\})$, so the changes happen at the values $l=0$ and $l=2R-2$, where the order of the poles \eqref{roots} changes. In Figure \ref{hami} the change at $l=0$ is illustrated. 

\begin{lemm}
Let $\ze_1,\ze_2,\ze_3$ be the roots of $P(p_2)$ as in \eqref{roots}. Then,
$$\dfrac{h-A(\ze_2)}{\sqrt{B(\ze_2)}} = 
\begin{cases}
\;\;\,1 &\text{ if } h>A(\max\{0,l\})\\
-1 &\text{ if } h<A(\max\{0,l\})
\end{cases}$$ and
$$\dfrac{h-A(\ze_3)}{\sqrt{B(\ze_3)}} = 
\begin{cases}
\;\;\,1 &\text{ if } h>A(\min\{l+2,2R\})\\
-1 &\text{ if } h<A(\min\{l+2,2R\}).
\end{cases}$$
\label{aux}
\end{lemm}
\begin{proof}
Consider a curve $\ga_{(l,h)}$ for some $(l,h)$. Since $\mcH_l(q_2,p_2)$ is an even function of $q_2$, it is enough to consider non-negative values of $q_2$. Equation \eqref{red} allows us to write $q_2$ as a function of $p_2$:
\begin{equation}
q_2(p_2; l,h) = \arccos \left( \dfrac{h-A(p_2)}{\sqrt{B(p_2)}}\right).
\label{arcco}
\end{equation} The function $q_2(p_2; l,h)$ is defined for $\ze_2 \leq p_2 \leq \ze_3$, since the value of $p_2$ must be in the physical region and the intersection points of $q_2(p_2; l,h)$ with the lines $q_2=\pm \pi, 0$ correspond to the roots of $P(p_2)$. More precisely, if $q_2(p_2; l,h)=\pm \pi, 0$, it means that 
$$\left( \dfrac{h-A(p_2)}{\sqrt{B(p_2)}}\right)^2=1,$$ which implies
$$P(p_2) = B(p_2)-(h-A(p_2))^2=0,$$
so $p_2$ must be a root of $P(p_2)$. Looking at Figure \ref{hami}, we see that 
$$\dfrac{h-A(\ze_2)}{\sqrt{B(\ze_2)}}=1$$ if the orbit is of the first type, otherwise it will be $-1$. Similarly, 
$$\dfrac{h-A(\ze_3)}{\sqrt{B(\ze_3)}}=1$$ if the orbit is of the first or the second type, otherwise it will be $-1$. Finally, the separatrices that delimit the regions with different types of orbits are given by $h=A(\max\{0,l\})$ and $h=A(\min\{l+2,2R\})$, which gives the desired result. 
\end{proof}

We now prove Lemma \ref{actint}.
\begin{proof}[Proof of Lemma \ref{actint}]
Let us fix some values $(l,h)$ close to $(0,0)$ and consider the curve $\be_{l,h}$ defined by $h=\mcH_l(q_2,p_2)$. Because of the symmetry $\mcH_l(-q_2,p_2)=\mcH_l(q_2,p_2)$ we can restrict ourselves to the non-negative values of $q_2$ and, moreover, multiply by 2. If we look at Figure \ref{hami2}, we need to calculate the area ``to the left" of the curve, since we want condition \eqref{derI} to hold and the grey area decreases with $h$. Using the expression \eqref{arcco}, we write the action integral $\mcI(l,h)$ as
$$\mcI(l,h) = \dfrac{1}{2\pi}\oint_{\be_{l,h}} \!\! q_2\, \dee p_2 = -\left( C_L(l,h) + \dfrac{1}{\pi} \int_{\ze_2}^{\ze_3} \arccos \left(\dfrac{h-A(p_2)}{\sqrt{B(p_2)}} \right) \dee p_2\right),
$$ where $C_L(l,h)$ is the area below the line $p_2=\ze_2$, more precisely
$$ C_L(l,h)= \begin{cases}
0 &\text{ if } h>A(\max\{0,l\})\\
\ze_2 &\text{ if } l<0 \;\text{ and }\; h<l(1-2t)=A(0)\\
\ze_2-l&\text{ if } l>0 \;\text{ and }\; Rh<lt=RA(l).
\end{cases}$$ 
We used the fact that for $(l,h)$ close enough to $(0,0)$ we only encounter orbits of type I and II and therefore 
$$\dfrac{h-A(\ze_3)}{\sqrt{B(\ze_3)}}=1.$$ Integrating by parts we obtain
\begin{align*}
\mcI(l,h) =& - C_L(l,h) - \dfrac{p_2}{\pi} \left. \arccos \left(\dfrac{h-A(p_2)}{\sqrt{B(p_2)}} \right) \right|_{\ze_2}^{\ze_3} +\dfrac{1}{\pi} \int_{\ze_2}^{\ze_3} p_2 \dfrac{d}{dp_2} \left( \arccos \left(\dfrac{h-A(p_2)}{\sqrt{B(p_2)}} \right)\right) \dee p_2 \\=& \,
C_T(l,h) +\dfrac{1}{\pi} \int_{\ze_2}^{\ze_3} R_\mcI(p_2) \frac{ \dee p_2}{\sqrt{P(p_2)}}= C_T(l,h) +\dfrac{1}{2\pi} \oint_\beta R_\mcI(p_2) \frac{ \dee p_2}{\sqrt{P(p_2)}},
\end{align*} where 
\begin{align*}R_\mcI(p_2) &= 1-3t+R+\dfrac{t}{R}-l(1-t)+2h+(1-t) p_2 \\
 &+\dfrac{1}{2} \left(\frac{( h - A(l)) l }{ p_2- l} + \frac{(h - A(2 R))2 R}{p_2 - 2 R}+ \frac{ (h - A(l+2))(l + 2) }{p_2 - l - 2} \right)
\end{align*} and 
\begin{align*}
C_T(l,h) &= -C_L(l,h) - \dfrac{p_2}{\pi} \left. \arccos \left(\dfrac{h-A(p_2)}{\sqrt{B(p_2)}} \right) \right|_{\ze_2}^{\ze_3} \\&= \begin{cases}
0 &\text{ if } h>A(\max\{0,l\})\\
-\ze_2 +\ze_2&\text{ if } l<0 \;\text{ and }\; h<l(1-2t)=A(0)\\
-\ze_2+l+\ze_2 &\text{ if } l>0 \;\text{ and }\; Rh<lt=RA(l)
\end{cases} \\
&= \begin{cases}
0 &\text{ if } l<0\;\; \text{ or }\;\; h >A(\max\{0,l\})\\
l&\text{ if } l>0 \;\text{ and }\; h<l(1-2t)=A(0).
\end{cases}
\end{align*}

\end{proof}
\begin{figure}[ht]
 \centering
\subfloat[$h<A(\max\{0,l\})$]{\includegraphics[width=5cm]{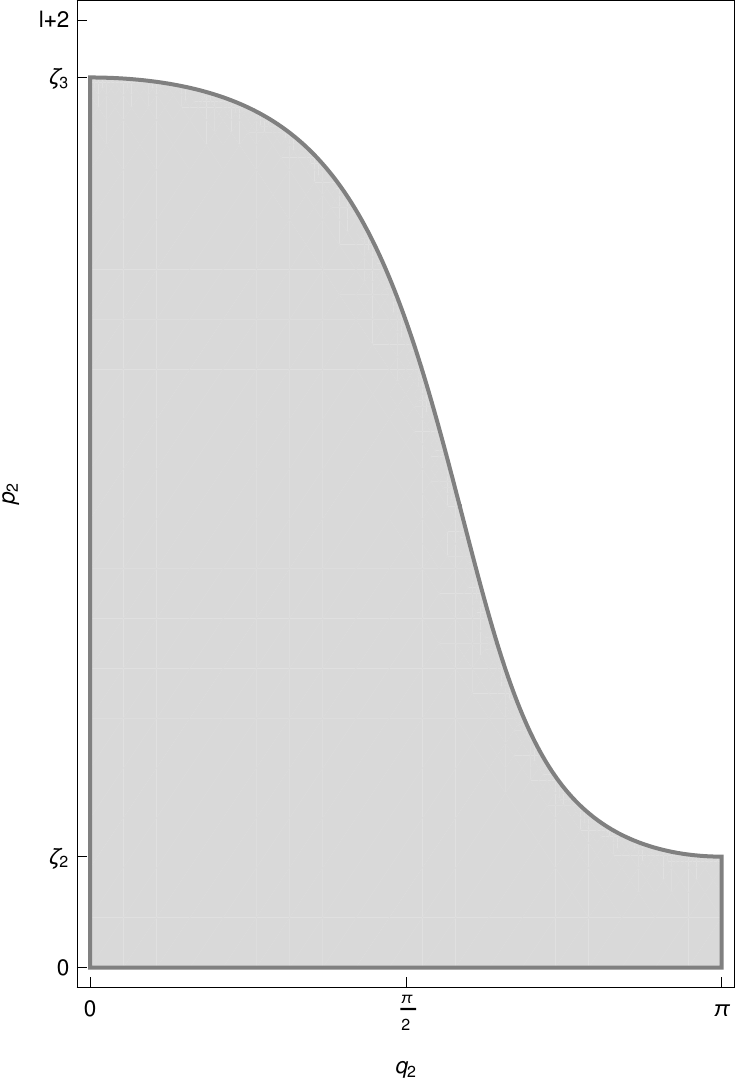}} \hspace{2cm}
\subfloat[$h>A(\max\{0,l\})$]{\includegraphics[width=5cm]{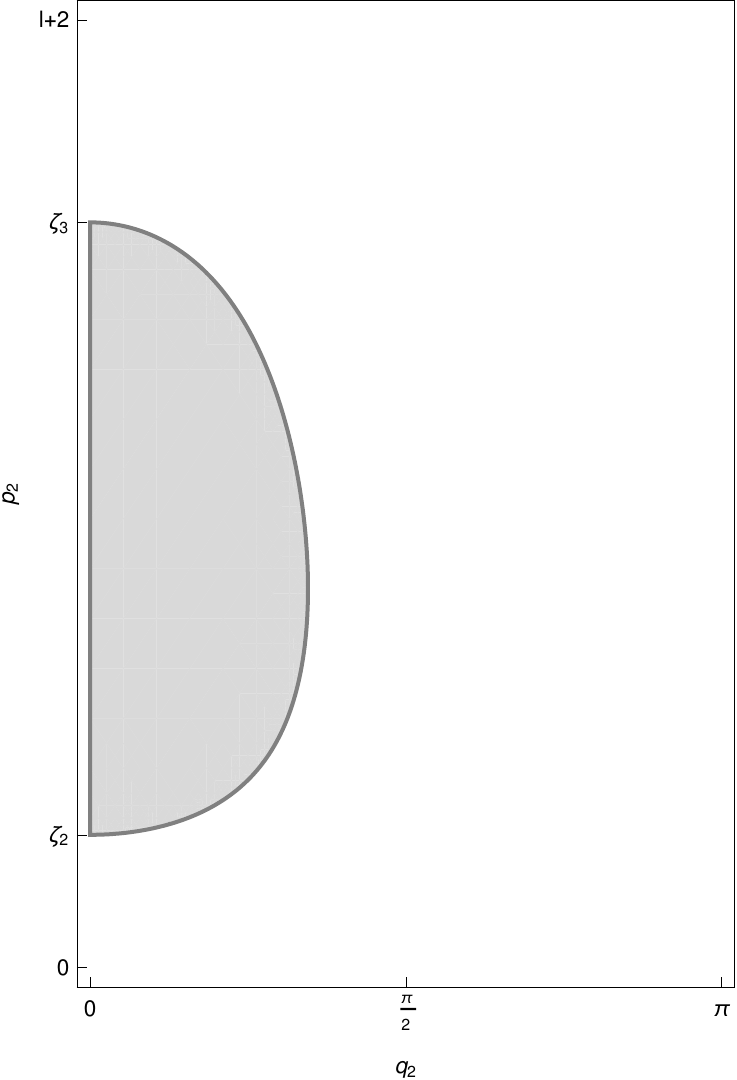}}
\caption{\small Area represented by the action integral $\mcI(l,h)$ for different values of $h$, considering only values $0\leq q_2 \leq \pi$. In the case $h<A(\max\{0,l\})$ we have orbits of type II, i.e.\ $q_2(\ze_2,l,h)=\pi$, so there is an area below the curve, $C_L(l,h)=\ze_2-\max\{0,l\}$. In the case $h>A(\max\{0,l\})$ we have orbits of type I, i.e.\ $q_2(\ze_2,l,h)=0$, so there is no area below the curve, $C_L(l,h)=0$. }
 \label{hami2}
\end{figure}

From now on we will assume that $l>0$ and $hR>lt$. This situation corresponds to the orbits of type I in Figure \ref{hami03}. We are free to make this assumption since the function $S(w)$ from Definition \ref{defS} is smooth, so a particular direction does not affect the Taylor series invariant. Under these conditions, the action integral $\mcI(l,h)=\mfI(l,h)$ is a complete Abelian integral over the real cycle $\beta$ of the elliptic curve $s^2 = P(p_2)$. In a similar fashion, we define the \emph{imaginary action} by integrating the action integral along the vanishing $\alpha$-cycle of the same elliptic curve
\begin{equation}
\mcJ(l,h) := \dfrac{1}{2\pi i} \oint_{\alpha_{l,h}} R_\mcI(p_2) \frac{ \dee p_2}{\sqrt{P(p_2)}}.
\label{imag}
\end{equation} The non-scaled version of the imaginary action, i.e.\ without applying the scaling \eqref{chan2}, will be $J = R_1 \mcJ$.

Other quantities of interest are the period of the reduced system, which we will call \emph{reduced period} $\mcT(l,h)$, and the \emph{rotation number} $\mcW(l,h)$, defined as
\begin{equation}
\mcT(l,h) := 2\pi \dfrac{\partial \mcI}{\partial h},\qquad \mcW(l,h) := -\dfrac{\partial \mcI}{\partial l}.
\label{TW}
\end{equation}

\begin{lemm}
\label{lemTW1}
The reduced period $\mcT(l,h)$ and the rotation number $\mcW(l,h)$ are complete elliptic integrals
\begin{equation}
\mcT(l,h) = \oint_{\beta_{l,h}} \frac{ \dee p_2}{\sqrt{P(p_2)}},\qquad \mcW(l,h) = \dfrac{1}{2\pi} \oint_{\beta_{l,h}} R_\mcW(p_2) \frac{ \dee p_2}{\sqrt{P(p_2)}}
 \label{perrot}
\end{equation}
over the elliptic curve $s^2 = P(p_2)$, where
$$R_\mcW(p_2) = -\dfrac{1}{2} \left( \dfrac{h-A(l)}{p_2-l} + \dfrac{h-A(l+2)}{p_2-l-2} \right).$$
\end{lemm}

This follows immediately from the definition of the reduced period and rotation number.

Similarly to the definition of the imaginary action in \eqref{imag}, we define the \emph{imaginary period} $T^\alpha(l,h)$ and the \emph{imaginary rotation number} $T^\alpha(l,h)$ integrating \eqref{perrot} along the imaginary cycle $\alpha$ of the elliptic curve $s^2 = P(p_2)$, more precisely
\begin{equation}
\mcT^\alpha(l,h) := 2\pi \dfrac{\partial \mcJ}{\partial h},\qquad \mcW^\alpha(l,h) := -\dfrac{\partial \mcJ}{\partial l}.
\label{TWal}
\end{equation}

\subsection{Computation of the Taylor series invariant}
We want to calculate the Taylor series symplectic invariant using \eqref{inv5}. One possibility would be to follow the method used by the authors to calculate the Taylor series invariant for the coupled spin-oscillator in \cite{ADH}. But if we look at the structure of the elliptic integral of the action in Lemma \ref{actint}, we see that it would involve expanding a complete elliptic integral of first kind, another of second kind and three of third kind, which makes calculations quite involved. In this section we will introduce another method based on computing partial derivatives of the symplectic invariant. 

The calculations in this subsection involve series expansions with polynomial coefficients depending on several parameters. In order to simplify notation, we use a short notation for these coefficients throughout this subsection. The full expressions are given in Appendix \ref{apcoef}. All computations are done with \emph{Mathematica}.

So far we have worked with the action, the reduced period and the rotation number as functions of $l$ and $h$. However, the Taylor series invariant is defined as a function of the value of the semi-global function $\Phi=(\Phi_1,\Phi_2)$, which is usually represented as a complex variable $w=w_1 + i w_2$ (see \S \ref{sec:taylor} for details). The components $\Phi_1$ and $\Phi_2$ are the semi-global extensions of $Q_1$ and $Q_2$ respectively. The first component coincides with the 
$\mbS^1$-action, so we can write $w_1 \equiv l$.

In order to calculate the value of the second component we will make use of a Birkhoff normal form for singularities of focus-focus type. As originally observed by Dullin \cite{Du}, this Birkhoff normal form coincides with the imaginary action $\mcJ(l,h)$ defined in \eqref{imag}, so we can write $w_2 \equiv j$. The imaginary action is the integral of the action along the imaginary cycle of the elliptic curve $s^2=P(p_2)$. 
Since the cycle vanishes as we approach the focus-focus singularity, we use the residue theorem to expand the value of $\mcJ(l,h)$. Inverting $j=\mcJ(l,h)$, we will obtain the Birkhoff normal form as $h=\mcB(l,j)$ and this is used to express our functions in terms of $l$ and $j$. The computation of the Birkhoff normal form can be done in the traditional iterative way, as e.g.\ described by Dullin \cite{Du}. However, the present method using residue calculus is more efficient and convenient.

\begin{lemm}\label{BNF}
The Birkhoff normal form of the focus-focus singularity of the coupled angular momenta is
\begin{equation}
\begin{aligned}
\mcB(l,j) =& \dfrac{1}{2 R} ( ( R + t - 2 R t  ) l + r_A j)\varepsilon + \dfrac{t}{8 R \ra^2} (a_{20} l^2 + a_{11} lj + a_{02} j^2)\varepsilon^2  \\
&+ \dfrac{t^2(-1+t)}{8 \ra^6} (a_{30} l^3 + a_{21} l^2j + a_{12}lj^2+a_{03}j^3)\varepsilon^3+\mcO (\varepsilon^4),
\label{polB}
\end{aligned}
\end{equation}
where $\ra$ is the determinant square root from Definition \ref{sqr} and $a_{mn}$ are polynomial coefficients in $\ra$, $R$, $t$ listed in Appendix \ref{apcoef}.
\end{lemm}

In Lemma \ref{BNF} and also in the following, we use the following convention for the indices of polynomial coefficients: the index $m$ refers to the degree of the first variable, the index $n$ refers to the degree of the second variable and the index $k$ refers to the third variable, if present.

\begin{proof}
 We have an integral along the imaginary cycle, which vanishes as we approach the singular value. Therefore we can use the residue theorem of complex analysis. More specifically, let us make the substitution $l \mapsto l \varepsilon$, $h \mapsto h \varepsilon$ in the integrand of \eqref{imag} and make a Taylor expansion in $\varepsilon$, which gives
 \begin{equation}
 \begin{aligned}
 \label{polRI}
 \dfrac{2 R_\mcI(p_2)}{\sqrt{P(p_2)}} &= \dfrac{b_{00}}{(p_2-2)(p_2-2R)\rb} + \dfrac{1}{p_2(p_2-2)^2\rb^3}(b_{10} l + b_{01} h)\varepsilon  \\&+ \dfrac{1}{{p_2}^2 (p_2-2)(p_2-2R)\rb^5} (b_{20} l^2 + b_{11} l h + b_{02} h^2)\varepsilon^2+\mcO (\varepsilon^3),
 \end{aligned}
 \end{equation} where $$ \rb := \sqrt{-R^2 (1 - 2 t)^2 + 2 R (1 + p_2 (-1 + t)) t - t^2}$$ and $b_{mn}$ are polynomial coefficients in $p_2$, $\rb$, $R$, $t$ listed in Appendix \ref{apcoef}. By computing now the residue of \eqref{polRI} at the pole $p_2=0$, we get the expansion of $\mcJ(l,h)$:
 \begin{equation}
 \begin{aligned}
 %
 %
\mcJ(l,h) &= \dfrac{1}{\ra} ( -(R + t - 2 R t) l + 2 R  h)\varepsilon + \dfrac{1}{\ra^5} (c_{20} l^2 + c_{11} lh + c_{02} h^2)\varepsilon^2  \\
&+ \dfrac{1}{\ra^9} (c_{30} l^3 + c_{21} l^2h + c_{12}lh^2+c_{03}h^3)\varepsilon^3+\mcO (\varepsilon^4),
\label{polJ}
\end{aligned}
\end{equation} where $\ra = \rb_{|_{p_2 =0}}$ is as in Definition \ref{sqr} and $c_{mn}$ are polynomial coefficients in $R$ and $t$, also listed in Appendix \ref{apcoef}. We now write $j=\mcJ(l,h)$ and invert the expansion \eqref{polJ} considering $l$ as constant. This way we obtain $h = \mcB(l,j)$ with the expansion \eqref{polB}.
\end{proof}

We note that the coefficients of $j$ and $l$ in the Birkhoff normal form are the real and imaginary part of the eigenvalues 
of the linearised Hamiltonian vector field at the focus-focus point, respectively.

As already mentioned, we now derive equations that express the derivatives of the symplectic invariant in terms of the reduced period and rotation number integrals over real and imaginary cycles.
Since we are working with scaled variables, cf.\ equation \eqref{chan2}, we will refer to the symplectic invariant in scaled coordinates by $\mcS$ and in the original coordinates by $S$.
\begin{theo}
\label{theder}
Let $\mcT(l,h)$ be the reduced period and $\mcW(l,h)$ the rotation number of the coupled angular momenta. Let also $\mcT^\alpha(l,h)$ and $\mcW^\alpha(l,h)$ be their imaginary counterparts. Then the partial derivatives of the (scaled) Taylor series symplectic invariant $\mcS(w)=\mcS(l,j)$ are given by
\begin{equation*}
\begin{aligned}
\frac{\partial \mcS}{\partial l} &= 2\pi \left.\left( \mcW^\alpha (l,h) \frac{\mcT(l,h)}{\mcT^\alpha(l,h)} - \mcW(l,h) \right)\right|_{h=\mcB(l,j)} \hspace{-0.2cm}+ \arg(w), \\[0.2cm] \frac{\partial \mcS}{\partial j} &= 2\pi \left.\frac{\mcT(l,h)}{\mcT^\alpha(l,h)}\right|_{h=\mcB(l,j)}\hspace{-0.2cm}+\ln |w|,
\end{aligned}
\end{equation*} where $w=l+ij$ and $\mcB(l,j)$ is the Birkhoff normal form.
\end{theo}
\begin{proof}
The Birkhoff normal form $\mcB(l,j)$ allows us to relate the action $\mcI (l,h)$ defined in \eqref{int1} with the area integral $\mcA(w)$ defined in \S \ref{sec:taylor}, since $w = w_1 + i w_2 = l + ij$. Thus,
\begin{equation}
\mcA(w) = 2\pi \mcI (l,\mcB(l,j)),
\label{AI}
\end{equation} or in other words, $\Phi = (\Phi_1,\Phi_2)$ with $\Phi_1=L$ and $\Phi_2=J\circ F$. Furthermore, if we denote the singular part of the area integral by $\mcA_{sing}(z) := -\imp(w \ln w - w) = -j \ln|w| -l \arg(w) +j$, where $\arg(w)$ takes values between $-\frac{3\pi}{2}$ and $\frac{\pi}{2}$, we can express the Taylor series symplectic invariant as 
$$ \mcS(w) = \mcA(w) - \mcA(0) -\mcA_{sing}(w).$$

The partial derivatives of the singular part of the action are
$$ \dfrac{\partial \mcA_{sing}}{\partial l} = -\arg(w),\qquad \dfrac{\partial \mcA_{sing}}{\partial j} = -\ln|w|
.$$ 

We  now find the partial derivatives of the Taylor series symplectic invariant $S(w)$. We start by
$$
\dfrac{\partial \mcS}{\partial j} = \dfrac{\partial \mcA}{\partial j} - \dfrac{\partial \mcA_{sing}}{\partial j} = 2\pi \left. \dfrac{\partial \mcI}{\partial h} \right|_{h = \mcB(l,j)}\hspace{-0.1cm} \dfrac{\partial \mcB}{\partial j} + \ln|w|. $$
If we differentiate the expression $h = B(l,\mcJ(l,h))$ with respect to $h$ we obtain
$$\dfrac{\partial \mcB}{\partial j} = \left( \dfrac{\partial \mcJ}{\partial h} \right)^{-1}$$
and thus
$$\dfrac{\partial \mcS}{\partial j}=\mcT(l,\mcB(l,j)) \left(\dfrac{\partial \mcJ}{\partial h}\right)^{-1}_{|_{h = \mcB(l,j)}} \hspace{-0.2cm}+ \ln|w| = 2\pi \left.\dfrac{\mcT(l,h)}{\mcT^\alpha(l,h)}\right|_{h = \mcB(l,j)}\hspace{-0.2cm} + \ln|w|.
$$
Similarly,
$$
\dfrac{\partial \mcS}{\partial l} = \dfrac{\partial \mcA}{\partial l} - \dfrac{\partial \mcA_{sing}}{\partial l} = 2\pi \left. \dfrac{\partial \mcI}{\partial l} \right|_{h = \mcB(l,j)}\hspace{-0.2cm} + 2\pi \left.\dfrac{\partial\mcI}{ \partial h}\right|_{h = \mcB(l,j)}\hspace{-0.1cm}\dfrac{\partial \mcB}{\partial l} + \arg(w). $$ 
Differentiating $h = B(l,\mcJ(l,h))$ with respect to $l$ we obtain
$$ 0 = \dfrac{\partial \mcB}{\partial j} \dfrac{\partial \mcJ}{\partial l} + \dfrac{\partial \mcB}{\partial l} = \left( \dfrac{\partial \mcJ}{\partial h} \right)^{-1}\dfrac{\partial \mcJ}{\partial l} + \dfrac{\partial \mcB}{\partial l} $$ and thus 
\begin{align*}
\dfrac{\partial \mcS}{\partial l} &= -2\pi \mcW(l,\mcB(l,j)) - \mcT(l,\mcB(l,j))\left(\dfrac{\partial \mcJ}{\partial h}\right)^{-1} \dfrac{\partial \mcJ}{\partial l}+ \arg(w)\\&= -2\pi \mcW(l,\mcB(l,j)) + 2\pi \mcT(l,\mcB(l,j))\left.\dfrac{\mcW^\alpha(l,h)}{\mcT^\alpha(l,h)}\right|_{h = \mcB(l,j)} \hspace{-0.2cm}+ \arg(w)\\&= 2\pi \left.\left( \mcW^\alpha (l,h) \frac{\mcT(l,h)}{\mcT^\alpha(l,h)} - \mcW(l,h) \right)\right|_{h=\mcB(l,j)} \hspace{-0.2cm}+ \arg(w),
\end{align*}
which is what we wanted to prove.
\end{proof}

As a consequence of this theorem, in order to find the Taylor series invariant it is enough to calculate the expansions of $\mcT$, $\mcW$, $\mcT^\alpha$ and $\mcW^\alpha$. The last two are obtained immediately from the definitions \eqref{TWal} by differentiating the Taylor expansion of $\mcJ$ in \eqref{polJ}. We find
\begin{equation}
\begin{aligned}
\mcT^\al(l,h) &= 2\pi \dfrac{\partial \mcJ}{\partial h} = 2\pi\dfrac{2R}{\ra} + \dfrac{2\pi}{\ra^5}(c_{11} l + 2c_{02} h)\varepsilon \\&+ \dfrac{2\pi}{\ra^9} (c_{21} l^2 + 2c_{12} l h + 3c_{03} h^2)\varepsilon^2 + \mcO (\varepsilon^3)
\label{expTa}
\end{aligned}
\end{equation} and 
\begin{equation}
\begin{aligned}
\mcW^\al(l,h) &= - \dfrac{\partial \mcJ}{\partial l} = - \dfrac{R + t - 2 R t}{\ra} - \dfrac{1}{\ra^5} (2c_{20} l + c_{11}h)\varepsilon \\&- \dfrac{1}{\ra^9} (3c_{30}l^2 + 2c_{21}lh + c_{12}h^2)\varepsilon^2 + \mcO (\varepsilon^3),
\label{expWa}
\end{aligned}
\end{equation} where $\ra$ is as in Definition \ref{sqr} and $c_{mn}$ are polynomial coefficients in $R$, $t$ listed in Appendix \ref{apcoef}. As noted in Def.~\ref{sqr} the real part of the eigenvalue of the focus-focus point is $\frac{r_A}{2 R}$, which constitutes the leading-oder term of the imaginary period. Similarly, the leading order terms of the imaginary rotation number is the ratio of the imaginary and real part of the eigenvalue of the focus-focus point, as proved in the general case in Dullin $\&$ \vungoc\!\! \cite{DV}.

Now we need to calculate the expansions of $\mcT$ and $\mcW$. The first step is to rewrite the integral expressions of $\mcT$ and $\mcW$ from Lemma \ref{lemTW1} in terms of Legendre canonical forms. For the reduced period we write 
$$\mcT = 2\pi \dfrac{\partial \mcI}{\partial h} = \oint_\beta \dfrac{\dee p_2}{\sqrt{P(p_2)}} = 2 \int_{\ze_2}^{\ze_3} \dfrac{\dee p_2}{s},$$ 
where $s^2 =P(p_2)$ is the elliptic curve on which the integral takes place. Using the change of variable 
\begin{equation}
 x := \sqrt{\dfrac{\ze_3-p_2}{\ze_3-\ze_2}}
 \label{chan}
\end{equation}
suggested for example in Abramowitz $\&$ Stegun \cite{Ab1} or Byrd $\&$ Friedman \cite{BF}, we transform as follows:
$$ \mcN_A := \int_{\ze_2}^{\ze_3} \dfrac{\dee p_2}{s} = \dfrac{\sqrt{2}}{\sqrt{\ze_3-\ze_1}} \int_0^1 \dfrac{\dee x}{\sqrt{(1-x^2)(1-k^2x^2)}} = \dfrac{\sqrt{2}}{\sqrt{\ze_3-\ze_1}} K(k),$$ where $k$ is the elliptic modulus, given by 
\begin{equation}
k^2 = \dfrac{\ze_3-\ze_2}{\ze_3-\ze_1}.
\label{kk}
\end{equation} and $K(k)$ is the complete elliptic integral of first kind defined in Appendix \ref{sec:ellipticInt}.

Thus, we can write 
\begin{equation}
 \mcT = 2 \mcN_A = \dfrac{2\sqrt{2}}{\sqrt{\ze_3-\ze_1}} K(k).
 \label{expTT}
\end{equation}
 Similarly, for the rotation number we have
\begin{align*}
\mcW &= -\dfrac{\partial \mcI}{\partial l} = \dfrac{1}{2\pi} \oint_\beta R_\mcW (p_2) \dfrac{\dee p_2}{\sqrt{P(p_2)}} \\&
= -\dfrac{h+A(l)}{2\pi} \int_{\ze_2}^{\ze_3} \dfrac{1}{p_2-l }\dfrac{\dee p_2}{s} - \dfrac{h+A(l+2)}{2\pi} \int_{\ze_2}^{\ze_3} \dfrac{1}{p_2-l-2 } \dfrac{\dee p_2}{s},
\end{align*}
where $A(p_2)$ is defined in \eqref{ABdef}. Using again the variable change \eqref{chan} we transform integrals of the form 
$$ \mcN_{B,\ga} = \int_{\ze_2}^{\ze_3} \dfrac{1}{p_2-\ga} \dfrac{\dee p_2}{s},$$
where $\ga$ is a constant, into
$$ \mcN_{B,\ga} = \dfrac{\sqrt{2}}{(\ze_3-\ga)\sqrt{\ze_3-\ze_1}} \int_0^1 \dfrac{\dee x}{(1-n_\ga x^2)\sqrt{(1-x^2)(1-k^2x^2)}} = \dfrac{\sqrt{2}}{(\ze_3-\ga)\sqrt{\ze_3-\ze_1}} \Pi(n_\ga,k),$$
where $n_\ga$ is the characteristic, given by
$$n_\ga = \dfrac{\ze_3-\ze_2}{\ze_3 - \ga}$$
and $\Pi(n_\ga,k)$ is the complete elliptic integral of third kind defined in Appendix \ref{sec:ellipticInt}. 

For the rotation number we thus obtain
\begin{align}
\mcW  \nonumber
= &-\dfrac{h+A(l)}{2\pi} \mcN_{B,l} - \dfrac{h+A(l+2)}{2\pi} \mcN_{B,l+2} \\ \label{expWW}
=&- \frac{h+A(l)}{\sqrt{2}\pi} \frac{1}{(\ze_3-l)\sqrt{\ze_3-\ze_1}} \Pi(n_l,k) \\&- \frac{h+A(l+2)}{\sqrt{2}\pi} \frac{1}{(\ze_3-l-2)\sqrt{\ze_3-\ze_1}} \Pi(n_{l+2},k) \nonumber
\end{align}
 where $k$ is again as in \eqref{kk} and the characteristics $n_l$ and $n_{l+2}$ are given by
\begin{equation}
n_l = \dfrac{\ze_3-\ze_2}{\ze_3 - l},\qquad n_{l+2} = \dfrac{\ze_3-\ze_2}{\ze_3-l-2}.
\label{chara}
\end{equation}

We recall that the roots have the ordering given in \eqref{roots}. This means that the characteristics $n_l$ and $n_{l+2}$ in \eqref{chara} belong to the circular case introduced in Appendix \ref{sec:ellipticInt}. More precisely, $n_l$ belongs to the positive circular case $(k^2<n<1)$ and $n_{l+2}$ belongs to the negative circular case $(n<0)$, so in both cases we can express the complete elliptic integral $\Pi(n_\ga,k)$ in terms of Heuman's lambda function defined in \eqref{lamlam}. Using now the expansions \eqref{expK1},\eqref{expL1} we are able to expand the reduced period $\mcT$ and the rotation number $\mcW$ in terms of the variables $l,h$, the elliptic modulus $k$ and the characteristics $n_l,n_{l+2}$.

But the elliptic modulus and the characteristics depend only on the roots of the polynomial $P(p_2)$, as we see in \eqref{kk} and \eqref{chara}, which are also functions of $h$ and $l$. 

\begin{lemm}
\label{polRoots}
The Taylor expansion of the roots $\ze_1, \ze_2, \ze_3$ of $P(p_2)$ around $(l,h)=(0,0)$ are 
\begin{equation}
\begin{aligned}
\ze_1(l,h) &= \dfrac{1}{\ra^2}(d_{100}l + d_{010}h + d_{001}\rg) \varepsilon + \dfrac{R}{2 \ra^6 \rg} (d_{300}l^3 +d_{210}l^2 h +d_{201}l^2 \rg \hspace{0.6cm} \\&+d_{120}l h^2 +d_{111}l h \rg +d_{030}h^3 +d_{021}h^2 \rg )\varepsilon^2+ \mcO (\varepsilon^3),
\label{polRoot1}
\end{aligned}
\end{equation}
\begin{equation}
\begin{aligned}
\ze_2(l,h) &= \dfrac{1}{\ra^2}(e_{100}l + e_{010}h + e_{001}\rg) \varepsilon + \dfrac{R}{2 \ra^6 \rg} (e_{300}l^3 +e_{210}l^2 h +e_{201}l^2 \rg \hspace{0.6cm} \\&+e_{120}l h^2 +e_{111}l h \rg +e_{030}h^3 +e_{021}h^2 \rg )\varepsilon^2+ \mcO (\varepsilon^3),
\label{polRoot2}
\end{aligned}
\end{equation}
and
\begin{equation}
\begin{aligned}
\ze_3(l,h) &=\dfrac{\ra^2}{2R t(1-t)} + \dfrac{R-t}{(1-t)\ra^2}(f_{10}l + f_{01}h)\varepsilon + \dfrac{2Rt}{\ra^6} (f_{20}l^2 + f_{11} l h + f_{02} h^2)\varepsilon^2 \\& + \dfrac{4R^2 t^2}{\ra^{10}} (f_{30} l^3 + f_{21} l^2 h + f_{12} l h^2 + f_{03}h^3)\varepsilon^3 + \mcO (\varepsilon^4),
\label{polRoot3}
\end{aligned}
\end{equation}
where 
\begin{equation}
\rg := \sqrt{ l^2 (1 - t) t + lh (R + t - 2 R t)+h^2 R}
\label{rrg}
\end{equation} and $d_{mnk}, e_{mnk}, f_{mnk}$ are polynomial coefficients in $R$, $t$ listed in Appendix \ref{apcoef}. The third subscript denotes powers of $r_C$ and only takes values 0 or 1.
\end{lemm}
\begin{proof}
The polynomial $P(p_2)$ defined in \eqref{Pdef} is given by
\begin{align*} P(p_2) = \dfrac{1}{R^2}&\left( t^2 p_2(p_2-l)(p_2-2R)(p_2-l-2) - (h R + p_2 (R + (-1 + p_2) t - 2 R t) \right.\\& \left.- l (R + p_2 t - 2 R t))^2 \right).
\end{align*} We find the Taylor expansion $p_2 = {g}_0 + {g}_1 \varepsilon + {g}_2 \varepsilon^2 + {g}_3 \varepsilon^3 + \mcO(\varepsilon^4)$ of the roots of $P(p_2)$ iteratively: we first make the substitution $l \mapsto l \varepsilon$ and $h \mapsto h \varepsilon$ and expand in $\varepsilon$. After that we solve order by order. 

For order 0 we get the equation
$$ \dfrac{{g_0}^2 (-R^2 (1 - 2 t)^2 + 2 R (1 + g_0 (-1 + t)) t - t^2)}{R^2} =0,$$ which has as solutions $g_0=0$ with multiplicity 2, corresponding to the constant terms of the roots $\ze_1,\ze_2$ and 
$$ g_0 = \dfrac{-R^2 (1 - 2 t)^2 + 2 R t - t^2}{2 R (1 - t) t} = \dfrac{\ra^2}{2 R (1 - t) t} $$ with multiplicity 1, corresponding to the constant term of the root $\ze_3$. 

Now we move on to compute the next two orders. For order 1 we have
$$ g_1=\dfrac{R (h (R + (-1 + g_0) t - 2 R t) + 
 l (R (1 - 2 t)^2 + t (-1 + 2 g_0 - 2 g_0 t)))}{
R^2 (1 - 2 t)^2 + R (-2 - 3 g_0 (-1 + t)) t + t^2}, $$ 
and for order 2
\begin{align*}
g_2 =& -\dfrac{( 2 l {g_1} R (-R (1 - 2 t)^2 + (1 + 4 {g_0} (-1 + t)) t) + 
 l^2 R (R (1 - 2 t)^2 - 2 {g_0} (-1 + t) t)}{2 {g_0} (R^2 (1 - 2 t)^2 + R (-2 - 3 {g_0} (-1 + t)) t + t^2)}\\
 		& -\dfrac{h^2 R^2 +{g_1}^2 (R^2 (1 - 2 t)^2 - 2 R (1 + 3 {g_0} (-1 + t)) t + t^2)}{2 {g_0} (R^2 (1 - 2 t)^2 + R (-2 - 3 {g_0} (-1 + t)) t + t^2)} \\
 		&-\dfrac{2 h R (l (R + {g_0} t - 2 R t) + {g_1} (-R + t - 2 {g_0} t + 2 R t))}{2 {g_0} (R^2 (1 - 2 t)^2 + R (-2 - 3 {g_0} (-1 + t)) t + t^2)}.
\end{align*}

By substituting each term into the following one, we obtain by iteration the expansion \eqref{polRoot3} of the root $\ze_3$. However, for $\ze_1$, $\ze_2$ to be well-defined, since $g_0=0$, the only solution is that the numerator of $g_2$ vanishes identically, which has as solutions
$$g_1 = \dfrac{R (-l R + l t + 4 l R t - 4 l R t^2 + h (t + R (-1 + 2 t)) \pm 
 2 t \rg )}{\ra^2},$$ where $\rg = \sqrt{ l^2 (1 - t) t + lh (R + t - 2 R t)+h^2 R}$. If we take the solution with positive sign as the second order term of $\ze_2$ and the one with negative sign as the second order term of $\ze_1$, we can find the rest of terms of the expansions \eqref{polRoot1} and \eqref{polRoot2} by iteration. 
\end{proof}

By combining these results we now compute the expansion of the reduced period. Notice that 
the coefficient of $-\ln|w|$ is the inverse of the real part of the eigenvalue of the focus-focus point.

\begin{lemm}
\label{lemT}
The expansion of the reduced period is
\begin{align}
 \mcT(l,j) =& \left( \dfrac{1}{4\ra^5}(h_{10}l + h_{01}j)\varepsilon 
 +\dfrac{1}{32 R \ra^8} (h_{20}l^2 + h_{11}lj + h_{02}j^2 )\varepsilon^2
 +\mcO(\varepsilon^3) \right) \nonumber \\&+ \left( \ln|w| - \ln \frac{4 \ra^3}{R^{3/2}(1-t)t^2} \right) \left( -\dfrac{2 R}{\ra}+\dfrac{R t}{\ra^4}(h_{L10}l + h_{L01}j)\varepsilon \right. \nonumber \\
 & \hspace{6cm}+ \left.\dfrac{R t^2}{2\ra^9}(h_{L20} l^2 + h_{L11}lj + h_{L02}j^2)\varepsilon^2 + \mcO(\varepsilon^3)
 \right), \label{PolT}
\end{align} where $h_{mn}$, $h_{Lmn}$ are polynomial coefficients in $\ra$, $R$, $t$ listed in Appendix \ref{apcoef}.
\end{lemm}
\begin{proof}
In \eqref{expTT} we expressed the period $\mcT$ in terms of the Legendre canonical elliptic integral of first kind, whose Taylor expansion is given in \eqref{expK1}. The expansion depends on the elliptic modulus $k$, which we expressed in \eqref{kk} in terms of the roots $\ze_1, \ze_2,\ze_3$ of the polynomial $P(p_2)$. In Lemma \ref{polRoots} we have expressed these roots in terms of $h$, $l$ and $\rg$.

We want our result as a function of $l$ and $j$, so we substitute for $h$ the Birkhoff normal form $\mcB(l,j)$ found in Lemma \ref{BNF} in order to eliminate the dependence on $h$. Similarly, we also replace $h$ by $\mcB(l,j)$ in the definition of $\rg$ in \eqref{rrg} and expand around the origin to obtain the Taylor series of $\rg$ 
\begin{equation}
\begin{aligned}
\rg(l,j) &= \dfrac{\ra}{2\sqrt{R}} \sqrt{l^2+j^2} \varepsilon+ \dfrac{t j}{8 \sqrt{R} \ra^2 \sqrt{l^2+j^2}}(u_{20} l^2 + u_{11} lj+ u_{02} j^2)\varepsilon^2 \\&+ \dfrac{t^2}{64 \sqrt{R} \ra^5 (l^2+j^2)^{3/2}}(u_{60} l^6 + u_{51} l^5 j + u_{42} l^4j^2 + u_{33}l^3j^3 + u_{24} l^2 j^4 \\&+ u_{15} l j^5 + u_{06} j^6) \varepsilon^3+\mcO (\varepsilon^4),
\label{polRBljser}
\end{aligned}
\end{equation} where $u_{mn}$ are polynomial coefficients in $R$, $t$, $\ra$ listed in Appendix \ref{apcoef}. Substituting $h$ by $\mcB(l,j)$ and $\rg$ by \eqref{polRBljser} we obtain the expansion of the roots $\ze_1,\ze_2,\ze_3$ as a function of $j$ and $l$:
\begin{equation}
\begin{aligned}
\ze_1(l,j) =& \dfrac{1}{2\ra}(\al_{100}l + \al_{010}j + \al_{001}\sqrt{l^2+j^2})\varepsilon \\&
+ \dfrac{1}{8 \ra^4\sqrt{R} \sqrt{l^2+j^2}} (\al_{300}l^3  +\al_{210}l^2 j+ \al_{120}lj^2+\al_{030}j^3 \\& + (\al_{201}l^2+\al_{111}lj+\al_{021}j^2)\sqrt{l^2+j^2} )\varepsilon^2+ \mcO (\varepsilon^3),
\end{aligned}
\label{pol2Root1}
\end{equation}
\begin{equation}
\begin{aligned}
\ze_2(l,j) =& \dfrac{1}{2\ra}(\be_{100}l + \be_{010}j + \be_{001}\sqrt{l^2+j^2})\varepsilon \\&+ \dfrac{1}{8 \ra^4\sqrt{R} \sqrt{l^2+j^2}} (\be_{300}l^3  +\be_{210}l^2 j+ \be_{120}lj^2+\be_{030}j^3 \\&+ (\be_{201}l^2+\be_{111}lj+\be_{021}j^2)\sqrt{l^2+j^2} )\varepsilon^2+ \mcO (\varepsilon^3),
\end{aligned}
\label{pol2Root2}
\end{equation} and 
\begin{equation}
\begin{aligned}
\ze_3(l,j) &=\dfrac{\ra^2}{2R t(1-t)} + \dfrac{1}{2R\ra(-1+t)}(\ga_{10}l + \ga_{01}j)\varepsilon \\&+ \dfrac{1}{8 R \ra^4 (-1 + t) } (\ga_{20}l^2 + \ga_{11} lj + \ga_{02} j^2)\varepsilon^2 \\& + \dfrac{t^2}{8 \ra^8
 } (\ga_{30} l^3 + \ga_{21} l^2 j + \ga_{12} l j^2 + \ga_{03}j^3)\varepsilon^3 + \mcO (\varepsilon^4),
\end{aligned}
\label{pol2Root3}
\end{equation} where the coefficients $\al_{mnk},\be_{mnk},\ga_{mn}$ are polynomials in $R,t,\ra$ listed in Appendix \ref{apcoef}. Using now \eqref{kk} we express the elliptic modulus $k$ as a series 
\begin{align}
\nonumber
k^2(l,j) = 1 &+ \dfrac{4 R^{3/2} (-1 + t) t^2}{\ra^3}\sqrt{l^2+j^2}\varepsilon + \dfrac{\sqrt{R}(1-t)t^2}{2\ra^9 \sqrt{l^2+j^2} }
\left( \delta_{300} l^3 + \delta_{210}l^2j \right.\\&+\left. \delta_{120}lj^2 + \delta_{030}j^3 + (\delta_{201}l^2 + \delta_{111}lj + \delta_{021}j^2)\sqrt{l^2+j^2} \right)\varepsilon^2+\mcO (\varepsilon^3),
\label{polKjl}
\end{align}
where the coefficients $\delta_{mnk}$ are polynomials in $R,t,\ra$ listed in Appendix \ref{apcoef}. Using now \eqref{expTT} and \eqref{expK1} we obtain the Taylor series of the period $\mcT$ as a function only of $l$ and $j$, presented in \eqref{PolT}. 
\end{proof}

It is interesting to note that the deviation of the modulus $k^2$ from 1 to leading order depends only on the modulus $|w|=| l + i j|$.
Notice that in the next statement the coefficient of $-\ln|w|$ is the ratio of the real and imaginary part 
of the eigenvalue of the focus-focus point.

\begin{lemm}
\label{lemW}
The expansion of the rotation number is:
\begin{equation}
\begin{aligned}
 2\pi\mcW(l,j) =& \pi - \arctan \left( \dfrac{t - R (1 + t) - R^2 (1 - 2 t)}{(1 - R) \ra} \right)+\arg(w) \\&+ \dfrac{1}{4\ra^5} (v_{10}l + v_{01}j)\varepsilon
 +\mcO(\varepsilon^2)  \\&+ \left( \ln|w| - \ln \frac{4 \ra^3}{R^{3/2}(1-t)t^2} \right)\\& \hspace{1cm}\left(-\dfrac{R+t-2Rt}{\ra}+\dfrac{R(-1+t)t}{\ra^4}(v_{L10}l + v_{L01}j)\varepsilon + \mcO(\varepsilon^2)
 \right), \label{PolW}
\end{aligned}
\end{equation} where  $v_{mn}$, $v_{Lmn}$ are polynomial coefficients in $\ra$, $R$, $t$ listed in Appendix \ref{apcoef}.
\end{lemm}
\begin{proof}
We repeat the procedure of the calculation of the Taylor series of the reduced period that we used in Lemma \ref{lemT}. In \eqref{expWW} we find the expression of the rotation number $\mcW$ in terms of two Legendre canonical elliptic integral of third kind. We substitute these elliptic integrals by Heuman's Lambda functions as described in Appendix \ref{sec:ellipticInt}, whose Taylor expansion is given in \eqref{expL1}. The expansion depends on the elliptic modulus $k$ and the characteristics $n_l,n_{l+2}$, which we expressed in terms of the roots $\ze_1, \ze_2,\ze_3$ of the polynomial $P(p_2)$ in \eqref{kk} and \eqref{chara} respectively. In Lemma \ref{polRoots} we have expressed these roots in terms of $h$, $l$ and $\rg$.

Since we want the expansion of $\mcW$ as a function of $l$ and $j$, we substitute $h$ by the Birkhoff normal form $\mcB(l,j)$ found in Lemma \ref{BNF} and $\rg$ by the series \eqref{polRBljser} to eliminate the dependence on $\rg$. This way we obtain the series \eqref{PolW}.
\end{proof}

Combining now Theorem \ref{theder} with expressions \eqref{expTa}, \eqref{expWa} and Lemmas \ref{lemT}, \ref{lemW}, we calculate the partial derivatives of the symplectic invariant.

\begin{co}
The partial derivative of the (scaled) Taylor series invariant with respect to the coordinate $l$ is
\begin{equation}
\begin{aligned}
\frac{\partial \mcS}{\partial l} &=-\pi + \arctan \left( \dfrac{t - R (1 + t) - R^2 (1 - 2 t)}{(1 - R) \ra} \right) \\&+ \dfrac{1}{8 R \ra^3 (t + R (-1 + 2 t))} (\mu_{10} l + \mu_{01} j)\varepsilon + \mcO (\varepsilon^2).
\label{poldsdl}
\end{aligned}
\end{equation}
\end{co}
\begin{co}
The partial derivative of the (scaled) Taylor series invariant with respect to the coordinate $j$ is
\begin{equation}
\begin{aligned}
\frac{\partial \mcS}{\partial j} =& \ln \left( \frac{4 \ra^3}{R^{3/2}(1-t)t^2} \right) + \dfrac{1}{8 R \ra^3} (\ka_{10}l + \ka_{01}j)\varepsilon \\&+ \dfrac{1}{64 R^2 \ra^6} (\ka_{20}l^2 + \ka_{11} l j + \ka_{02} j^2)\varepsilon^2 + \mcO (\varepsilon^3).
\label{poldsdj}
\end{aligned}
\end{equation}
\end{co}

From these two corollaries we reconstruct the invariant. Exceptionally, in the formulation of Theorem \ref{theoinv} and Theorem A, $l$ and $j$ denote the values of the unscaled functions $L$ and $J$.

\begin{theo}
\label{theoinv}
The Taylor series symplectic invariant of the coupled angular momenta is given by 
\begin{align*}
S(l,j) =& \, l \arctan \left( \dfrac{{R_2}^2 (-1 + 2 t) - R_1 R_2 (1 + t) + {R_1}^2 t}{(R_1 - R_2)R_1 \ra} \right)+j \ln \left( \frac{4 {R_1}^{5/2}\ra^3}{{R_2}^{3/2}(1-t)t^2} \right) \\ \nonumber
+& \dfrac{l^2}{16 {R_1}^4R_2 \ra^3} \left( - {R_2}^4 (-1 + 2 t)^3 +R_1 {R_2}^3 (1 - 17 t + 46 t^2 - 32 t^3)\right. \\ \nonumber &\hspace{2.2cm} \left. -3 {R_1}^2 {R_2}^2 t (1 - 7 t + 4 t^2)   
 +{R_1}^3 R_2 (3 - 5 t) t^2- {R_1}^4 t^3   \right) \\ \nonumber
+& \dfrac{lj}{8 {R_1}^3 R_2\ra^2} \left( (R_2 - R_1)({R_2}^2 (1 - 2 t)^2   + 2 R_1 R_2 t (-1 + 6 t)+{R_1}^2 t^2) \right)\\ \nonumber
+& \dfrac{j^2}{16 {R_1}^4 R_2 \ra^3} \left( 
  {R_2}^4 (-1 + 2 t)^3 
- R_1 {R_2}^3 (1 + 15 t - 42 t^2 + 16 t^3)
\right. \\ \nonumber &\hspace{2.2cm} \left.
+  {R_1}^2 {R_2}^2 t (3 + 3 t - 28 t^2) 
+ {R_1}^3 R_2 t^2 (-3 + 13 t)
+ {R_1}^4 t^3
 \right)
+ \mcO(3),
\end{align*}
where $\ra$ is as in Definition \ref{sqr}, $l$ is the value of the (unscaled) first integral $L$ and $j$ is the value of Eliasson's (unscaled) $Q_2$ function.
\end{theo}
\begin{proof}
We integrate expressions \eqref{poldsdj} and \eqref{poldsdl}. By definition of the Taylor series invariant the independent term is 0. Since by construction the linear coefficient in $l$ is defined mod $\pi$, we can ignore the $-\pi$ in equation \eqref{poldsdl}.

\begin{equation}
\begin{aligned}
\mcS(l,j) =& \arctan \left( \dfrac{ R^2 (-1 + 2 t) - R (1 + t) +t}{(1 - R) \ra} \right)l+\ln \left( \frac{4 \ra^3}{R^{3/2}(1-t)t^2} \right)j \\ 
+& \dfrac{l^2}{16 R \ra^3} \left(- R^4 (-1 + 2 t)^3+ 
 R^3 (1 - 17 t + 46 t^2 - 32 t^3)    \right. \\ &\hspace{1.5cm} \left.- 3 R^2 t (1 - 7 t + 4 t^2)+ R (3 - 5 t) t^2 - t^3\right) \\ 
+& \dfrac{lj}{{8 R \ra^2}} (-1 + R) ( R^2 (1 - 2 t)^2 + 2 R t (-1 + 6 t)+t^2 ) \\ 
+& \dfrac{j^2}{16 R \ra^3} \left( R^4 (-1 + 2 t)^3 - 
 R^3 (1 + 15 t - 42 t^2 + 16 t^3)   \right. \\  &\hspace{1.5cm} \left. + 
 R^2 t (3 + 3 t - 28 t^2) + R t^2 (-3 + 13 t)+t^3\right) + \mcO(3).
\end{aligned}
\label{Seqq}
\end{equation}
The last step is to revert the scaling introduced in equation \eqref{chan2} to obtain the unscaled symplectic invariant from $\mcS(l,j)$. In other words, we have to go from scaled functions $\mcL,\mcI,\mcJ,\mcB, \mcS$ and scaled variables $l,j,w$ to unscaled functions $L,I,J,B, S$ and unscaled variables $\tilde{l},\tilde{j},\tilde{w}$. The unscaled symplectic invariant will then be
\begin{align*}
	S(\tilde{l},\tilde{j})&=2\pi I(\tilde{l},B(\tilde{l},\tilde{j})) - 2\pi I(0,0) + \imp(\tilde w \ln \tilde w - \tilde w) \\
	&= 2\pi R_1 \mcI (l,\mcB (l,j)) - 2\pi R_1 \mcI (0,0) + R_1 \imp (w \ln w - w) + R_1 j \ln R_1 = \\
	&= R_1 \mcS(l,j) + R_1 j \ln R_1 = R_1 \mcS \left( \frac{\tilde{l}}{R_1},\frac{\tilde{j}}{R_1} \right) + \tilde{j} \ln R_1,
\end{align*}
 where we used that $\tilde w \ln \tilde w - \tilde w = R_1 w \ln (R_1 w) - R_1 w = R_1 w \ln w - R_1 w + R_1 j \ln R_1$. 
 We use equation \eqref{Seqq} 
 to obtain the  unscaled invariant $S(\tilde{l},\tilde{j})$. 
 Now we drop the tildes to obtain the final result as stated in this Theorem and in Theorem A.
\end{proof}

\subsection{Properties of the Taylor series invariant} 
We discuss now some consequences of Theorem \ref{theoinv}. 

\subsubsection*{\textbf{The reverse coupled angular momenta}} 
In \S \ref{sec:CAM} we have defined the reverse coupled angular momenta as the semitoric system given by the standard coupled angular momenta but with the condition $R_1 > R_2$. Both semitoric systems are isomorphic by changing the sign of $L$ (Proposition \ref{proptrans}). Sepe and \vungoc describe in \cite{SV} how the calculation of the Taylor series is affected by changing the signs of the two variables. In particular, they describe an action of the group $K_4 = \Z_2 \times \Z_2$ on the space of Taylor series of focus-focus points given by 
\begin{equation}
\begin{array}{ccl}
\quad \text{\textbf{Element of }} K_4 \quad &\quad \text{\textbf{Signs}} \quad &\quad \text{\textbf{Taylor series}} \quad \\[0.1cm]
(0,0) & (+,+) & \;\;\,S(X,Y) \\[0.1cm]
(1,0) & (-,+) & \;\;\,S(-X,Y) + \pi X \\[0.1cm]
(0,1) & (+,-) & -S(X,-Y) \\[0.1cm]
(1,1) & (-,-) & -S(-X,-Y)+\pi X,\\[0.1cm] 
\end{array}
\label{septrans} 
\end{equation} where the first column indicates an element of $K_4$, the second column the corresponding sign changes that we do on the Eliasson's coordinates \eqref{qs} and the third column the effect on the Taylor series. In particular, since the first Eliasson coordinate coincides with the first component of the energy-momentum map, $Q_1=L$, a change in the sign of $L$ changes the series by $S(X,Y) \mapsto S(-X,Y) + \pi X$. As a consequence, we can immediately calculate the Taylor series invariant of the reverse coupled angular momenta.

\begin{co}
\label{Taylinv} Let $S_{(R_1,R_2,t)}(l,j)$ denote the Taylor series invariant of the standard coupled angular momenta with parameters $R_1, R_2, t$ given in Theorem \ref{theoinv}, where $R_1 < R_2$. Then,
$$ S_{(R_1,R_2,t)}(l,j) = S_{(R_2,R_1,t')}(-l,j),\quad t' := \dfrac{R_1 t}{R_2 + R_1 t - R_2 t}. $$
\end{co}
\begin{proof}
Proposition \ref{proptrans} establishes that the standard coupled angular momenta and the reverse coupled angular momenta are isomorphic as semitoric systems, except for a sign change in $L$. The isomorphism swaps $R_1$ and $R_2$ and changes $t$ to $t'$. This implies that the Taylor series invariant, together with the other four semitoric invariants, must coincide. Alternatively, it can also be verified by direct substitution in the Taylor series in Theorem \ref{theoinv}.
\end{proof}

\subsubsection*{\textbf{The Kepler problem}}
The Taylor series invariant of the coupled angular momenta can also be used in the analysis of the Kepler problem, the classical analogue of the hydrogen atom. This problem is known to be separable in four different systems of coordinates: spherical, parabolic, prolate spheroidal and sphero-conical coordinates (cf.\ Cordani \cite{Co}). The second author and Waalkens proved in \cite{DW} that the Kepler problem in prolate spheroidal coordinates has a pinched torus and as a consequence it exhibits monodromy for a certain energy range.

The Kepler problem in prolate spheroidal coordinates is a completely integrable system defined on $T^*\R^3$ with momentum map $(H,L_z,G)$, where $H$ is the Hamiltonian, $L_z$ is the $z$-component of the angular momentum, $G = {L}^2 + 2 a e_z$, $a$ is a parameter and $e_z$ is the $z$-component of the Laplace-Runge-Lenz vector. After performing symplectic reduction on the Hamiltonian flow, we obtain $\mbS^2 \times \mbS^2$ as reduced phase space, see Dullin $\&$ Waalkens \cite{DW} for more details. The reduced system can be expressed in coordinates $(x_1,y_1,z_1,x_2,y_2,z_2) \in \mbS^2 \times \mbS^2$ as
\begin{equation}
\begin{cases}
L_z(x_1,y_1,z_1,x_2,y_2,z_2)\;:= n (z_1+z_2), \\
G\;(x_1,y_1,z_1,x_2,y_2,z_2)\;:= (x_1x_2 + y_1y_2 + z_1z_2) +\ja{c}(z_1-z_2),\\
\end{cases}
\label{kep1}
\end{equation} where $n$ and $c$ are positive constants and $\om_K = -n\, (\om_{\mbS^2} \oplus \om_{\mbS^2})$.

\begin{pro}
The reduced Kepler problem \eqref{kep1} is isomorphic to the coupled angular momenta with $R_1=R_2$. 
\label{propkep}
\end{pro}
\begin{proof}
This follows immediately from Proposition \ref{proptrans}, since for $c=t b$ and $R_1 = R_2 = n$,  
the system $(L', H')$ defined in the proof satisfies $H'= t G$ and thus has the same invariants as the Kepler problem $(L_z, G)$.

\end{proof}

\begin{co}
The Taylor series invariant of the reduced Kepler problem in spheroidal coordinates is
\begin{align*}
S_K (l,j) &= \dfrac{\pi}{2} l + j \ln \left(16 n \sqrt{c} (2 - c)^{3/2} \right) \\& 
- \dfrac{ (-9 + 10 c - 2 c^2) }{16 n \sqrt{c} (2 - c )^{3/2}}l^2 - \dfrac{(-3 + 6 c + 2 c^2)}{16n \sqrt{c} (2 - c)^{3/2}}j^2 + \mcO(3).
\end{align*}
\end{co}
\begin{proof}
We substitute $R_1=R_2=n$ and equation \eqref{fort} into Theorem \ref{theoinv}.
\end{proof}
 
The focus-focus points exists for $c \in\ ]0,2[$ and we observe that when 
the boundaries of this interval are approached the 
coefficients of the symplectic invariant diverge.
It is interesting to note that the leading order of the divergence is, however, different at the two endpoints. Another observation is that the Taylor series invariant of the Kepler problem does not have second order terms that are odd in $l$. More generally we obtain:

\begin{pro}
\label{prop25}
The Taylor series invariant of the Kepler problem does not have non-linear odd powers of $l$.
\end{pro}
\begin{proof}
The Kepler problem corresponds to the case $R_1=R_2$, which is common for both the standard coupled angular momenta and the reverse coupled angular momenta. Since they are related by a change of sign of $L$ (see Proposition \ref{proptrans}), the Taylor series must be invariant under the transformation $l \mapsto -l$. From equation \eqref{septrans}, we conclude that $S(l,j) = S(-l,j) + \pi l$, so all higher order terms that are odd in $l$ must vanish.
\end{proof}

\subsubsection*{\textbf{Asymptotic behaviour of the Taylor series invariant}}
We know that the singular point $m \in M$ is of focus-focus type only when $t \in\ ]t^+, t^-[$, where $t^\pm$ is defined in \eqref{sys1}. Since the Taylor series is an invariant associated only to focus-focus singularities, a reasonable question to ask is what happens to the Taylor series in the limits $t \to t^+$ and $t \to t^-$. 

\begin{re}
\label{reminf}
In the limit $t \to t^\pm$, the coefficients of the Taylor series invariant in Theorem \ref{theoinv} diverge. More specifically, the linear coefficient in $j$ and the quadratic coefficients go to $\infty$, while the linear coefficient in $l$, which is only defined mod $\pi$, goes to $\frac{\pi}{2}$.
\end{re}

This is an immediate consequence of the fact that 
$$ \lim_{t \to t^\pm} r_A =0,$$ as we can see from equation \eqref{RA}. We expect that the divergence of the coefficients of the Taylor series invariant is typical when approaching the Hamiltonian Hopf bifurcation. For more details on the Hamiltonian Hopf bifurcation as it appears in deformations of semitoric systems, see Dullin $\&$ Pelayo \cite{DP}. When approaching the Hopf bifurcation where the focus-focus point becomes degenerate the real part of the eigenvalue of the focus-focus point vanishes. Recall that in our case this is given by $\frac{r_A}{2R}$, and hence the coefficients of the Taylor series invariant diverge with $r_A \to 0$.

The linear term of the Taylor series invariant in $l$ has $\ra$ in the denominator of the argument of the arctan, while the numerator does not vanish when $t \to t^\pm$. The linear term in $j$ has $\ra$ in the numerator of the argument of the ln and therefore diverges. The quadratic coefficients have powers of $\ra$ in the denominator, while the numerators do not vanish when $t \to t^\pm$.

To understand how the Taylor series invariant behaves as we approach this limit, it is helpful to define the following coordinates on the space of parameters,
\begin{equation}
\begin{aligned}
u&:=-\dfrac{1}{2} \log \left( R \right) 
\\
v&:= \text{artanh} \left( \dfrac{R- t - 2R t}{2 \sqrt{R}\, t} \right) 
\end{aligned}
\label{uv}
\end{equation}
as well as the scaling factor $\ka := \sqrt{R_1 R_2} >0$. This defines a diffeomorphism between $(R, t) \in \R^+ \times\ ]t^-, t^+[$ and $(u,v) \in \R^2$.
 The parameters $u,v$ conveniently adapt to the symmetries of the system.

\begin{re}
For fixed $R_1,R_2$, or equivalently $u,\ka$, the interval of existence of the focus-focus point $t^- < t < t^+$ corresponds to $-\infty < v < + \infty$. In particular, the limits $t \to t^{\pm} $ correspond to $v \to \pm \infty$.
\end{re} 

\begin{re}
The region $u<0$ corresponds to the usual coupled angular momenta and the region $u>0$ to the reverse coupled angular momenta. The case $u=0$ is the Kepler problem.
\end{re}

\begin{re}
The transformation described in Proposition \ref{proptrans} corresponds to changing $u \mapsto -u$ and leaving $v$ (and $\kappa$) unchanged.
	The parameters $u,v$ are defined so that the involution of Proposition \ref{proptrans} becomes a reflection.
\end{re}
 
The Taylor series invariant in these coordinates takes the form
\begin{align}
\hspace{1cm} S(l,j)= 
&\arctan \left( \dfrac{2e^{u} \cosh v + \sinh v + e^{2 u} \sinh v }{1-e^{2u}} \right)l \nonumber \\
&+ \ln \left( \dfrac{16 \ka}{\cosh^3 v\, (\cosh u + \tanh v)}\right) j \nonumber\\
&-\dfrac{1}{64\ka}(-7 \cosh v+3\cosh 3v+\cosh u)(11 \sinh v + 3 \sinh 3v)l^2 \\&-\dfrac{1}{8\ka}(1+3\cosh 2v )(\sinh u) lj \nonumber \\&-\dfrac{1}{64\ka}  (3 \cosh v + 17 \cosh 3v + \cosh u)(9 \sinh v + 17 \sinh 3v) j^2 + \mcO(3)\nonumber.
\end{align}

The dependence on $u$ takes the forms $\sinh u$, which is an odd function, and $\cosh u$, which is an even function. In particular, we see that the transformation in Proposition \ref{proptrans}, which corresponds to $u \mapsto -u$, changes the sign of the argument of the arctan in the linear coefficient of $l$ and changes also the sign of the quadratic term of $lj$, while leaving the rest unchanged, as expected from Corollary \ref{Taylinv}. For the Kepler problem, which corresponds to $u=0$, we see that the term quadratic of $lj$ vanishes, as expected from Proposition \ref{prop25}.

The dependence of the coefficients of the linear terms on the parameters $(u,v)$ is shown in Figure~\ref{linterms}. We see now what happens as we take the limit $v \to \pm \infty$, corresponding to the region where the focus-focus point --and therefore the Taylor series invariant-- is defined. 

The coefficient of $l$ can be rewritten in the form
\[
     -\frac{\pi}{2} + 2 \arctan( e^{-v} \tanh(u/2) ) \mod \pi,
\]
where the asymptotic behaviour for large $|v|$ becomes clear.

The coefficient of the linear term in $j$ diverges linearly to $-\infty$ when $v \to \pm \infty$. More precisely, for $v \gg 0$ we have
$$ \ln \left( \dfrac{16 \ka}{\cosh^3 v\, (\cosh u + \tanh v)}\right) \simeq -3v +\ln \left( \dfrac{64 \ka}{\cosh^2(\frac{u}{2})} \right) \to -\infty. $$ For $v\ll 0$ we must distinguish between the case $u \neq 0$, for which we have
$$ \ln \left( \dfrac{16\ka}{\cosh^3 v\, (\cosh u + \tanh v)}\right) \simeq 3v +\ln \left( \dfrac{64\ka}{\sinh^2(\frac{u}{2})} \right) \to -\infty$$ and the Kepler case $u=0$, for which we have
$$ \ln \left( \dfrac{16\ka}{ e^v \cosh^2 v }\right) \simeq  v +\ln \left( 64 \ka \right)\to -\infty. $$
This is consistent with the expected divergence of the coefficients in Remark \ref{reminf}. Note that while for $u\not = 0$ the leading order behaviour is the same up to a sign, 
for the Kepler case $u=0$ the leading order behaviour for $v \to +\infty$ and $v \to -\infty$ is different
by a factor of $-3$.

\begin{figure}[ht]
 \centering
\subfloat[Coefficient of $l$]{\includegraphics[width=6cm]{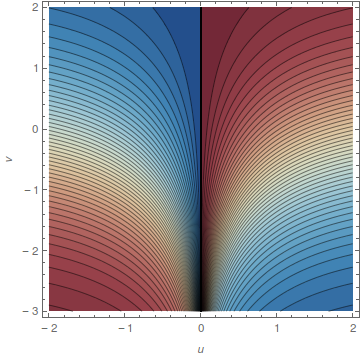}} \hspace{0.5cm}
\subfloat[Coefficient of $j$]{\includegraphics[width=6cm]{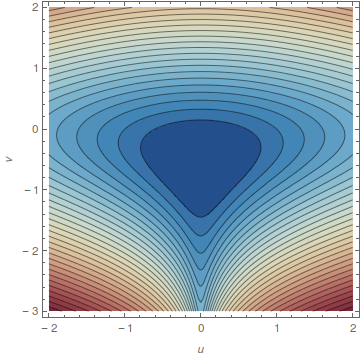}}
\caption{\small Contour plots of the coefficients of the linear terms of the Taylor series symplectic invariant as a function of $u,v$ while setting $\ka=1$. Blue colours represent positive values and red colours represent negative values.}
 \label{linterms}
\end{figure}

The coefficients of the quadratic terms are shown in Figure \ref{quadterms}.
The coefficient of $l^2$ behaves as
\[
    -\sign(v) \frac{9}{256 \kappa} e^{6 |v|} \text{ for $|v| \to \infty$}.
\]
The coefficient of $j^2$ behaves as
\[
    -\sign(v) \frac{289}{256 \kappa} e^{6 |v|} \text{ for $|v| \to \infty$}.
\]
The coefficient of $jl$ behaves as
\[
   - \frac{3}{16 \kappa} \sinh u \, e^{2 |v|}  \text{ for $|v| \to \infty$}.
\]

\begin{figure}[ht]
 \centering
\subfloat[Coefficient of $l^2$]{\includegraphics[width=6cm]{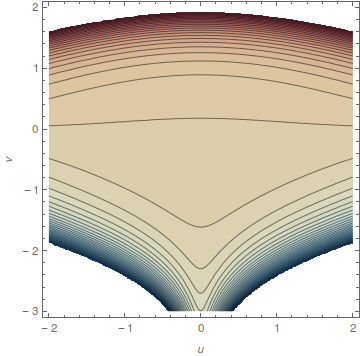}} \hspace{0.5cm}
\subfloat[Coefficient of $j^2$]{\includegraphics[width=6cm]{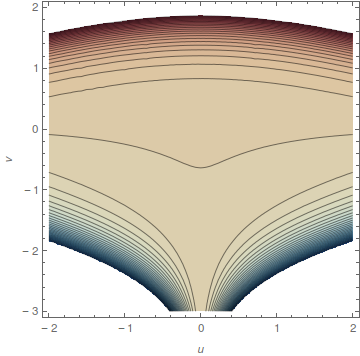}}\\
\subfloat[Coefficient of $lj$]{\includegraphics[width=6cm]{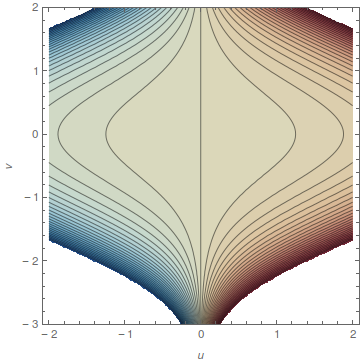}}
\caption{\small Contour plots of the coefficients of the quadratic terms of the Taylor series symplectic invariant as a function of $u,v$ while setting $\ka=1$. Blue colours represent positive values and red colours represent negative values.}
 \label{quadterms}
\end{figure}

\newpage
\section{The polygon invariant of the coupled angular momenta}
\label{sec:polpol}

In this section we compute the polygon invariant of the coupled angular momenta \eqref{sys1}. For the standard case $R_1<R_2$, Le Floch \& Pelayo already computed this in \cite{LFP}. We extend their results to the reverse case $R_1 > R_2$ and to the Kepler problem case $R_1 = R_2$. We will assume that we are in the case with one non-degenerate singularity of focus-focus type, i.e.\ $\nff=1$, thus $t^- < t < t^+$.

The polygon invariant consists of a collection of rational convex polygons, related by the action of the group $\Z_2 \times \mcG$ (cf.\ \S\ref{sec:polygon}). In Figure \ref{figpolstan} we can see some of them for the standard case $R_1<R_2$. The polygons on the same row are related by the $\Z_2$-action, i.e.\ they correspond to different choices of sign $\epsilon$ -- or equivalently of cut direction -- namely $\epsilon=+1$ on the left and $\epsilon=-1$ on the right. The polygons on the same column are related by the $\mcG$-action, i.e.\ by a linear transformation $T^k$ defined in \eqref{TT}.

\begin{figure}[ht]
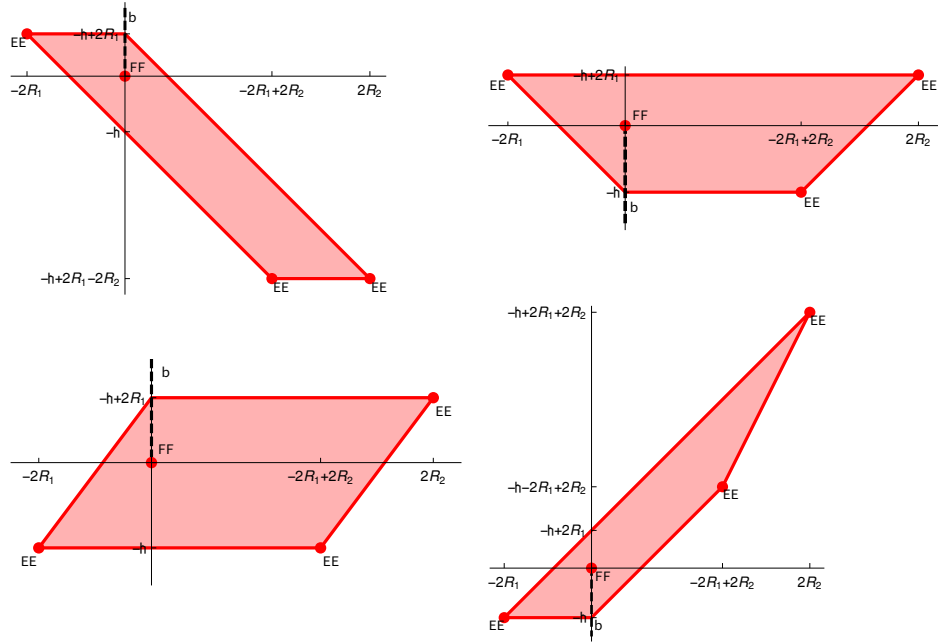

\begin{tabular}{m{6cm}m{6cm}}
\includegraphics[width=5cm]{pol-1+1ni.pdf} &
\includegraphics[width=6cm]{pol-1-1ni.pdf} \\
\includegraphics[width=6cm]{pol0+1ni.pdf} &
\includegraphics[width=4.5cm]{pol0-1ni.pdf}
\end{tabular}
\caption{Some elements of the polygon invariant for $R_1<R_2$, where $\mfh$ denotes the height invariant.}
\label{figpolstan}
\end{figure}
\begin{theo}
\label{theopol}
The polygon invariant of the coupled angular momenta for the reverse case $R_1 < R_2$ and the Kepler problem case $R_1 = R_2$ are given by the $(\Z_2 \times \mcG)$-orbits represented in Figures \ref{figpolrev} and \ref{figpolkep} respectively.
\end{theo}
\begin{proof}
Proposition \ref{proptrans} establishes an isomorphism of semitoric systems between the standard case $R_1<R_2$ and the reverse case $R_1 > R_2$ of the coupled angular momenta by changing the sign of $L$. As a consequence, we can obtain the polygon invariant of the reverse case simply by reflecting the polygons with respect to the vertical axis. An alternative way to obtain this result is to compute directly the polygons from the Duistermaat-Heckman function given in Theorem 5.3 of \vungoc\!\! \cite{Vu2}, as Le Floch \& Pelayo do in \cite{LFP}. The isotropy weights coincide, but the ordering of the critical points of $L$ gets exchanged, obtaining again the desired result.

For the Kepler problem case, we can either take the polygon invariant for the standard case or for the reverse case and take the limit $R_1 = R_2$. Note that since the result in both ways must coincide, the reflection with respect to the vertical axis of each polygon in the invariant must also be in the invariant.
\end{proof}

\begin{figure}[ht]
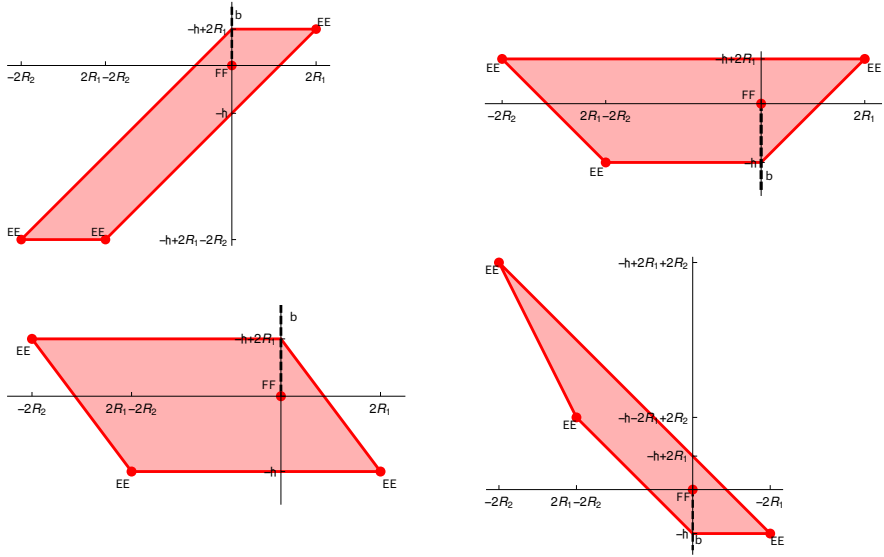
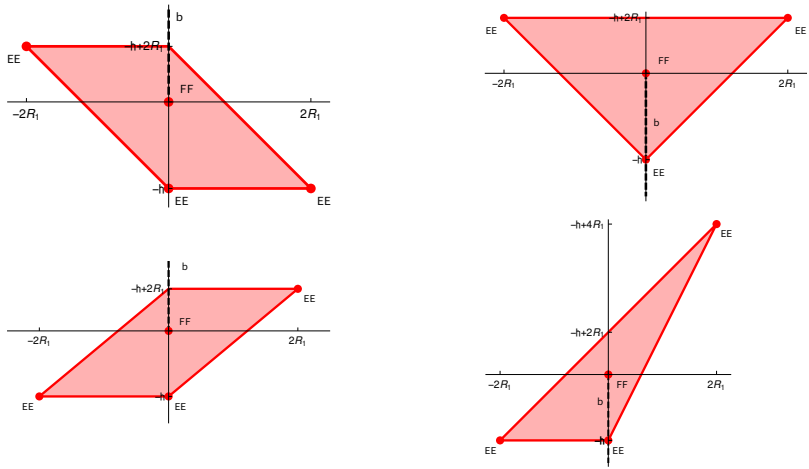

\begin{tabular}{m{6cm}m{6cm}}
\includegraphics[width=4.3cm]{polr-1+1.pdf} &
\includegraphics[width=5.3cm]{polr-1-1.pdf} \\
\includegraphics[width=5.3cm]{polr0+1.pdf} &
\includegraphics[width=4cm]{polr0-1.pdf}
\end{tabular}
\caption{Some elements of the polygon invariant for $R_1>R_2$, where $\mfh$ denotes the height invariant.}
\label{figpolrev}
\end{figure}

\begin{figure}[ht]
\begin{tabular}{m{6cm}m{6cm}}
\includegraphics[width=4.3cm]{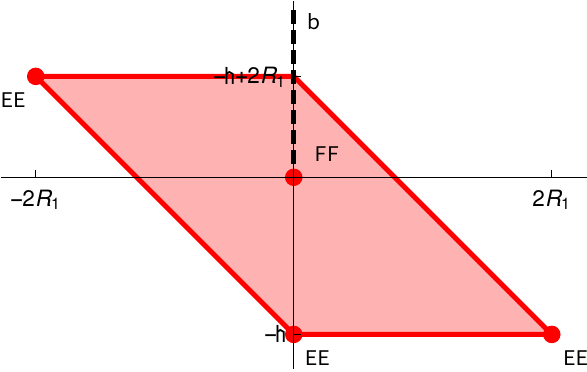} &
\includegraphics[width=4.3cm]{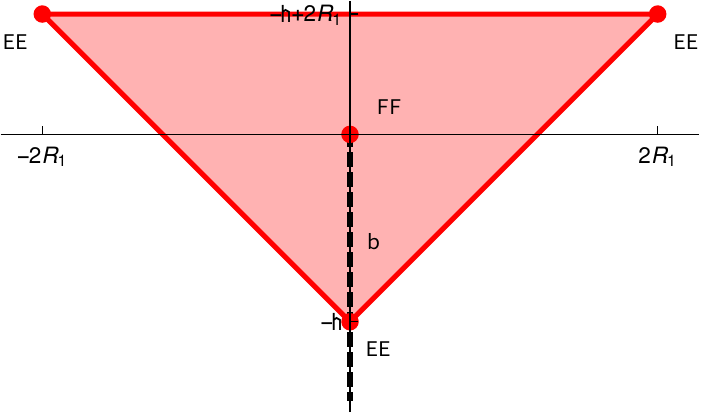} \\
\includegraphics[width=4.3cm]{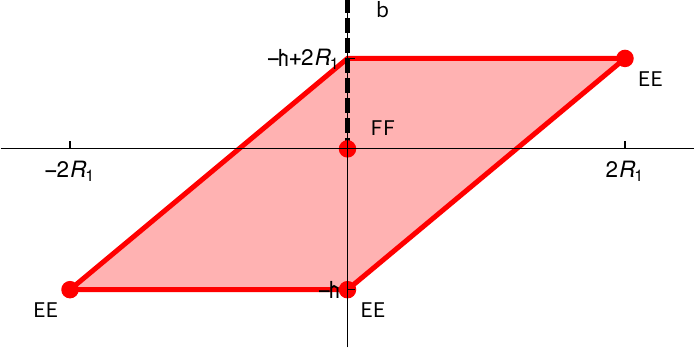} &
\includegraphics[width=3.3cm]{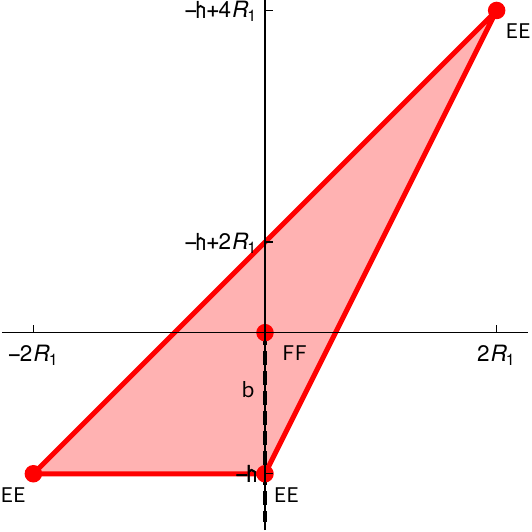}
\end{tabular}
\caption{Some elements of the polygon invariant for $R_1=R_2$, where $\mfh$ denotes the height invariant.}
\label{figpolkep}
\end{figure}

\section{The height invariant of the coupled angular momenta}

In this section we compute the height invariant of the coupled angular momenta \eqref{sys1}. We recall (cf.\ \S \ref{sec:height}) that the height invariant is trivial if there is no singularity of focus-focus type, i.e.\ $\nff=0$. Hence, as usually, we assume that $t^- < t < t^+$, so that $m=N \times S \in M$ is a non-degenerate focus-focus  singularity. For the specific case $t=\frac{1}{2}$, the height invariant has been already computed by Le Floch \& Pelayo in \cite{LFP}.

Consider the submanifold $Y := L^{-1} (L(m)) =  \{ p \in M \,|\, L(p) = L(m)=0 \}$, depicted in Figure \ref{hami}. The height invariant $\mfh$ is the symplectic volume of 
$$Y^- := Y \cup \{ p \in M \,|\, H(p)<H(m)=0 \} = \{ p \in M | L(p)=0, H(p)<0 \},$$
which is the area of the region outside of the blue curve $\mcH_0(q_2,p_2)=0$. Looking at the definition \eqref{int1} of the action $\mcI(l,h)$ of the reduced system, which corresponds to the area of $Y^+ := Y \cup \{ p \in M \,|\, H(p)>H(m)=0 \} = \{ p \in M | L(p)=0, H(p)>0 \}$ with a negative sign, the height invariant is given by the complementary area with a positive sign, i.e.\
\begin{equation}
\mfh = 2\min\{R_1,R_2\}  + I(0,0) 
= 2\min\{R_1,R_2\} + \dfrac{R_1}{2\pi} \oint_{\beta_{0,0}} q_2 \, \dee p_2,
\label{hei1}
\end{equation} where $2\min\{R_1,R_2\}$ corresponds to the symplectic area of the reduced phase space for $l=0$. In this section, $\arctan$ will always denote the determination that takes values in $[0,\pi]$.

\begin{theo}
\label{theoheight}
The height invariant of the coupled angular momenta is
\begin{align*}
\mfh=2\min\{R_1,R_2\} + \dfrac{R_1}{\pi t} &\left( 
\ra - 
2R\,t \arctan \left( \dfrac{\ra}{R-t} \right)-
2\,t \arctan \left( \dfrac{\ra}{R+t-2R\,t} \right)
\right) 
\end{align*} where $\ra$ is the discriminant square root from Definition \ref{sqr}. 
\end{theo}

\begin{proof}
We want to compute the value of \eqref{hei1},
\begin{align*}
\mfh &= 2\min\{R_1,R_2\}+ \dfrac{R_1}{2\pi} \oint_{\beta_{0,0}} q_2 \, \dee p_2= 2\min\{R_1,R_2\} + \dfrac{R_1}{\pi} \int_{\ze_2}^{\ze_3} q_2 \, \dee p_2 \\
&= 2\min\{R_1,R_2\} + \dfrac{R_1}{\pi} \int_{\ze_2}^{\ze_3} \arccos \left(  \dfrac{-A(p_2)}{\sqrt{B(p_2)}}\right) \dee p_2.
\end{align*}  For $(l,h)=(0,0)$, the polynomials $A(p_2)$ and $B(p_2)$ defined in \eqref{ABdef} take the form
\begin{align*}
A(p_2) &= \dfrac{1}{R} \left(-{p_2}^2 t + p_2 (-R + t + 2 R t) \right)\\
B(p_2) &= \dfrac{1}{R^2} (-2 + p_2) {p_2}^2 (p_2 - 2 R) t^2 
\end{align*}
and the polynomial $P(p_2)$ defined in \eqref{Pdef} has the form 
\begin{align*}
P(p_2)&=\dfrac{(-2 + p_2) {p_2}^2 (p_2 - 2 R) t^2}{R^2} - \dfrac{(-{p_2}^2 t + 
   p_2 (-R + t + 2 R t))^2}{R^2} \\
& = \dfrac{{p_2}^2}{R^2} \left( -R^2 (1 - 2 t)^2 + 2 R (1 + p_2 (-1 + t)) t - t^2) \right), 
\end{align*} which has the roots $$ \ze_1 = \ze_2 =0,\qquad \ze_3 = \dfrac{R^2 (1 - 2 t)^2 - 2 R t + t^2}{2 R (-1 + t) t}. $$ We can now compute the integral
\begin{small}
\begin{align*}
\mfh =& 2\min\{R_1,R_2\} - \dfrac{R_1}{\pi} \int_{\ze_2}^{\ze_3} \arccos \left( \dfrac{R + (-1 + p_2) t - 2 R t}{t \sqrt{(-2 + p_2) (p_2 - 2 R)}} \right) \dee p_2 \\
=& 2\min\{R_1,R_2\} -\dfrac{R_1}{\pi t} \left[ \sqrt{-R^2 (1 - 2 t)^2 + 2 R (1 + p_2 (-1 + t)) t - t^2}\right. \\
 & \left.+ p_2 t \arccos \left(\dfrac{R + (-1 + p_2) t - 2 R t}{t\sqrt{(-2 + p_2) (p_2 - 2 R)}} \right) \right. \\
&- \left. 2Rt \arctan \left( \dfrac{\sqrt{-R^2 (1 - 2 t)^2 + 2 R (1 + p_2 (-1 + t)) t - t^2}}{R - t}\right)  \right. \\
\\& \left. - 2t \arctan \left( \dfrac{\sqrt{-R^2 (1 - 2 t)^2 + 2 R (1 + p_2 (-1 + t)) t - t^2}}{R+t-2Rt}\right)    \right]_{\ze_2}^{\ze_3} \\
=& 2\min\{R_1,R_2\} + \dfrac{R_1}{\pi t} \left( 
\sqrt{-R^2 (1 - 2 t)^2 + 2 R\, t - t^2} \right. \\&\left.- 
2R\,t  \arctan \left( \dfrac{\sqrt{-R^2 (1 - 2 t)^2 + 2 R\, t - t^2}}{R-t} \right)-
2t \arctan \left( \dfrac{\sqrt{-R^2 (1 - 2 t)^2 + 2 R\, t - t^2}}{R+t-2Rt} \right)
\right),
\end{align*}
\end{small}
 which is what we wanted to show.
\end{proof}

The height invariant $\mfh$ can also be expressed in terms of the parameters $(u,v)$ defined in \eqref{uv}, where the symmetry between the usual coupled angular momenta $R_1<R_2$ and the reverse case $R_1>R_2$ becomes apparent,
\begin{align*}
 \dfrac{\mfh}{\kappa} =& 2 e^{-|u|} + \dfrac{2}{\pi\cosh v} - \dfrac{2 e^u}{\pi} \arctan \left( \frac{1}{\sinh v + e^u \cosh v}\right) \\&- \dfrac{2 e^{-u}}{\pi} \arctan \left( \frac{1}{\sinh v + e^{-u} \cosh v}\right)  ,
\end{align*}
since it is invariant under the change $u \mapsto -u$. In Figure \ref{figheight} we can see how the height invariant changes with respect to $u$ and $v$. In the limit $v \to +\infty$, which corresponds to $t \to t^+$, the height invariant approaches the total volume of the reduced phase space, $2\kappa \min(e^u,e^{-u}) = 2\min\{R_1,R_2\}$. This means that the focus-focus singularity approaches the top part of the polygons in Figure \ref{twistbig}. In the limit $v \to -\infty$, which corresponds to $t \to t^+$, the height invariant approaches the value 0, which corresponds to the bottom of the polygons. 

In the Kepler case $u=0$ the heigh invariant simplifies to 
\[
  \mfh = 2 - \frac{2}{\pi} ( 2 \arctan e^v  -  \sech v).
\]

\begin{figure}[ht]
 \centering
 \includegraphics[width=7cm]{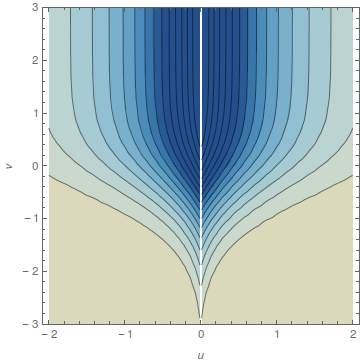}
 \caption{\small Contour plot of the height invariant as a function of $u$ and $v$ and setting $\kappa=1$. The intensity of the blue means higher positive values.}
 \label{figheight}
\end{figure}

\section{The twisting-index invariant of the coupled angular momenta}

In this section we compute the twisting-index invariant of the coupled angular momenta \eqref{sys1}. We recall that the twisting index is trivial if $\nff =0$ (cf.\ \S \ref{sec:twistingIndex}), thus we will make our usual assumption that we are in the situation $t^- < t < t^+$ for which $\nff=1$ (cf.\ \S \ref{sec:CAM}), i.e.\ the singularity in $m \in M$ is non-degenerate and of focus-focus type.

The twisting-index invariant is determined by fixing a weighted polygon $\De_{weight}=(\De,b_{\lam},\epsilon)$ of the polygon invariant and calculating the twisting index $k \in \Z$ corresponding to the focus-focus critical point. The twisting-index invariant is then the $\Z_2 \times \mcG$-orbit of $\De_{weight}=(\De,b_{\lam},\epsilon,k) \in \mcW$Polyg$(\R^2)\times\Z$, where $\mcG \simeq \Z$ and the $\Z_2 \times \mcG$-action is defined in \S \ref{sec:twistingIndex}. In the case of the coupled angular momenta, we obtain the following result. 

\begin{theo}
\label{theotwis}
The twisting-index invariant of the standard coupled angular momenta is the orbit of $(\De,b_{\lam},\epsilon,k)$ by $\Z_2 \times \mcG \simeq \Z_2 \times \Z$, where $\De$ is the rational polygon in Figure \ref{twisfig}, $b_{\lam} = \{(0,y)\;|\;y \in \R\}$, $\epsilon =+1$ and $k=0$. It is independent of the parameters $t, R_1$ and $R_2$.
\end{theo}

\begin{figure}[ht]
 \centering
 \includegraphics[width=9cm]{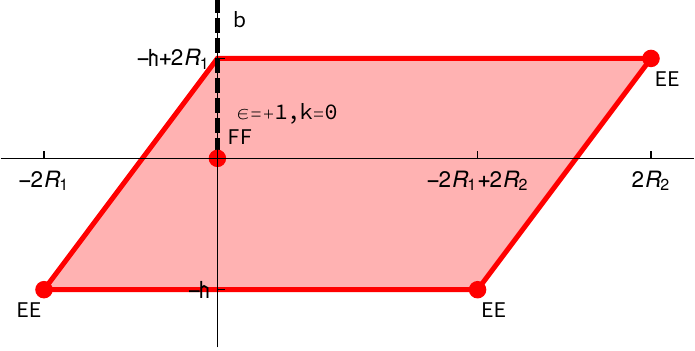}
 \caption{\small Weighted polygon $\De_{weight}$ from the polygon invariant, corresponding of the polygon $\De$, the line $b_{\lam}=b_0$ and the sign choice $\epsilon =+1$. This polygon has the twisting index $k=0$ associated. The image of the focus-focus singularity is $(0,0)$, while the three elliptic-elliptic singularities are in $(-2R_1,-\mfh)$, $(-2R_1+2R_2,-\mfh)$ and $(2R_2,-\mfh+2R_1)$, where $\mfh$ denotes the height invariant.}
 \label{twisfig}
\end{figure}

\begin{proof}
Fix $\epsilon=+1$. The critical value corresponding to the focus-focus singularity is $c = F(m) = (L,H)(m) = (0,0)$, thus $\lambda=0$ and $b_{\lam} = \{ (0,y) \;|\; y \in \R\}$. With our choice of $\epsilon$ we also have $b_{\lam}^{\epsilon} = \{(0,y) \;|\; y>0\}$. Let $U \subseteq M$ be a neighbourhood of $m$. We have seen in \S \ref{sec:twistingIndex} that there is a unique smooth map $I_m: F^{-1}(F(U)\backslash b_{\lam}^{\epsilon}) \to \R$ with \eqref{Xp} as Hamiltonian vector field and $\lim_{x \to m} I_m(x)=0$. We also know that the map can be made continuous in $W=F^{-1}(F(U))$ and it is of the form $I_m = \mfp \circ \Phi \circ F$, where $\mfp$ has as partial derivatives $\partial_i \mfp = \frac{\tau_i}{2\pi}$, with $\tau_i$ as in \eqref{taus}, $i=1,2$ and thus is of the form
\begin{align}
\label{Simw}
2\pi \mfp (\tilde{w}) &= S(\tilde{w})-\imp(\tilde{w} \ln \tilde{w}-\tilde{w})+\text{Const} \\&= S(\tilde{w}) - \tilde{w} \ln |\tilde{w}| - \tilde{w} \arg(\tilde{w}) +\tilde{w} +\text{Const}. \nonumber
\end{align}

If we compare this with \eqref{inv5} and \eqref{AI}, and use unscaled quantities, we see that $I(\tilde{l},B(\tilde{l},\tilde{j}))-I(0,0)$ satisfies these properties.  Therefore, by uniqueness, it must be the map $\mfp$. As a consequence, if we write $I_0 = I(0,0)$, we must have $I_m = \mfp \circ \Phi = I \circ F - I_0$.

The last step is to compare the privileged momentum map, which is $\nu = (L,I_m)$, with the polygons in the polygon invariant. Even though we only know $S(\tilde{w})$, and therefore $I_m(x)$, up to a finite order, it is enough to depict the image of $\nu$, as we can see in Figure \ref{twisfigmat}. We could compute the exact map using the Abelian integral \eqref{i2l} derived in the beginning, but using the (truncated) Taylor series invariant serves also as a numerical test of the correctness of the series.

Since it coincides with one of the polygons in the polygon invariant, namely the polygon $\De$ in Figure \ref{twisfig}, we conclude that this is the polygon for which $k=0$, since $\mu = \nu$, where $\mu:M \to \De$ is the corresponding momentum map to $\De$. Moreover, if we look at the linear term in $l$ of the Taylor series invariant in Theorem \ref{theoinv}, we see that it is continuous for all $t \in\, ]\ t^-,t^+[$, thus the determination of the $\arg$ function in \eqref{Simw} does not change. As a consequence, the twisting-index invariant is constant for all $t \in\, ]\ t^-,t^+[$.
\end{proof}

\begin{figure}[ht]
 \centering
 \includegraphics[width=9cm]{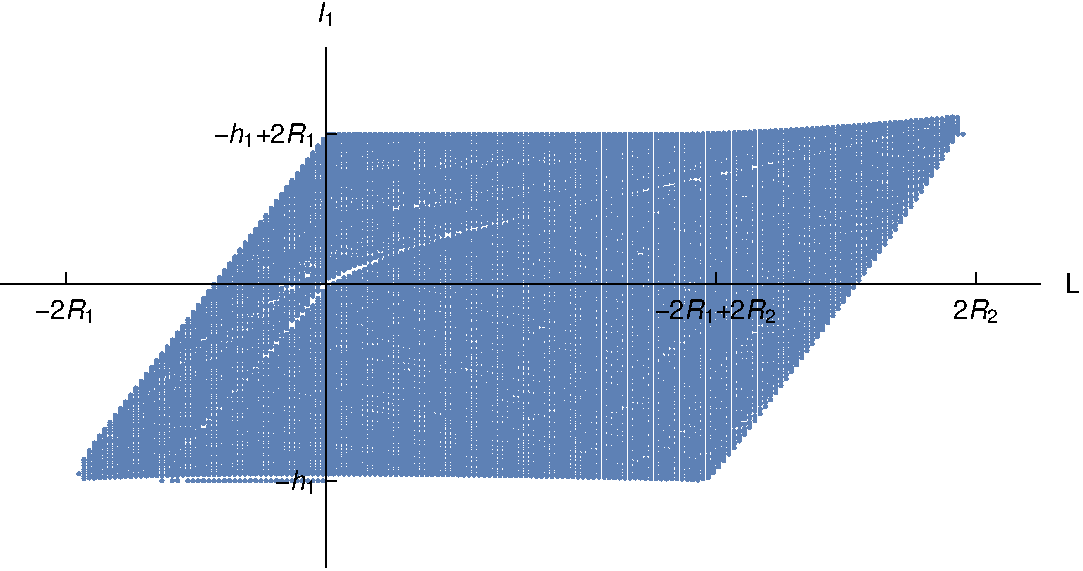}
 \caption{\small Computational depiction of the image of the privileged momentum map $\nu =(L,I_m)$ for $t=\frac{1}{2}$ using 265392 points and the Taylor series invariant up to order 2.}
 \label{twisfigmat}
\end{figure}

A last remark about Figure \ref{twisfig}: if we compare it with the one depicted by Le Floch and Pelayo in \cite{LFP}, we observe a horizontal translation by $R_1-R_2$ and a vertical translation by $\mfh$. The first one is due to our definition of the function $L$ in \eqref{sysin}, which differs by $R_1-R_2$ from theirs because we wanted to impose $(L,M)(m)=(0,0)$. The second corresponds to the requirement that $I_m(x)$ tends to 0 as $x \to m$. 

The $\Z_2 \times \mcG$-action on $(\De,b_{\lam},\epsilon,k)=(\De,b_0,+1,0)$ completely determines the twisting-index invariant and in particular allows us to associate the corresponding twisting index to all polygons of the polygon invariant. In Figure \ref{twistbig} some of them are depicted.

\begin{figure}[ht!]
\subfloat[$\epsilon=+1$, $k=-1$]{
 \includegraphics[width=0.4\textwidth]{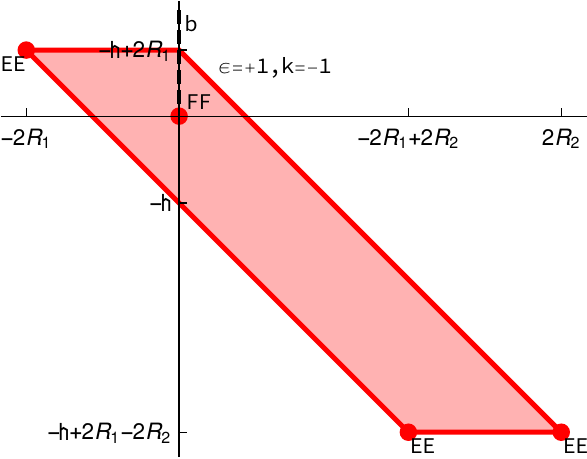}
}
\subfloat[$\epsilon=-1$, $k=-1$]{
 \includegraphics[width=0.45\textwidth]{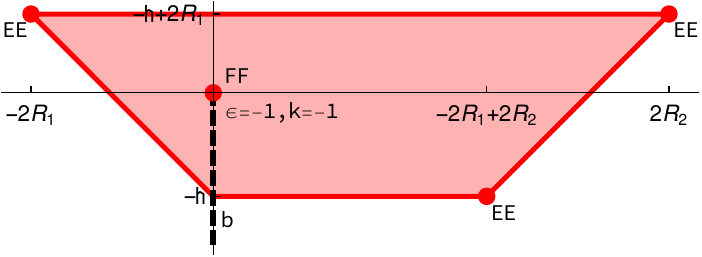}
}
\end{figure}

\begin{figure}[ht!]
\subfloat[$\epsilon=+1$, $k=0$]{
 \includegraphics[width=0.5\textwidth]{pol0+1.pdf}
}
\subfloat[$\epsilon=-1$, $k=0$]{
 \includegraphics[width=0.28\textwidth]{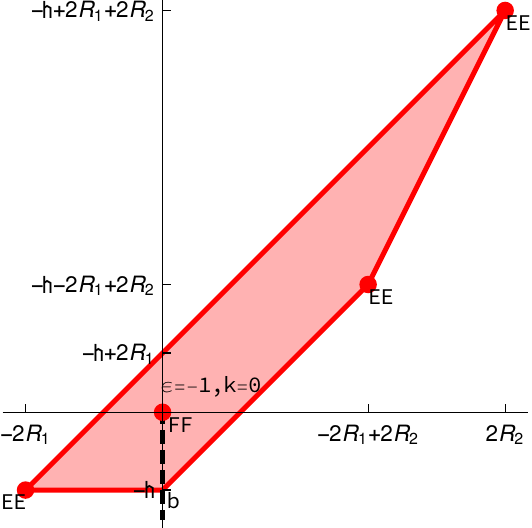}
}
\end{figure}

\setcounter{figure}{14}
\begin{figure}[ht!]
\subfloat[$\epsilon=+1$, $k=+1$]{
 \includegraphics[width=0.35\textwidth]{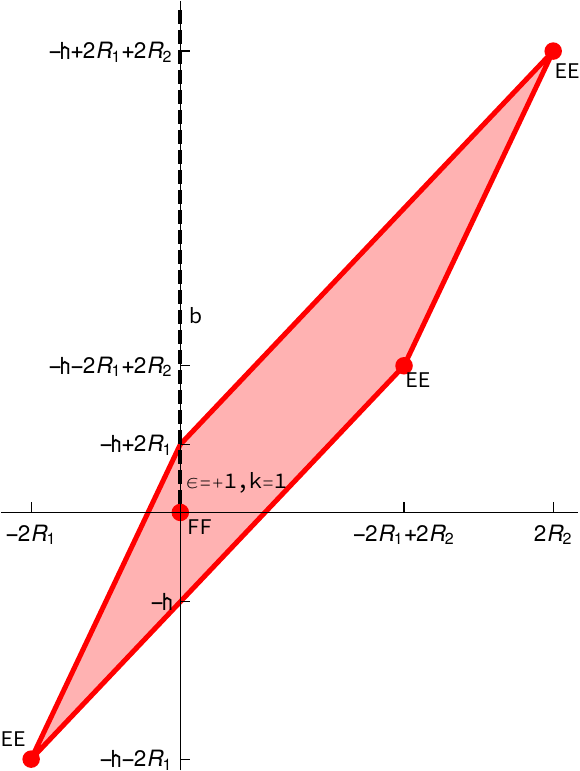}
}
\subfloat[$\epsilon=-1$, $k=+1$]{
 \includegraphics[width=0.3\textwidth]{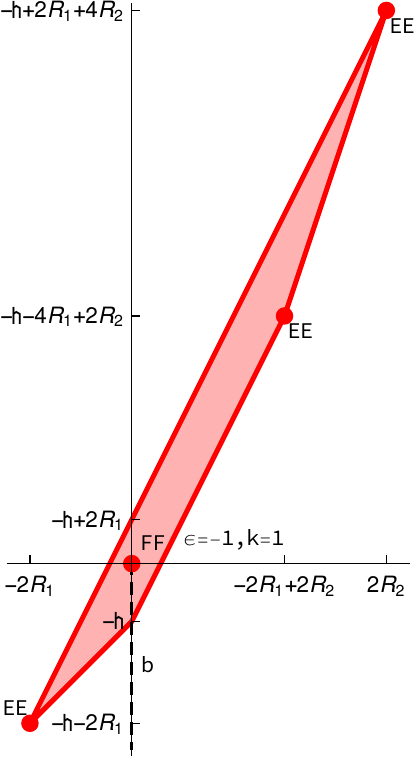}
}
 \caption{Representation of the twisting-index invariant of the standard coupled angular momenta. The polygon consists of an infinite collection of weighted polygons and the twisting-index invariant consists of the association of an index $k$ to each of the polygons. Polygons for $k=-1,0,1$ are depicted.}
 \label{twistbig}
\end{figure}
\newpage

\begin{co}
The twisting-index invariant of the coupled angular momenta for the reverse case $R_1>R_2$ and the Kepler problem $R_1=R_2$ is the association of indices to the elements of the polygon invariant represented in \S \ref{sec:intro} in Figure \ref{figtwistab}. 
\end{co}
\begin{proof}
We need to assign twisting indices to the polygons of the polygon invariant calculated in \S \ref{sec:polpol} and represented in Figures \ref{figpolrev} and \ref{figpolkep}. We know that the polygons on the same row must have the same index. Moreover, we have seen that the twisting indices do not vary with the parameters $R_1,R_2$, so in particular they must stay constant when approaching the Kepler case from either the standard or the reverse case. This leaves as only possibility the association represented in Figure \ref{figtwistab}.
\end{proof}

This completes the symplectic classification of the coupled angular momenta.

\newpage
\appendix

\section{Elliptic integrals}
\label{sec:ellipticInt}
In this section we briefly enumerate some properties of elliptic integrals that are used in the present paper. A more detailed overview can be found for example in Siegel \cite{Si1}, \cite{Si2} and Bliss \cite{Bl}.

An \emph{elliptic integral} is an integral of the form 
\begin{equation}
\mathcal{N}(x) := \int_c^x R(z,\sqrt{P(z)})\, d z,
\label{ellip}
\end{equation} where $z$ is a variable that may be real or complex, $c$ is a constant, $R(z,w)$ is a rational function and $P(z)$ is a polynomial of degree 3 or 4. In general, elliptic integrals cannot be expressed in terms of elementary functions, but it is possible to reduce them in terms of rational functions and the three \emph{Legendre canonical forms}:
\begin{align*}
& F(x; k) := \int_0^x \dfrac{dt}{\sqrt{(1-t^2)(1-k^2t^2)}}, \\
& E(x; k) := \int_0^x \dfrac{\sqrt{1-k^2t^2}}{\sqrt{1-t^2}}dt, \\
& \Pi(n; x; k) := \int_0^x \dfrac{dt}{(1-nt^2) \sqrt{(1-t^2)(1-k^2t^2)}}.
\end{align*} 

$F(x;k)$, $E(x;k)$ and $\Pi(n; x; k)$ are usually called \emph{incomplete elliptic integrals of first, second} and \emph{third kind} respectively. Their \emph{complete} counterparts correspond to the case $x=1$, i.e.\ $K(k):=F(1;k)$, $E(k):=E(1,k)$ and $\Pi(n,k):=\Pi(n;1;k)$ respectively. The parameter $k$ receives the names of \emph{(elliptic) modulus} or \emph{excentricity}, $n$ is called the \emph{characteristic} and $x$ is the \emph{argument}.

The function $K(k)$ can be expanded as
\begin{equation}
\begin{aligned}
K(k) &= \dfrac{1}{4}(k^2-1)- \dfrac{21}{128} (k^2-1)^2 + \dfrac{185}{1536} (k^2-1)^3 + ... \label{expK1} \\
 &+ \left(\ln \dfrac{1-k^2}{16} \right)\left( -\dfrac{1}{2} +\dfrac{1}{8}(k^2-1) - \dfrac{9}{128} (k^2-1)^2 - \dfrac{25}{512}(k^2-1)^3 + ... \right). 
\end{aligned}
\end{equation} The behaviour of $\Pi(n,k)$ depends on the parameters. In this paper we are always in the so-called \emph{circular case}, i.e. when either $k^2 < n < 1$ or $n<0$, cf.\ Cayley \cite{Ca}. In this regime, $\Pi(n,k)$ can be rewritten in terms of the Heuman's lambda function $\Lam_0$, defined by
\begin{equation}
\Lambda_0(\vt,k) := \dfrac{2}{\pi} \left( E(k) F(\vt,k') + K(k)E(\vt,k')-K(k)F(\vt,k') \right).
\label{lamlam}
\end{equation} In particular, for the positive circular case ($k^2 < n < 1$) we have 
\begin{equation*}
\Pi(n,k) = K(k) + \dfrac{\pi}{2} \sqrt{\dfrac{n}{(1-n)(n-k)}} \left( 1-\Lambda_0(\vt,k) \right),\quad \vt = \arcsin \sqrt{\dfrac{1-n}{n-k}} 
\end{equation*} and for the negative circular case ($n<0$) we have
\begin{equation*}
\Pi(n,k) = \dfrac{1}{1-n}K(k) + \dfrac{\pi}{2} \sqrt{\dfrac{n}{(1-n)(n-k)}} \left( 1-\Lambda_0(\vt,k) \right),\quad \vt = \arcsin \dfrac{1}{\sqrt{-n}}.
\end{equation*} Heuman's lambda function can be expanded as
\begin{equation}
\begin{aligned}
 \Lambda_0(\vt,k) &= \dfrac{2}{\pi} \vt + \sin \vt \cos \vt\left( \dfrac{1}{2\pi}(k^2-1) - \left( \dfrac{13}{32\pi} + \dfrac{3\sin^2 \vt}{16\pi} \right)(k^2-1)^2 + ...\right) \label{expL1} \\
 &+ \sin \vt \cos \vt \left(\ln \dfrac{1-k^2}{16} \right)\left( \dfrac{1}{2\pi}(k^2-1) - \left( \dfrac{3}{16\pi} + \dfrac{\sin^2 \vt}{8\pi} \right)(k^2-1)^2 + ...\right) . 
\end{aligned}
\end{equation}
While the variable $z$ from \eqref{ellip} is real in most applications, it is sometimes convenient to let it take complex values, too. The elliptic integral on which the integration takes place, 
\begin{equation*}
\Gamma := \{ (z,s) \in \mathbb{C}^2 \; :\; s^2 = P(z) \},
\end{equation*} will be a Riemann surface, i.e.\ a one-dimensional complex manifold. Using complex function theory we know that $\Ga$ will be a compact surface of genus $\frac{N-2}{2}$ if $n$ is even or $\frac{N-1}{2}$ if $n$ is odd, where $N$ is the degree of $P$. For the elliptic case, i.e.\ $\deg P=3$ or $4$, $\Ga$ will be homeomorphic to a torus $\T^2 =\mbS^1 \times \mbS^1$. Its first homology group will then be generated by two independent non-contractible cycles $\alpha$ and $\beta$ (see Figure \ref{cycles}). When considering elliptic integrals on a regular fibre close to the singular focus-focus fibre, we will refer to the cycle $\alpha$ as the \emph{`imaginary'} or \emph{`vanishing'} cycle, since it vanishes as we approach the singular fibre, and to the cycle $\beta$ as the \emph{`real'} cycle, since it coincides with the real elliptic curve when $z$ is considered as a real variable. 

\begin{figure}[ht!]
 \centering
 \includegraphics[width=6cm]{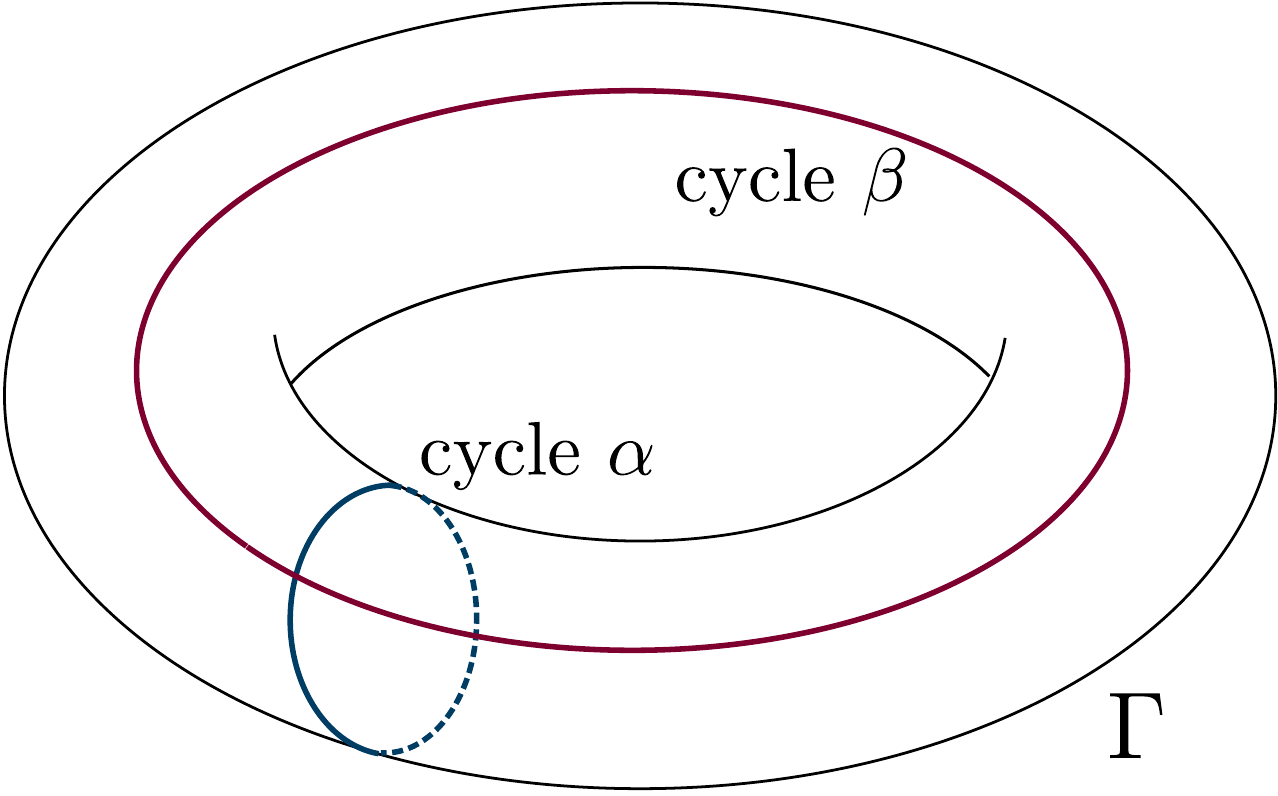}
 \caption{\small The elliptic curve $\Gamma$ with the imaginary cycle $\alpha$ (blue colour) and the real cycle $\beta$ (red colour).}
 \label{cycles}
\end{figure}

\newpage

\newpage
\section{Polynomial coefficients}
\label{apcoef}
\raggedright

\subsection{Coefficients of the Birkhoff normal form in \eqref{polB}}
\begin{align*}
a_{20} =& - R^2 (2 - 7 t + 6 t^2) -R (-3 + t) t  - t^2 \\
a_{11} =& 2 (-1 + R) \ra t \\
a_{02} =&- R^2 (2 - 5 t + 2 t^2)  - R t (-1 + 3 t) +t^2 \\
a_{30}=& -(-1 + R) t ( R^3 (-1 + 2 t)^3 + 
 R^2 t (3 - 8 t + 4 t^2)+ R t^2 (-3 + 2 t)+t^3)\\
a_{21}=& -R \ra (-1 + t) (R^2 (1 - 2 t)^2  + R t (-2 + 9 t)+ t^2) \\
a_{12}=& -3 (-1 + R) t ( R^3 (-1 + 2 t)^3 + 
 R^2 t (3 - 8 t + 4 t^2)+ R t^2 (-3 + 2 t)+t^3  )\\
a_{03}=&\ra (
 R^3 (1 - 3 t + 4 t^3)+R^2 (7 - 9 t) t^2 + R (-3 + t) t^2 + 2 t^3)\\ 
\end{align*}

\subsection{Coefficients of equation \eqref{polRI} in the proof of Lemma \ref{BNF}}
\begin{align*}
b_{00}=&-2 p_2 (R^2 (1 - 2 t)  + R (1 + p_2 (-1 + t) + t)- t)\\
b_{10}=&2 (-t + R (-1 + 2 t)) (2 (-1 + p_2) R^2 (1 - 2 t)^2 + 2 (-1 + p_2) t^2 \\&+
 R t (4 - 5 {p_2}^2 (-1 + t) + 2 p_2 (-5 + 3 t)))\\
b_{01}=&2 (-2 + p_2)^2 R (R^2 (1 - 2 t)^2 + R (-2 - 3 p_2 (-1 + t)) t + t^2)\\
b_{20}=& 2 (-7 {p_2}^5 R^2 (-1 + t)^2 t^2 (-3 t + R (-2 + 4 t)) + 
 8 R (5 R t^4 - t^5 + R^5 (-1 + 2 t)^5 \\&- 
 R^4 t (-1 + 2 t)^3 (1 + 4 t)
 - 2 R^2 t^3 (4 - 5 t + 6 t^2) + 
 2 R^3 t^2 (2 - 9 t + 8 t^2 + 4 t^3)) \\&+ 
 {p_2}^4 R (-1 + t) t (R (80 - 37 t) t^2 - 15 t^3 
 + R^2 t (32 - 13 t - 60 t^2) \\&+ 
 R^3 (-13 + 106 t - 240 t^2 + 160 t^3)) + 
 2 {p_2}^2 (R (25 - 13 t) t^4 - 3 t^5 \\&+ 3 R^6 (-1 + 2 t)^5 + 
 3 R^5 (-1 + 2 t)^3 (1 - 15 t + 10 t^2) - 
 3 R^2 t^3 (16 - 31 t + 20 t^2) \\&+ 
 2 R^3 t^2 (8 - 113 t + 151 t^2 - 46 t^3) + 
 R^4 t (13 - 82 t + 72 t^2 + 148 t^3 - 136 t^4)) \\&+ 
 {p_2}^3 (3 t^5 + 15 R t^4 (-3 + 2 t) + 
 6 R^2 t^3 (21 - 29 t + 13 t^2) - 
 R^5 (-1 + 2 t)^3 (3 - 38 t + 38 t^2) \\&+ 
 2 R^3 t^2 (3 + 116 t - 231 t^2 + 97 t^3) + 
 R^4 t (-33 + 280 t - 678 t^2 + 554 t^3 - 108 t^4)) \\&- 
 4 p_2 (R (5 - 3 t) t^4 - t^5 + 3 R^6 (-1 + 2 t)^5 + 
 R^2 t^3 (-8 + 35 t - 22 t^2) \\&+ 
 R^5 (-1 + 2 t)^3 (1 - 17 t + 2 t^2) + 
 2 R^3 t^2 (2 - 33 t + 32 t^2 - 11 t^3) \\&+ 
 R^4 (t + 14 t^2 - 112 t^3 + 198 t^4 - 76 t^5)))\\
b_{11}=& 2 (-2 + p_2)^3 (p_2 - 2 R) R^2 (2 R^3 (1 - 2 t)^4 - 
 2 (1 + p_2 (-1 + t)) t^3 \\&- 3 (2 + 3 p_2 (-1 + t)) t (R - 2 R t)^2 + 
 R t^2 (6 + 11  (-1 + t) + 10 {p_2}^2 (-1 + t)^2 - 8 t + 8 t^2))\\
b_{02}=& (-2 + {p_2})^3 ({p_2} - 2 R) R^2 (-(-2 + {p_2}) t^3 + 2 R^3 (-1 + 2 t)^3 \\&+ 
 R t^2 (-6 + {p_2} (9 - 7 t) + 5 {p_2}^2 (-1 + t) + 4 t) - 
 R^2 t (-1 + 2 t) (6 - 4 t + {p_2} (-8 + 9 t)))
\end{align*}

\subsection{Coefficients of the imaginary action in equation \eqref{polJ}:}
\begin{align*}
c_{20} =& -3 R^2 (-1 + t)^2 t^2 (t + R (-1 + 2 t))\\
c_{11} =& 2 R^2 (-1 + t) t (R^2 (1 - 2 t)^2 + R t - 2 t^2)\\
c_{02} =&R^2 t (-t^2 + R t (-1 + 3 t) + R^2 (2 - 5 t + 2 t^2))\\
c_{30} =&-5 R^3 (-1 + t)^3 t^3 (R (3 - 5 t) t^2 - 2 t^3 + 
 3 R^2 t^2 (-1 + 2 t) + R^3 (-1 + 2 t)^3)\\
c_{21} =& 3 R^3 (-1 + t)^2 t^2 (R^4 (1 - 2 t)^4 + 6 R^3 (1 - 2 t)^2 t - 
 4 R t^3 + 6 t^4 - 3 R^2 t^2 (3 - 14 t + 14 t^2))\\
c_{12} =&3 R^3 (-1 + t) t^2 (R (3 - 10 t) t^3 + 3 t^4 + 
 R^4 (-3 + t) (-1 + 2 t)^3 - 3 R^2 t^2 (4 - 12 t + 7 t^2) \\&+ 
 3 R^3 t (1 + t - 11 t^2 + 10 t^3))\\
c_{03} =&R^3 t^2 (2 R (3 - 5 t) t^3 + t^4 - R^4 (1 - 2 t)^2 (-6 + 4 t + t^2) + 
 3 R^2 t^2 (-3 + 4 t + t^2) \\&+ 2 R^3 t (-2 + 18 t - 33 t^2 + 15 t^3))
\end{align*}

\subsection{Coefficients of the root $\ze_1$ of $P$ as a function of $l$ and $h$ in \eqref{polRoot1}:}
\begin{align*}
d_{100} =& R (-R (1 - 2 t)^2 + t)\\
d_{010} =& R (R - t - 2 R t)\\
d_{001} =& -2 R t\\
d_{300} =& (-1 + t) t^2 (3 R^2 (1 - 2 t)^2 t - 3 R (1 - 2 t)^2 t^2 + t^3 + 
 R^3 (1 - 2 t)^2 (-1 - 4 t + 4 t^2))\\
d_{210} =& t (-(-2 + t) t^4 + R t^3 (-5 + 19 t - 18 t^2) - 
 R^4 (-1 + 2 t)^3 (-1 - 4 t + 4 t^2) \\&+ 
 3 R^2 t^2 (1 + t - 12 t^2 + 12 t^3) + 
 R^3 t (1 - 23 t + 78 t^2 - 84 t^3 + 24 t^4))\\
d_{201} =& -2 (-1 + t) t^2 (2 R^3 (1 - 2 t)^2 - 3 R^2 (1 - 2 t)^2 t + t^3)\\
d_{120} =& t (R^5 (1 - 2 t)^4 - t^4 - R^4 (1 - 2 t)^2 (-2 - 9 t + 8 t^2) + 
 R^2 t^2 (3 - 25 t + 36 t^2) \\&- R^3 t (5 - 27 t + 24 t^2 + 12 t^3) + 
 R (t^3 - 4 t^4))\\
d_{111} =& -2 t (2 R (1 - 2 t) t^3 + t^4 - R^4 (-1 + 2 t)^3 - 
 6 R^2 t^2 (1 - 3 t + 2 t^2) + 2 R^3 (t - 8 t^3 + 8 t^4))\\
d_{030} =& R t (t^3 + R^4 (-1 + 2 t)^3 + R t^2 (-3 + 7 t) + 
 R^2 t (3 - 3 t - 10 t^2) \\&- R^3 (1 + 3 t - 12 t^2 + 4 t^3))\\
d_{021} =& 2 R t (2 R^2 (3 - 4 t) t^2 + R (-3 + t) t^2 + 2 t^3 + 
 R^3 (1 - 3 t + 4 t^3))\\
\end{align*}

\subsection{Coefficients of the root $\ze_2$ of $P$ as a function of $l$ and $h$ in \eqref{polRoot2}:}
\begin{align*}
e_{100} =& d_{100}=R (-R (1 - 2 t)^2 + t)\\
e_{010} =& d_{010}=R (R - t - 2 R t)\\
e_{001} =& -d_{001}=2 R t\\
e_{300} =& -d_{300}=(1 - t) t^2 (3 R^2 (1 - 2 t)^2 t - 3 R (1 - 2 t)^2 t^2 + t^3 + 
 R^3 (1 - 2 t)^2 (-1 - 4 t + 4 t^2))\\
e_{210} =& -d_{210}=t ((-2 + t) t^4 + R^4 (-1 + 2 t)^3 (-1 - 4 t + 4 t^2) + 
 R t^3 (5 - 19 t + 18 t^2) \\&- 3 R^2 t^2 (1 + t - 12 t^2 + 12 t^3) - 
 R^3 t (1 - 23 t + 78 t^2 - 84 t^3 + 24 t^4))\\
e_{201} =& d_{201}=-2 (-1 + t) t^2 (2 R^3 (1 - 2 t)^2 - 3 R^2 (1 - 2 t)^2 t + t^3)\\
e_{120} =& -d_{120}=-t (R^5 (1 - 2 t)^4 - t^4 - R^4 (1 - 2 t)^2 (-2 - 9 t + 8 t^2) + 
 R^2 t^2 (3 - 25 t + 36 t^2) \\&- R^3 t (5 - 27 t + 24 t^2 + 12 t^3)+ 
 R (t^3 - 4 t^4))\\
e_{111} =& d_{111}=-2 t (2 R (1 - 2 t) t^3 + t^4 - R^4 (-1 + 2 t)^3 - 
 6 R^2 t^2 (1 - 3 t + 2 t^2) \\&+ 2 R^3 (t - 8 t^3 + 8 t^4))\\
e_{030} =& -d_{030}=R t (R (3 - 7 t) t^2 - t^3 - R^4 (-1 + 2 t)^3 + 
 R^2 t (-3 + 3 t + 10 t^2) \\&+ R^3 (1 + 3 t - 12 t^2 + 4 t^3))\\
e_{021} =& d_{021}=2 R t (2 R^2 (3 - 4 t) t^2 + R (-3 + t) t^2 + 2 t^3 + 
 R^3 (1 - 3 t + 4 t^3))\\
\end{align*}

\subsection{Coefficients of the root $\ze_3$ of $P$ as a function of $l$ and $h$ in \eqref{polRoot3}:}
\begin{align*}
f_{10} =& -2 (-1 + t) t\\
f_{01} =& -t + R (-1 + 2 t)\\
f_{20} =& (-1 + t) t (2 R^3 (1 - 2 t)^2 - 3 R^2 (1 - 2 t)^2 t + t^3)\\
f_{11} =& 2 R (1 - 2 t) t^3 + t^4 - R^4 (-1 + 2 t)^3 - 
 6 R^2 t^2 (1 - 3 t + 2 t^2) + 2 R^3 (t - 8 t^3 + 8 t^4)\\
f_{02} =& -R (2 R^2 (3 - 4 t) t^2 + R (-3 + t) t^2 + 2 t^3 + 
 R^3 (1 - 3 t + 4 t^3))\\
f_{30} =& -2 (-1 + t)^2 t (-R^5 (1 - 2 t)^4 + 5 R^3 (1 - 2 t)^2 t^2 \\&+ 
 8 R^4 (1 - 2 t)^2 (-1 + t) t^2 - 5 R^2 (1 - 2 t)^2 t^3 + t^5)\\
f_{21} =& -(-1 + t) (6 R (1 - 2 t) t^5 + 3 t^6 + R^6 (-1 + 2 t)^5 - 
 2 R^5 t (-1 + 2 t)^3 (-4 - 3 t + 6 t^2) \\&- 
 5 R^2 t^4 (5 - 16 t + 12 t^2) + 
 10 R^3 t^3 (1 + t - 12 t^2 + 12 t^3) \\&+ 
 R^4 t^2 (15 - 130 t + 348 t^2 - 344 t^3 + 96 t^4))\\
f_{12} =& -t^6 + 4 R t^5 (-2 + 3 t) - R^6 (1 - 2 t)^4 (-3 + 2 t^2) + 
 R^2 t^4 (15 - 28 t + 8 t^2) \\&+ 
 2 R^5 (1 - 2 t)^2 t (3 + 12 t - 29 t^2 + 12 t^3) - 
 10 R^3 t^3 (-1 + 10 t - 21 t^2 + 12 t^3) \\&+ 
 R^4 t^2 (-25 + 140 t - 222 t^2 + 72 t^3 + 40 t^4) \\
f_{03} =& -2 R (-t^5 + 3 R t^5 + R^5 (-1 + 2 t)^3 (-1 - t + t^2) + 
 R^2 t^3 (5 - 18 t + 11 t^2) \\&+ 
 R^3 t^2 (-5 + 10 t + 11 t^2 - 18 t^3) + 
 R^4 t^2 (10 - 35 t + 32 t^2 - 4 t^3))\\
\end{align*}

\subsection{Coefficients of the reduced period in equation \eqref{PolT}:}
\begin{align*}
h_{10} =&-(-1 + R) (R^4 (1 - 2 t)^4 + t^4 \\&+ 4 R^3 (1 - 2 t)^2 t (-1 + 3 t) + 
 4 R t^3 (-1 + 3 t) + 2 R^2 t^2 (3 - 16 t + 4 t^2))\\
h_{01} =& \ra (t^3 + R^4 (-1 + 2 t)^3 + R t^2 (-3 + 13 t) + 
 R^2 t (3 + 3 t - 28 t^2) \\&- R^3 (1 + 15 t - 42 t^2 + 16 t^3))\\
h_{20} =& -\ra (R^8 (1 - 2 t)^6 + t^6 - 2 R^7 (1 - 2 t)^4 t (3 + 2 t) - 
 2 R t^5 (3 + 2 t) \\&+ R^2 t^4 (15 + 4 t - 49 t^2) + 
 2 R^3 t^3 (-10 + 12 t + 39 t^2 + 56 t^3) \\&+ 
 R^4 t^2 (15 - 56 t + 113 t^2 - 212 t^3 - 170 t^4) \\&+ 
 R^6 (1 - 2 t)^2 (1 - 8 t + 95 t^2 - 164 t^3 + 46 t^4) \\&+ 
 2 R^5 t (-3 + 22 t - 134 t^2 + 490 t^3 - 698 t^4 + 420 t^5)) \\
h_{11} =&- 2 (-1 + R) (t^7 + R^8 (-1 + 2 t)^7 - R t^6 (7 + 3 t) + 
 R^2 t^5 (21 - 5 t - 56 t^2) \\&+ R^7 (-1 + 2 t)^5 (1 - 15 t + 4 t^2) +
 R^6 t (-1 + 2 t)^3 (-7 + 65 t - 322 t^2 + 224 t^3) \\&+ 
 R^3 t^4 (-35 + 71 t - 130 t^2 + 328 t^3) + 
 R^4 t^3 (35 - 175 t + 878 t^2 - 1308 t^3 + 200 t^4) \\&+ 
 3 R^5 t^2 (-7 + 65 t - 436 t^2 + 1112 t^3 - 1120 t^4 + 464 t^5))\\
h_{02} =&\ra (R^8 (1 - 2 t)^6 + t^6 + 2 R^7 (1 - 2 t)^4 t (-7 + 2 t) - 
 2 R t^5 (3 + 2 t) + R^2 t^4 (15 - 4 t - 57 t^2) \\&+ 
 2 R^3 t^3 (-10 + 28 t - 237 t^2 + 348 t^3) + 
 R^4 t^2 (15 - 104 t + 833 t^2 - 1012 t^3 - 138 t^4) \\&+ 
 R^6 (1 - 2 t)^2 (1 - 16 t - 329 t^2 + 188 t^3 + 110 t^4) \\&- 
 2 R^5 t (3 - 38 t + 14 t^2 + 886 t^3 - 1878 t^4 + 884 t^5))\\
h_{L10} =& (-1 + R) \ra t\\
h_{L01} =& t^2 + R (t - 3 t^2) + R^2 (-2 + 5 t - 2 t^2)\\
h_{L20} =& -t^6 + 2 R t^5 (2 + t) - R^6 (-1 + 2 t)^5 - 
 R^2 t^4 (5 + 2 t + 8 t^2) + R^3 t^4 (13 - 18 t + 25 t^2) \\&+ 
 R^5 (1 - 2 t)^2 t (-4 + 19 t - 26 t^2 + 17 t^3) + 
 R^4 t^2 (5 - 38 t + 68 t^2 - 34 t^3 - 16 t^4)\\
h_{L11} =& -2 (-1 + R) \ra t (-t^4 + 2 R t^3 (-1 + 3 t) + 
 R^4 (-1 + 2 t)^3 (-5 + 4 t) \\&+ 2 R^2 t^2 (6 - 11 t + 2 t^2) + 
 2 R^3 t (-7 + 25 t - 28 t^2 + 12 t^3))\\
h_{L02} =& -\ra^2 (4 R (2 - 3 t) t^3 + t^4 + 2 R^2 t^2 (-6 + 8 t + t^2) - 
 R^4 (1 - 2 t)^2 (-7 + 4 t + 2 t^2) 	\\&+ 
 R^3 t (-4 + 43 t - 82 t^2 + 39 t^3))\\
\end{align*}

\subsection{Coefficients of $\rg$ in equation \eqref{polRBljser} in the proof of Lemma \ref{lemT}:}
\begin{align*}
u_{20} =& -R (-3 + t) t - t^2 - R^2 (2 - 7 t + 6 t^2)\\
u_{11} =& 2 (-1 + R) \ra t \\
u_{02} =& t^2 + R (t - 3 t^2) + R^2 (-2 + 5 t - 2 t^2) \\
u_{60} =& (R (-3 + t) t + t^2 + R^2 (2 - 7 t + 6 t^2))^2\\
u_{51} =& 4 \ra t^2 (R^3 (1 - 2 t) + t + 3 R^2 t - R (1 + 2 t))\\
u_{42} =& -6 t^2 (R^3 (6 - 10 t) + R^4 (1 - 2 t)^2 + t^2 - 2 R t (1 + t) + 
 R^2 (1 + 5 t^2))\\
u_{33} =& 4 (-1 + R) \ra t (t^2 + R t (-7 + 5 t) + R^2 (6 - 19 t + 14 t^2))\\
u_{24} =& t^4 + 2 R t^3 (-7 + 5 t) + R^4 (1 - 2 t)^2 (-12 + 12 t + t^2) + 
 R^2 t^2 (13 - 12 t + 5 t^2) \\&- 2 R^3 t (-6 + 69 t - 127 t^2 + 66 t^3)\\
u_{15} =& 24 (-1 + R) R \ra (-1 + t) t (t + R (-1 + 2 t))\\
u_{06} =& 8 R (-1 + t) (R^2 (7 - 9 t) t^2 + R (-3 + t) t^2 + 2 t^3 + 
 R^3 (1 - 3 t + 4 t^3))\\
\end{align*}

\subsection{Coefficients of the root $\ze_1$ of $P$ as a function of $l$ and $j$ in \eqref{pol2Root1}:}
\begin{align*}
\al_{100} =&\ra\\ 
\al_{010} =& R - t - 2 R t\\
\al_{001} =& -2 \sqrt{R} t\\
\al_{300} =& 0\\
\al_{210} =& t (t^3 + R^4 (-1 + 2 t)^3 + R t^2 (-3 + 5 t) + 
 R^2 t (3 - 5 t - 4 t^2) + R^3 (-1 + t + 2 t^2))\\
\al_{120} =& (-1 + R) \ra t (t + R (-1 + 2 t))^2\\
\al_{030} =& 0\\
\al_{201} =& \sqrt{R} t^2 (R^2 (1 - 2 t) + R^3 (1 - 2 t)^2 + R (-2 + t) t + t^2)\\
\al_{111} =& 2 (-1 + R) \sqrt{R} \ra t^2 (t + R (-1 + 2 t))\\
\al_{021} =& \sqrt{R} t^2 (3 R^3 (1 - 2 t)^2 + 3 R (-2 + t) t + 3 t^2 + 
 R^2 (3 + 2 t - 8 t^2))\\
\end{align*}

\subsection{Coefficients of the root $\ze_2$ of $P$ as a function of $l$ and $j$ in \eqref{pol2Root2}:}
\begin{align*}
\be_{100} =& \al_{100} =\ra\\
\be_{010} =& \al_{010} =R - t - 2 R t\\
\be_{001} =& -\al_{001} = 2 \sqrt{R} t\\
\be_{300} =& -\al_{300} =0\\
\be_{210} =& -\al_{210} =t (R (3 - 5 t) t^2 - t^3 - R^4 (-1 + 2 t)^3 - R^3 (-1 + t + 2 t^2) \\&+ 
 R^2 t (-3 + 5 t + 4 t^2))\\
\be_{120} =& -\al_{120} =-(-1 + R) \ra t (t + R (-1 + 2 t))^2\\
\be_{030} =& -\al_{030} =0\\
\be_{201} =&\al_{201} = \sqrt{R} t^2 (R^2 (1 - 2 t) + R^3 (1 - 2 t)^2 + R (-2 + t) t + t^2)\\
\be_{111} =&\al_{111} = 2 (-1 + R) \sqrt{R} \ra t^2 (t + R (-1 + 2 t))\\
\be_{021} =& \al_{021}=\sqrt{R} t^2 (3 R^3 (1 - 2 t)^2 + 3 R (-2 + t) t + 3 t^2 + 
 R^2 (3 + 2 t - 8 t^2))\\
\end{align*}

\subsection{Coefficients of the root $\ze_3$ of $P$ as a function of $l$ and $j$ in \eqref{pol2Root3}:}
\begin{align*}
\ga_{10} =&\ra (-R + t)\\
\ga_{01} =&R^2 - 2 R^2 t + (-1 + 2 R) t^2 \\
\ga_{20} =&t (R^3 (1 - 2 t)^2 (-5 + t) t + t^4 - R t^3 (3 + t) + 
 R^4 (1 - 2 t)^2 (2 - 5 t + 4 t^2) \\&+ R^2 t^2 (5 - 7 t + 8 t^2))\\
\ga_{11} =&-2 (-1 + R) \ra t^2 (2 R (-2 + t) t + t^2 + R^2 (3 - 10 t + 8 t^2))\\
\ga_{02} =&t (R (7 - 3 t) t^3 - t^4 + R^2 t^2 (-9 + 11 t - 8 t^2) + 
 R^3 t (1 + 15 t - 40 t^2 + 28 t^3) \\&+ 
 R^4 (2 - 7 t + 20 t^3 - 16 t^4))\\
\ga_{30} =&(-1 + R) t (-t^5 + R^5 (1 - 2 t)^4 (-3 + 2 t) + R t^4 (1 + 2 t) + 
 2 R^2 t^3 (3 - 8 t + 4 t^2)\\& + 
 R^4 (1 - 2 t)^2 t (11 - 20 t + 12 t^2) + 
 2 R^3 t^2 (-7 + 26 t - 28 t^2 + 8 t^3))\\
\ga_{21} =&R \ra (-1 + t) (-t^4 - 3 R t^4 + R^4 (-3 + 2 t) (-1 + 2 t)^3 + 
 2 R^2 t^2 (3 - 22 t + 23 t^2)\\& + 
 R^3 t (-8 + 67 t - 142 t^2 + 80 t^3))\\
\ga_{12} =&(-1 + R) t (-3 t^5 + R^5 (1 - 2 t)^4 (-13 + 10 t) + 
 R t^4 (-1 + 10 t) + 2 R^2 t^3 (17 - 48 t + 28 t^2)\\& + 
 R^4 (1 - 2 t)^2 t (49 - 108 t + 68 t^2) + 
 2 R^3 t^2 (-33 + 138 t - 148 t^2 + 40 t^3))\\
\ga_{03} =&-\ra (-2 t^5 + R t^4 (-5 + 11 t) + R^5 (-1 + 2 t)^3 (-3 - 5 t + 6 t^2) +
 R^2 t^3 (24 - 75 t + 47 t^2)\\& + 
 2 R^3 t^2 (-11 + 21 t + 23 t^2 - 35 t^3) + 
 R^4 t (2 + 35 t - 141 t^2 + 142 t^3 - 32 t^4))\\
\end{align*}

\subsection{Coefficients of the elliptic modulus $k$ in equation \eqref{polKjl} in the proof of Lemma \ref{lemT}:}
\begin{align*}
\delta_{300} =&0 \\
\delta_{210} =& \ra (R (5 - 13 t) t^4 - t^5 - R^6 (-1 + 2 t)^5 + 
 R^5 (-1 + 2 t)^3 (-1 - 15 t + 8 t^2) \\&+ 
 R^2 t^3 (-10 + 27 t + 24 t^2) - 
 2 R^3 t^2 (-5 + 2 t + 17 t^2 + 18 t^3) \\&+ 
 R^4 t (-5 - 20 t + 106 t^2 - 144 t^3 + 104 t^4))\\
\delta_{120} =& (-1 + R) \ra^4 (R^2 (1 - 2 t)^2 + t^2 + 2 R t (-1 + 6 t))\\
\delta_{030} =&0 \\
\delta_{201} =&16 R^{5/2} \ra^3 (-1 + t) t^2 \\
\delta_{111} =&0 \\
\delta_{021} =&-16 R^{5/2} \ra^3 (-1 + t) t^2 \\
\end{align*}

\subsection{Coefficients of the rotation number in equation \eqref{PolW}:}
\begin{align*}
v_{10} =& t^4 (1 + 2 t) + R t^3 (-4 + 13 t - 21 t^2) - 
 R^4 (-1 + 2 t)^3 (1 - 9 t + 5 t^2) \\&+ 
 R^2 t^2 (6 - 47 t + 61 t^2 - 2 t^3) + 
 R^3 t (-4 + 47 t - 143 t^2 + 172 t^3 - 84 t^4)\\
v_{01} =& \ra (t^3 + R t^2 (-3 + 25 t - 25 t^2) + 
 R^3 (1 - 2 t)^2 (-1 - 9 t + 9 t^2) + R^2 t (3 - 20 t + 20 t^2))\\
v_{L10} =& \ra (2 t + R (-1 + 2 t))\\
v_{L01} =& R^2 (1 - 2 t)^2 + R t - 2 t^2\\
\end{align*}

\subsection{Coefficients of the partial derivative of the invariant $S$ with respect to $l$ in equation \eqref{poldsdl}:}
\begin{align*}
\mu_{10} =& (t + R (-1 + 2 t)) (t^3 + R^4 (-1 + 2 t)^3 + R t^2 (-3 + 5 t) + 
 3 R^2 t (1 - 7 t + 4 t^2) \\&+ R^3 (-1 + 17 t - 46 t^2 + 32 t^3))\\
\mu_{01} =& (-1 + R) \ra (t^3 + R^3 (-1 + 2 t)^3 + R t^2 (-3 + 14 t) + 
 R^2 t (3 - 20 t + 28 t^2))\\
\end{align*}

\subsection{Coefficients of the partial derivative of the invariant $S$ with respect to $j$ in equation \eqref{poldsdj}:}
\begin{align*}
\ka_{10} =& (-1 + R) \ra (R^2 (1 - 2 t)^2 + t^2 + 2 R t (-1 + 6 t))\\
\ka_{01} =& t^3 + R^4 (-1 + 2 t)^3 + R t^2 (-3 + 13 t) + 
 R^2 t (3 + 3 t - 28 t^2) - R^3 (1 + 15 t - 42 t^2 + 16 t^3)\\
\ka_{20} =& -R^8 (1 - 2 t)^6 + 6 R^7 (1 - 2 t)^4 t + 6 R t^5 - t^6 + 
 3 R^2 t^4 (-5 + 4 t + 3 t^2) \\&+ 
 2 R^3 t^3 (10 - 24 t + 9 t^2 - 26 t^3) - 
 3 R^4 t^2 (5 - 24 t + 59 t^2 - 76 t^3 + 2 t^4) \\&- 
 R^6 (1 - 2 t)^2 (1 - 8 t + 87 t^2 - 148 t^3 + 62 t^4) \\&+ 
 2 R^5 t (3 - 24 t + 134 t^2 - 438 t^3 + 570 t^4 - 276 t^5)\\
\ka_{11} =& 2 (-1 + R) \ra (R (-5 + t) t^4 + t^5 + R^6 (-1 + 2 t)^5 + 
 R^2 t^3 (10 - 15 t - 16 t^2) \\&+ 
 R^5 (-1 + 2 t)^3 (1 - 9 t + 4 t^2) + 
 2 R^3 t^2 (-5 + 20 t - 99 t^2 + 108 t^3) \\&+ 
 R^4 t (5 - 40 t + 274 t^2 - 596 t^3 + 336 t^4))\\
\ka_{02} =& R^8 (1 - 2 t)^6 - 6 R^7 (1 - 2 t)^4 t - 6 R t^5 + t^6 + 
 R^2 t^4 (15 - 12 t - 17 t^2) \\&+ 
 2 R^3 t^3 (-10 + 24 t - 177 t^2 + 210 t^3) + 
 3 R^4 t^2 (5 - 24 t + 195 t^2 - 236 t^3 + 10 t^4) \\&+ 
 R^6 (1 - 2 t)^2 (1 - 8 t - 201 t^2 + 76 t^3 + 118 t^4) \\&- 
 2 R^5 t (3 - 24 t + 22 t^2 + 570 t^3 - 1278 t^4 + 660 t^5) \\
\end{align*}

\bibliographystyle{alpha}
\bibliography{PhD}

\vspace{0.5cm}

\textbf{Jaume Alonso}\\
Department of Mathematics and Computer Science\\
University of Antwerp\\
Middelheimlaan 1\\
2020 Antwerp, Belgium\\
\textit{E-mail:} \texttt{jaume.alonsofernandez@uantwerpen.be}

\vspace{0.5cm}
\textbf{Holger R. Dullin}\\
School of Mathematics and Statistics\\
University of Sydney\\
Camperdown Campus\\
Sydney, NSW 2006, Australia \\
\textit{E-mail:} \texttt{holger.dullin@sydney.edu.au}

\vspace{0.5cm}
\textbf{Sonja Hohloch}\\
Department of Mathematics and Computer Science\\
University of Antwerp\\
Middelheimlaan 1\\
2020 Antwerp, Belgium\\
\textit{E-mail:} \texttt{sonja.hohloch@uantwerpen.be}

\end{document}